\documentclass[11pt,reqno,a4paper]{amsart}
\usepackage{amsmath,amsthm,verbatim,amscd,amssymb,setspace,enumitem,hyperref}
\usepackage{exscale,color}
\usepackage{enumitem}
\usepackage{mathrsfs}

\renewcommand{\epsilon}{\varepsilon}
\newcommand{\N}{\mathbb{N}}
\newcommand{\Z}{\mathbb{Z}}

\newcommand{\R}{\mathbb{R}}
\newcommand{\C}{\mathbb{C}}
\renewcommand{\P}{\mathbb{P}}

\renewcommand{\Re}{\operatorname{Re}}

\newcounter{mtheorem}
\newtheorem{mtheorem}[mtheorem]{Theorem}

\newtheorem{mcor}[mtheorem]{Corollary}

\renewcommand{\P}{\mathbb{P}}
\newcommand{{\vol}}{\rm vol}

\newcommand{\Ric}{\operatorname{Ric}}

\newcommand{\Rm}{\operatorname{Rm}}

\def\tr{\operatorname{tr}}

\newtheoremstyle{fancy}{}{}{\itshape}{}{\textbf\bgroup}{.\egroup}{ }{}
\newtheoremstyle{fancy2}{}{}{\rm}{}{\textbf\bgroup}{.\egroup}{ }{}

\theoremstyle{fancy}
\newtheorem{theorem}{Theorem}[section]
\newtheorem{lemma}[theorem]{Lemma}
\newtheorem{corollary}[theorem]{Corollary}

\newtheorem{prop}[theorem]{Proposition}

\theoremstyle{fancy2}
\newtheorem{definition}[theorem]{Definition}

\newtheorem{remark}[theorem]{Remark}

\newtheorem{claim}[theorem]{Claim}

\newtheorem{def+prop}[theorem]{Definition $\&$ Proposition}
\theoremstyle{plain}

\setlist{leftmargin=*}

\numberwithin{equation}{section}

\begin{document}
\title{Stability of asymptotically conical gradient Kähler-Ricci expanders}
\date{\today}
\author{Longteng Chen}
\address{Université Paris-Saclay, CNRS, Laboratoire de Mathématiques d'Orsay, 91405 Orsay, France }
\email{longteng.chen@universite-paris-saclay.fr}
\markboth{Longteng Chen}{Stability of Asymptotically conical gradient Kähler-Ricci expanders}
\begin{abstract}
In this work, we consider a perturbation of an asymptotically conical gradient expanding Kähler-Ricci soliton metric $g$ in the same K\"ahler class. We demonstrate that, under suitable assumptions, the normalized Kähler-Ricci flow starting from the initial perturbed metric exists for all time and converges uniformly to an asymptotically conical gradient expanding Kähler-Ricci soliton metric $g_\infty$. Moreover, if the perturbed initial metric is asymptotic to $g$ at spatial infinity, then the limiting metric coincides with the original soliton, that is, $g_\infty = g$.
\end{abstract}
\maketitle
\section{Introduction}
\subsection{Overview}
In this article, we study the stability of the Kähler–Ricci flow on non compact manifolds admitting a Kähler–Ricci soliton metric. As a natural generalization of Kähler–Einstein metrics, Kähler–Ricci solitons play a central role in complex geometry.
 An \emph{expanding K\"ahler-Ricci soliton} is a triple $(M,g,X)$, where $M$ is a complex manifold with complex structure $J$ and a complete K\"ahler metric $g$ and a complete real-holomorphic vector field $X$ satisfying the equation
\begin{equation}\label{soliton2}
\frac{1}{2}\mathcal{L}_{X}g=\Ric(g)+ g.
\end{equation}
If $X=\nabla^{g} f$ for some real-valued smooth function $f$ on $M$,
then we say that $(M,g,X)$ is \emph{gradient}. In this case, the soliton equation \eqref{soliton2}
reduces to\begin{equation}\label{soliton1}
\operatorname{Hess}_{g}(f)=\Ric(g)+ g,
\end{equation}
if $\omega$ is the K\"ahler form of $g$ and $\rho_{\omega}$ is the Ricci form of $\omega$, we rewrite \eqref{soliton1} as:
\begin{equation}
i\partial\bar{\partial}f=\rho_{\omega}+\omega.
\end{equation}

For a gradient K\"ahler-Ricci soliton $(M,g,X)$, the vector field $X$ is called the
\emph{soliton vector field}. Its completeness is guaranteed by the completeness of $g$
\cite{MR2497489}. The smooth real-valued function $f$ satisfying $X=\nabla^g f$ is called the \emph{soliton potential}. It is unique up to a constant.

A K\"ahler-Ricci expander plays an essential role in the analysis of K\"ahler-Ricci flow, because each K\"ahler-Ricci expander provides a \emph{self-similar} solution to K\"ahler-Ricci flow (see Proposition \ref{self-similar solution}).

The first nontrivial gradient expanding Kähler-Ricci soliton is the Gaussian soliton $(\C^n, g_{\mathrm{eucl}}, r\partial_r)$, where the self-similar solution $g(t)$ is static and converges as $t \to 0^+$ to the Kähler cone $(\C^n, g_{\mathrm{eucl}})$. Another example is due to Cao \cite{MR1449972}, who for any $\lambda>1$ constructed a complete soliton $g_\lambda$ with positive bisectional curvature on $\C^n$; its flow $g_\lambda(t)$ converges locally smoothly to the conical metric $\partial\bar{\partial}(|\cdot|^{\frac{2}{\lambda}})$ on $\C^n \setminus \{0\}$.

Feldman–Ilmanen–Knopf \cite{MR2058261} extended this picture by constructing, for $k>n\ge2$, complete gradient expanding solitons on the line bundle $\mathcal{O}(-k)$ over $\C\P^{n-1}$. Their associated flow converges as $t \to 0^+$ to the Kähler cone $(\C^n/\Z^k,\; i\partial\bar\partial\left(\frac{|\cdot|^{2p}}{p}\right)),$
with $p>0$, away from the zero section. In this case, the soliton carries positive scalar curvature.

We observe that the above examples describe an important class of Kähler-Ricci solitons modeled by Kähler cones. Inspired by these examples, we define the notion of an asymptotically conical gradient Kähler-Ricci expander, following the work of Conlon-Deruelle-Sun \cite{MR4711837}.

\begin{definition}[Asymptotically conical gradient K\"ahler-Ricci expander]\label{ACKR expander}
    An asymptotically conical gradient K\"ahler-Ricci expander is a triple $(M,g,X)$ being a complete expanding gradient K\"ahler-Ricci soliton of complex dimension $n\geq 2$ with complex structure $J$
whose curvature $\operatorname{Rm}(g)$ satisfies
\begin{equation*}
\sup_{x\in M}|(\nabla^{g})^{k}\operatorname{Rm}(g)|_{g}(x)d_{g}(p,\,x)^{2+k}<\infty\quad\textrm{for all $k\in\mathbb{N}_{0}$,}
\end{equation*}
where $d_{g}(p,\,\cdot)$ denotes the distance to a fixed point $p\in M$ with respect to $g$. 
\end{definition}

It can be verified that the above three examples (Gaussian soliton, Cao's solitons and FIK solitons) are asymptotically conical gradient K\"ahler-Ricci expanders. 

\subsection{Main results}\label{main results}
Let $(M,g,X)$ be an asymptotically conical gradient Kähler-Ricci expander. In this article, we perturb this metric $g$ by adding a term of the form $\partial\bar\partial\psi_0$ and investigate the resulting metric's evolution under the Kähler-Ricci flow. 
\subsubsection{Statement of main theorems}
Our main results are as follows:
\begin{mtheorem}[Convergence theorem]\label{convergence theorem}
Let $(M, g, X)$ be an asymptotically conical gradient Kähler-Ricci expander with asymptotic K\"ahler cone $(C,g_0)$. Let $f$ be the normalized soliton potential as in Lemma \ref{soliton indentities}. 

Assume that $\psi_0$ is a smooth real-valued function on $M$ satisfying \textbf{Condition II} (see Definition \ref{condition II}). Then there exists a smooth function $\psi\in C^{\infty}(M\times [0,\infty);\R)$ such that $g_\psi(\tau):=g+\partial\bar\partial\psi(\tau)$ is a complete immortal solution to the normalized Kähler-Ricci flow (see Definition \ref{NKRF eq}) for $\tau\ge0$, starting from the initial metric $g_{\psi_0}:=g+\partial\bar\partial\psi_0$. 

Moreover, 
\begin{enumerate}[label=\textnormal{(\alph{*})}, ref=(\alph{*})]
\item the initial metric $g_{\psi_0}$ is asymptotic conical with a unique asymptotic K\"ahler cone $(C,g_0')$;
\item as $\tau \to \infty$, the solution $g_\psi(\tau)$ converges smoothly and uniformly to an asymptotically conical gradient Kähler-Ricci expander $(M, g_\infty, X)$. More precisely,
\begin{enumerate}[label=\textnormal{(\roman{*})}, ref=(\roman{*})]
    \item for every $k \in \mathbb{N}_0$, there exists a constant $C_k > 0$ such that for all $(x, \tau) \in M \times [0, \infty)$,
\begin{equation}\label{convergence rate}
    \begin{split}
    &\left|\left(\nabla^{g_\psi(\tau)}\right)^k\Rm(g_\psi(\tau))\right|_{g_\psi(\tau)}(x)\le C_k f(x)^{-1-\frac{k}{2}};\\
        &\left|(\nabla^{g_\psi(\tau)})^k \left(g_\psi(\tau) - g_\infty\right)\right|_{g_\psi(\tau)}(x) \le C_k e^{-\tau} f(x)^{-1 - \frac{k}{2}}.
    \end{split}
\end{equation}
\item The expanding soliton $(M, g_\infty, X)$ is the unique (up to biholomorphisms) asymptotically conical gradient Kähler-Ricci expander with asymptotic cone $(C,g_0')$.
\end{enumerate}
 \item Furthermore, if the initial perturbation satisfies
\begin{equation}\label{quantitative asymptotic decay}
    \left|(\nabla^g)^k (g_{\psi_0} - g)\right|_g = o\left(f^{-\frac{k}{2}}\right)
\quad \text{for all } k \in \mathbb{N}_0,
\end{equation}
then the asymptotic cone of $(M,g_{\psi_0})$ is $(C,g_0)$ and the limiting metric satisfies $g_\infty = g$.
\end{enumerate}
\end{mtheorem}
\begin{remark}
The assumption \eqref{quantitative asymptotic decay} guarantees that the asymptotic cone of $(M,g_{\psi_0})$ is precisely $(C,g_0)$. If, instead of assuming \eqref{quantitative asymptotic decay}, we only require that the asymptotic cone of $(M,g_{\psi_0})$ is $(C,g_0)$, then by Theorem \ref{CDS theorem} we may conclude that $(M,g,X)$ and $(M,g_\infty,X)$ differ by a biholomorphism. However, in the absence of the quantitative condition \eqref{quantitative asymptotic decay}, we cannot assert that this biholomorphism is the identity.

  If the initial perturbation is not necessarily small, we cannot, in general, expect $g_\infty = g$. A simple counterexample is provided by the Gaussian soliton $(\mathbb{C}^n, g_{\mathrm{eucl}})$. Let $\psi_0 = \alpha \frac{r^2}{2}$ for some $\alpha > -1$, where $r$ denotes the radial function. Then $g_\psi(\tau) \equiv g_{\psi_0} = (1+\alpha) g_{\mathrm{eucl}}$ is a solution to the normalized Kähler–Ricci flow as in Theorem \ref{convergence theorem}. Hence, $g_\infty = (1+\alpha) g_{\mathrm{eucl}} \neq g_{\mathrm{eucl}}$ unless $\alpha = 0$.

 From a dynamical systems perspective, Theorem \ref{convergence theorem} implies that the moduli space of solutions to the normalized Kähler–Ricci flow has infinitely many fixed points. Each fixed point, appearing as an asymptotically conical gradient Kähler–Ricci expander, is uniquely characterized by its asymptotic Kähler cone, as established by Conlon, Deruelle, and Sun \cite{MR4711837}. This result drastically differs from Cao's work (Theorem \ref{cao}): on a closed complex manifold $M$ with $c_1(M)<0$, the classical normalized K\"ahler-Ricci flow converges to the unique K\"ahler-Einstein metric on $M$.
\end{remark}
\begin{remark}
    This spatial convergence rate \eqref{convergence rate} is optimal. We consider an initial K\"ahler potential $\psi_0$ such that $(\nabla^g)^k\psi_0=O(f^{-\frac{k}{2}})$ for all $k\in\N_0$, then by Theorem $\ref{convergence theorem}$, the solution to the normalized K\"ahler-Ricci flow $g_\psi(\tau)$ starting from $g_{\psi_0}$ will converge to $g_{\infty}=g$. Thus when $\tau=0$, we have for all $k\in\N_0$,
    \begin{equation*}
        \left|(\nabla^g)^k \left(g_{\psi_0} - g\right)\right| =O(f^{-1 - \frac{k}{2}}). 
    \end{equation*}
\end{remark}

In 2005, Chau and Schnürer \cite{MR2191907} established one of the first stability results for gradient expanding Kähler–Ricci solitons. They proved that Cao’s expander with positive bisectional curvature is stable under smooth perturbations of the soliton potential, provided the perturbed metric remains complete with bounded curvature and decays appropriately outside the unit disk in $\C^n$ ($n\ge2$). In particular, the Kähler–Ricci flow with such initial data, after rescaling, converges back to the soliton as $t\to\infty$. A natural class of admissible perturbations consists of compactly supported smooth functions.

In particular, as a corollary of our main results (Theorem \ref{convergence theorem}), the theorem of Chau and Schnürer extends to general asymptotically conical Kähler–Ricci expanders, namely:
\begin{mcor}
Let $(M,g,X)$ be an asymptotically conical gradient Kähler–Ricci expander with complex structure $J$. Suppose $\psi_0$ is a $JX$-invariant smooth real-valued function on $M$ with compact support, and assume that the perturbed tensor
\begin{equation*}
    g_{\psi_0} := g + \partial \bar{\partial} \psi_0
\end{equation*}
defines a Kähler metric. Then there exists a complete immortal solution $g(t)$ to the Kähler–Ricci flow with initial metric $g(0) = g_{\psi_0}$. Moreover, after a suitable rescaling, $g(t)$ converges uniformly to the original soliton metric $g$ as $t \to \infty$.

\end{mcor}
Estimating the maximal existence time of the Kähler–Ricci flow is often a crucial step in the analysis. By Theorem \ref{convergence theorem}, one obtains an immortal solution to the Kähler–Ricci flow (see Proposition \ref{correspondence}) whenever the initial Kähler potential satisfies Condition II (see Definition \ref{condition II}). The next theorem shows that even under weaker assumptions on the initial data, one can still ensure the existence of an immortal solution to the Kähler–Ricci flow:

\begin{mtheorem}[Long time existence theorem]\label{longtime existence theorem}
Let $(M, g, X)$ be an asymptotically conical gradient Kähler-Ricci expander. Let $f$ be the normalized soliton potential as in Lemma \ref{soliton indentities}. Suppose $\psi_0$ is a smooth real-valued function on $M$ satisfying \textbf{Condition I} (see Definition \ref{condition I}).

Then there exists a smooth function $\psi\in C^{\infty}(M\times [0,\infty);\R)$ such that $g_\psi(\tau):=g+\partial\bar\partial\psi(\tau)$ is a complete immortal solution to the normalized Kähler-Ricci flow (see Definition \ref{NKRF eq}) for $\tau\ge0$, starting from the initial metric $g_{\psi_0}:=g+\partial\bar\partial\psi_0$. Moreover, there exists a constant $C>1$ such that for all $(x,\tau)\in M\times [0,\infty)$
\begin{enumerate}
\item $|\frac{\partial}{\partial\tau}\psi(x,\tau)|+f(x)|\nabla^{g_\psi(x,\tau)}\frac{\partial}{\partial\tau}\psi(x,\tau)|_{g_\psi(x,\tau)}\le Ce^{-\tau}$;
    \item $\frac{1}{C}g(x)\le g_\psi(x,\tau)\le Cg(x)$;
    \item $f(x)|\nabla^gg_\psi(x,\tau)|^2_g\le C$;
    \item $|\Rm(g_\psi(x,\tau))|_{g_\psi(x,\tau)}\le C$.
\end{enumerate}
And for all $m\in\N^*$, for any $\alpha>0$ there exists a constant $C_m=C(\dim M,m,\psi_0,\alpha)>0$ such that
\begin{equation*}
    \sup_M\left|\left(\nabla^{g_\psi(\tau)}\right)^m\Rm(g_\psi(\tau))\right|_{g_\psi(\tau)}\le C_m, \quad \textrm{for all $\tau\ge \alpha$}.
\end{equation*}
\end{mtheorem}
\begin{remark}
    Once we have an immortal solution to the normalized Kähler-Ricci flow, the correspondence between the normalized Kähler-Ricci flow and the Kähler-Ricci flow (Proposition \ref{correspondence}) allows us to obtain an immortal solution to the Kähler-Ricci flow as well.
\end{remark}
For closed complex manifolds, Tian–Zhang’s criterion (Theorem \ref{Tian-Zhang}) provides a characterization of the maximal existence time. In 2011, Lott and Zhang \cite{MR2746389} characterized the first singularity time for a K\"ahler-Ricci flow solution on a general complex manifold. Building on \cite{MR2746389}, Chau, Li, and Tam established an estimate for the maximal existence time of the Kähler–Ricci flow in 2016.

Our Theorem \ref{longtime existence theorem} not only recovers Chau–Li–Tam’s Theorem \cite[Theorem 2.2]{MR3532144} in the context of asymptotically conical gradient Kähler–Ricci expanders, but also strengthens it by providing refined quantitative estimates for the Kähler potential on the spacetime $M \times [1,\infty)$. We need to notice that our proof is self-contained and does not rely on their results.

\subsubsection{Initial conditions}
In the context of Kähler geometry, the key idea is to impose appropriate growth conditions on the Kähler potential. To this end, we introduce two growth conditions.

\begin{definition}[Condition I]\label{condition I}
     Let $(M,g,X)$ be an asymptotically conical gradient K\"ahler-Ricci expander. Let $f$ denote the normalized soliton potential as in Lemma \ref{soliton indentities}. A real-valued smooth function $\psi_0$ defined on $M$ satisfies Condition I if
   \begin{enumerate}
        \item(Metric condition) the function $\psi_0$ is strictly $g-$psh, i.e. $g_{\psi_0}:=g+\partial\bar\partial\psi_0>0$. And there exists a constant $C_0>1$ such that $\frac{1}{C_0}g\le g_{\psi_0}\le C_0g$;
        \item(Killing condition) if $J$ denotes the complex structure of $M$, then $\mathcal{L}_{JX}g_{\psi_0}=0$;
        \item(Asymptotic condition) at spatial infinity $|\psi_0|=O(f)$, and $|(\nabla^g)^i (\frac{X}{2}\cdot\psi_0-\psi_0)|_g=O(f^{-\frac{i}{2}})$ for $i=0,1$, and
        \begin{equation*}
            |\Rm(g_{\psi_0})|_g+f^{\frac{1}{2}}|\nabla^gg_{\psi_0}|_g=O(1).
        \end{equation*}
    \end{enumerate}
\end{definition}
\begin{definition}[Condition II]\label{condition II}
   Let $(M,g,X)$ be an asymptotically conical gradient K\"ahler-Ricci expander. Let $f$ denote the normalized soliton potential as in Lemma \ref{soliton indentities}. A real-valued smooth function $\psi_0$ defined on $M$ satisfies Condition II if
   \begin{enumerate}
        \item(Metric condition) the function $\psi_0$ is strictly $g-$psh, i.e. $g_{\psi_0}:=g+\partial\bar\partial\psi_0>0$;
        \item(Killing condition) if $J$ denotes the complex structure of $M$, then $\mathcal{L}_{JX}g_{\psi_0}=0$;
        \item(Asymptotic condition) at spatial infinity $|\psi_0|=O(f)$, and $|(\nabla^g)^i (\frac{X}{2}\cdot\psi_0-\psi_0)|_g=O(f^{-\frac{i}{2}})$ for $i=0,1$, and
        \begin{equation*}
            \begin{split}
                &|(\nabla^g)^k(g_{\psi_0}-g)|_g=O(f^{-\frac{k}{2}}),\\
                &|(\nabla^g)^k(\mathcal{L}_{\frac{X}{2}}g_{\psi_0}-g_{\psi_0})|_g=O(f^{-1-\frac{k}{2}}), \quad \textrm{for all $k\in\N_0$}.
            \end{split}
        \end{equation*}
    \end{enumerate}
\end{definition}
\begin{remark}
By Lemma \ref{proper function}, the function $f$ is asymptotic to $\tfrac{1}{2} d_g(p,\cdot)^2$ at infinity, where $d_g(p,\cdot)$ denotes the distance function with respect to $g$ from a fixed point $p \in M$. Hence, we use $f$ to measure the asymptotic behavior in place of the distance function $d_g(p,\cdot)$.

In Section \ref{Preliminaries}, we give geometric interpretations for Kähler potentials satisfying Condition I or II. We show that certain asymptotic properties follow directly as corollaries of the metric condition and the Killing condition (see Remark \ref{refined initial condition of condition I}). We also show that if $\psi_0$ satisfies Condition II, then there exists a constant $C_0>1$ such that $\frac{1}{C_0}g\le g_{\psi_0}\le C_0g$. (see Remark \ref{refined geometric condition II})

Moreover, the condition
\begin{equation*}
    \bigl|(\nabla^g)^k(g_{\psi_0}-g)\bigr|_g = O\left(f^{-\frac{k}{2}}\right), \quad \textrm{for all $k \in \mathbb{N}_0$},
\end{equation*}
is automatically satisfied, since
\begin{equation*}
    \bigl|(\nabla^g)^k\bigl(\mathcal{L}_{\frac{X}{2}} g_{\psi_0} - g_{\psi_0}\bigr)\bigr|_g = O\left(f^{-1-\frac{k}{2}}\right), \quad \textrm{for all $k \in \mathbb{N}_0$}.
\end{equation*}

Furthermore, we prove that if $\psi_0$ satisfies Condition II, then the corresponding initial metric is asymptotically conical with a unique asymptotic K\"ahler cone.

\end{remark}
\begin{remark}
   Since $(M,g,X)$ is an expanding gradient Kähler-Ricci soliton, the Reeb vector field $JX$ naturally constitutes a Killing vector field for $g$ (see \cite[Lemma 4.1]{2025arXiv250500167C}). To initiate the normalized Kähler-Ricci flow, we consequently require the initial metric $g_{\psi_0}$ to preserve this Killing symmetry - specifically, that $JX$ remains a Killing vector field for $g_{\psi_0}$. Given that $JX$ is a real-holomorphic vector field, this condition is equivalent to the vanishing of the $(1,1)$-form $\partial\bar{\partial}(JX\cdot\psi_0)$. In Section \ref{reduction}, we show that, without altering the initial metric $g_{\psi_0}$, one may always assume that the Kähler potential is $JX$-invariant, i.e. $JX \cdot \psi_0 = 0$. This invariance, together with the associated Killing vector field, allows us to reduce the Kähler–Ricci flow equation to a complex Monge–Ampère equation on the non compact complex manifold $M$.
\end{remark}

We can easily see that Condition II is stronger than Condition I. In fact, the class of functions satisfying Condition I or II is considerably larger than it first appears.
Let $f$ be the normalized Hamiltonian potential of $X$ with respect to $g$, as in Lemma \ref{soliton indentities}. Let $\psi_0$ be a $g-$psh function, and we assume that $JX\cdot\psi_0=0$ and  
\begin{equation*}  
    (\nabla^g)^k\psi_0 = O(f^{-\frac{k}{2}}) \quad \text{for all } k \in \mathbb{N}_0,  
\end{equation*}  
then $\psi_0$ automatically satisfies Condition II. This condition implies that our perturbation is \emph{small}. Under this assumption, the initial metric $g_{\psi_0}$ is asymptotic to $g$ at spatial infinity.

Another example is given by $\psi_0 = \alpha f$ for some small constant $\alpha$. Thus we have $JX\cdot\psi_0=\alpha JX\cdot f=0$. By the soliton equation \eqref{soliton2}, we have  
\begin{equation*}
    g_{\psi_0} = g + \alpha \partial\bar{\partial} f = (1 + \alpha)g + \alpha \mathrm{Ric}(g).
\end{equation*}
For sufficiently small $\alpha$, the initial metric $g_{\psi_0}$ is bi-Lipschitz to $g$.  

Moreover, Lemma \ref{soliton indentities} implies  
\begin{equation*}
    \frac{X}{2} \cdot \psi_0 - \psi_0 = \alpha(|\partial f|_g^2 - f) = -\alpha(R_\omega + n) = O(1). 
\end{equation*}
For higher-order terms, we verify that  
\begin{equation*}  
    \begin{split}  
        (\nabla^g)^k \psi_0 &= O(f^{1-\frac{k}{2}}); \\  
        (\nabla^g)^k \left(\frac{X}{2} \cdot \psi_0 - \psi_0\right) &= O(f^{-\frac{k}{2}}), \quad \text{for all } k \in \mathbb{N}_0.  
    \end{split}  
\end{equation*}  
Thus, in this case, \(\psi_0 = \alpha f\) satisfies Condition II.  

If $\psi_0$ is a Kähler potential satisfying Condition I or II. Unlike in previous work, where the perturbation of the canonical metric was assumed to be small, the difference $g_{\psi_0} - g$ is not necessarily small in our setting. Nevertheless, we are still able to study the stability of the Kähler–Ricci flow on asymptotically conical gradient Kähler–Ricci expanders.
\subsection{Stability of K\"ahler-Ricci flow on K\"ahler-Einstein manifolds}
A central problem in Kähler geometry is the construction of \emph{canonical metrics} on Kähler manifolds via geometric flows. Among these, the Kähler–Ricci flow has proven to be an especially powerful method. A smooth family of K\"ahler metrics $g(t)_{t\in (0,T)}$ on a complex manifold $M$ is a solution to K\"ahler-Ricci flow if
\begin{equation*}
    \frac{\partial}{\partial t}g(t)=-\Ric(g(t)), \quad t\in (0,T).
\end{equation*}
If the maximal existence time is finite, meaning that the Kähler–Ricci flow develops singularities in finite time, an important aspect of its analysis is to study the asymptotic behavior of the flow near these singularities. On the other hand, if the maximal existence time is infinite, it becomes crucial to investigate the limiting metric (if it exists) as time tends to infinity. In 2006, Tian and Zhang \cite{MR2243679} established an algebraic criterion for determining the maximal existence time of the Kähler–Ricci flow on closed complex manifolds.
\begin{theorem}[Tian-Zhang \cite{MR2243679}]\label{Tian-Zhang}
    Let $M$ be a closed complex manifold endowed with K\"ahler metric $g_0$. Let $c_1(M)$ denote the first Chern class of $M$ and let $[\omega_0]$ be the K\"ahler class of $g_0$. If we define
    \begin{equation*}
        T:=\sup\{t>0\ |\ [\omega_0]-tc_1(M)>0\},
    \end{equation*}
    then there exists a unique maximal solution $g(t)$ to K\"ahler-Ricci flow with $g(0)=g_0$ for $t\in [0,T)$.
\end{theorem}
A direct application arises on manifolds with negative first Chern class: in 1985, Cao \cite{MR799272} used the Kähler–Ricci flow to construct Kähler–Einstein metrics in this setting.
\begin{theorem}[Cao \cite{MR799272}]\label{cao}
Let $M$ be a closed complex manifold with $c_1(M)<0$. For any Kähler metric $g_0$ whose Kähler form lies in $-c_1(M)$, there exists a unique solution $g(t)$ to the Kähler–Ricci flow with $g(0)=g_0$ for all $t \geq 0$. Moreover, as $t \to \infty$, the rescaled metrics $\tfrac{1}{t} g(t)$ converge smoothly to the unique Kähler–Einstein metric $g_{\textrm{KE}} \in -c_1(M)$.
\end{theorem}

On non compact complex manifolds, the analysis becomes considerably more difficult, unless one has detailed knowledge of the underlying geometric structure. In 2005, Chau \cite{MR2112629} studied the stability of the Kähler–Ricci flow on a non compact complex manifold $M$ admitting a complete negatively curved Kähler–Einstein metric $g_{\textrm{KE}}$ (i.e. $\Ric(g_{\textrm{KE}}) + g_{\textrm{KE}} = 0$) with bounded curvature. He proved that if $g_{\textrm{KE}}$ is perturbed by a term of the form $\partial\bar\partial u$, where $u$ is a smooth bounded function, then there exists an immortal solution $g(t)$, $t \in [0,\infty)$, to the Kähler–Ricci flow with initial condition $g(0) = g_{\textrm{KE}} + \partial\bar\partial u$. Moreover, as $t \to \infty$, the rescaled metrics $\tfrac{1}{t} g(t)$ converge smoothly to $g_{\textrm{KE}}$ on every compact subset of $M$. On $\C^n$, Chau, Li and Tam \cite{MR3646777} proved the stability of K\"ahler-Ricci flow without supposing that the initial metric has bounded curvature. Their main result is the following: for any complete $U(n)-$invariant initial metric with non-negative bisectional curvature on $\C^n$, the K\"ahler-Ricci flow admits a $U(n)-$invariant long-time solution and converges, after a suitable rescaling at the origin, to the Euclidean metric. Moreover, the curvature becomes immediately bounded and the bisectional curvature stays non-negative.

\subsection{Strategy of proof}
Our approach builds on Shi’s fundamental work \cite{MR1001277} concerning the Ricci flow on non compact manifolds. The key technical tool is the Killing vector field $JX$, which allows us to reduce the (normalized) Kähler-Ricci flow to a complex (normalized) Monge-Ampère equation. To establish both theorems, we develop a parabolic maximum principle tailored to asymptotically conical gradient Kähler-Ricci expanders, enabling us to control essential geometric quantities via their initial values.

For long-time existence (Theorem \ref{longtime existence theorem}),
by applying the maximum principle on non-compact manifolds, we derive uniform bounds for the Riemann curvature tensor in terms of the initial data. Under Condition I, we prove that the curvature is bounded along the (normalized) Kähler-Ricci flow. This crucial curvature control, combined with Shi’s existence criterion, yields the desired long-time existence of solutions.

To prove the convergence Theorem (Theorem \ref{convergence theorem}), we systematically study the evolution equations for all higher-order derivatives $(\nabla^{g_\psi})^k\frac{\partial}{\partial\tau}g_\psi$ ($k\in\mathbb{N}_0$). Under Condition II, careful integration of these evolution equations ultimately establishes the convergence result.
\subsection{Outline of proof}
In Section \ref{Preliminaries}, we establish some fundamental properties of asymptotically conical gradient Kähler-Ricci expanders that will be essential for our analysis. In the remainder of Section \ref{Preliminaries}, we provide a geometric interpretation of Kähler potentials satisfying Conditions I and II. 

Section \ref{Shi's solution section} presents the construction of solutions to the (normalized) Kähler-Ricci flow using Shi's existence theorem. Here, we reduce the tensorial flow equation to an equivalent (normalized) complex Monge-Ampère equation, which serves as the foundation for our subsequent arguments.

The proof of Theorem \ref{longtime existence theorem} is developed in Section \ref{proof of 1}, where we implement a maximum principle adapted to the non-compact manifold $M$. Finally, in Section \ref{proof of 2}, we strengthen our initial assumptions to establish the convergence result in Theorem \ref{convergence theorem}.
\subsection{Acknowledgment}I would like to express my heartfelt gratitude to my PhD supervisor, Alix Deruelle, for his patient guidance, continuous encouragement, and invaluable assistance with the key steps of the proof. And I would like to thank Professor Frank Pacard for his useful suggestions to the first version of this article.
\section{Preliminaries}\label{Preliminaries}
\subsection{Self-similar solutions and normalized K\"ahler-Ricci flow}
A K\"ahler-Ricci expander plays an essential role in the analysis of K\"ahler-Ricci flow, because each K\"ahler-Ricci expander provides a \emph{self-similar} solution to K\"ahler-Ricci flow.
\begin{prop}[Self-similar solution]\label{self-similar solution}
Given a complete expanding Kähler–Ricci soliton $(M, g, X)$, one can construct a solution $g(t)$ to the Kähler–Ricci flow for all $t>0$ such that $g(1) = g$.

Let $\Phi_X^\cdot \colon M \to M$ denote the flow generated by the vector field $X$, and define $\Phi_t := \Phi_X^{-\frac{1}{2} \log t}$ for all $t > 0$. Then, it can be verified that
\begin{equation}
    g(t) := t \cdot \Phi_t^* g,\quad t>0,
\end{equation}
defines a solution to the Kähler–Ricci flow on $M \times (0, \infty)$. We refer to $g(t)_{t>0}$ as the self-similar solution to the Kähler–Ricci flow associated with $(M, g, X)$.
\end{prop}
As $t\rightarrow 0^+$, the self-similar solution $g(t)$ may develop a singularity. It is widely believed that singularities of a Kähler–Ricci flow can be modeled by self-similar solutions. Therefore, understanding the singularities of self-similar solutions is of fundamental importance.

To investigate the long-time behavior of solutions to the Kähler–Ricci flow, it is natural to consider the rescaled flow, namely the normalized Kähler–Ricci flow.
\begin{definition}[Normalized K\"ahler-Ricci flow]
Let $(M,g,X)$ be an asymptotically conical gradient Kähler-Ricci expander.
    A smooth family $g(\tau)_{\tau\in [0,T)}$ of K\"ahler metrics on $M$ is a solution to \emph{normalized K\"ahler-Ricci flow} if
    \begin{equation}\label{NKRF eq}
        \frac{\partial}{\partial\tau}g(\tau)=\mathcal{L}_{\frac{X}{2}}g(\tau)-\Ric(g(\tau))-g(\tau).
    \end{equation}
\end{definition}
Notice that the static flow $g(\tau)=g$ for all $\tau\in \R$ is a solution to normalized K\"ahler-Ricci flow since $(M,g,X)$ is a K\"ahler-Ricci expander. Moreover, we have a 1-1 correspondence between solutions to normalized K\"ahler-Ricci flow and solutions to K\"ahler-Ricci flow.
\begin{prop}\label{correspondence}
    Let $g(\tau)_{\tau\in [0,T)}$ be a solution to the normalized K\"ahler-Ricci flow, then 
    \begin{equation*}
        \bar g(t):=t\Phi_t^*g(\log t),
    \end{equation*}
    for $t\in [1,e^T)$ is a solution to K\"ahler-Ricci flow.
\end{prop}
\subsection{Asymptotic conical gradient K\"ahler-Ricci expanders and their structure theorem}
In this section, we summarize the key properties of asymptotically conical gradient Kähler-Ricci expanders that will be used throughout the remainder of the article. Let $(M,g,X)$ be an asymptotically conical gradient Kähler-Ricci expander as in Definition \ref{ACKR expander}, let $\omega$ be the K\"ahler form of $g$ and let $f\in C^\infty(M;\R)$ be a Hamiltonian potential of $X$ such that $\nabla^g f=X$. In their paper \cite{MR4711837}, they established a general structure theorem for an asymptotically conical gradient K\"ahler-Ricci expander.

\begin{theorem}[Structure theorem \cite{MR4711837}]\label{CDS theorem}Let $(M,g,X)$ be an asymptotically conical gradient K\"ahler-Ricci expander. Then there exists a Kähler resolution $\pi: M \to C$ of a Kähler cone $(C, g_0)$ with exceptional set $E$, such that
\begin{enumerate}
    \item the map $\pi$ is a biholomorphism between $M \setminus E$ and $C \setminus \{o\}$, where $o$ denotes the apex of the cone $C$;
    \item let $g(t)_{t>0}$ denote the self-similar solution to K\"ahler-Ricci flow associated to $g$, then when $t$ goes to 0, $\pi_*g(t)$ converges smoothly locally to $g_0$ on $C\setminus\{o\}$;
    \item let $r$ denote the radial function on the K\"ahler cone $(C, g_0)$, we have $d\pi(X)=r\partial_r$;
    \item the asymptotic conical gradient K\"ahler-Ricci expander $(M,g,X)$ is the unique (up to biholomorphisms) asymptotic conical gradient K\"ahler-Ricci expander with asymptotic cone $(C,g_0)$.
\end{enumerate}
\end{theorem}

A fundamental feature of gradient Kähler–Ricci expanders is encapsulated in the celebrated soliton identities, which play a central role in the analysis of their geometry and dynamics.
\begin{lemma}[Soliton identities]\label{soliton indentities}
    \begin{equation*}
        \begin{split}
            &\Delta_\omega f=n+R_\omega,\\
            &\nabla^g R_\omega+\Ric(g)(X)=0,\\
            &|\partial f|_g^2+R_\omega=f+\textrm{constant}.
        \end{split}
    \end{equation*}
    Here $n=\dim_\C M$, $\Delta_\omega$ is the K\"ahler Laplacian, $R_\omega=\frac{1}{2}R_g$ is the K\"ahler scalar curvature, $R_g$ is the scalar curvature of $g$.
\end{lemma}
\begin{proof}
    The proof is standard (see \cite[Section 2 of Chapter 1]{MR2302600}) since $\Ric(g)+g=\partial\bar\partial f$.
\end{proof}
\textbf{From now on, we normalize $f$ such that $|\partial f|_g^2+R_\omega+n=f$}.

Like the potential of Gauss' soliton in the Euclidean case, our soliton potential $f$ is comparable to $\frac{d_g(p,\cdot)^2}{2}$ for any fixed point $p\in M$ at spatial infinity.
\begin{lemma}\label{proper function}
    The normalized potential $f$ is a proper function, and for a fixed point $p\in M$ there exist constants $C,c_1,c_2>0$ such that 
    \begin{equation*}
        \frac{(d_g(p,x)-c_1)^2}{2}-C\le f(x)\le  \frac{(d_g(p,x)+c_2)^2}{2},\quad \forall x\in M.
    \end{equation*}
\end{lemma}
\begin{proof}
    See \cite[Proposition 2.19]{MR4711837}.
\end{proof}
The normalization that we exploit ensures that our soliton potential $f$ is strictly positive.
\begin{prop}\label{lower bound of f}
    There exists an $\varepsilon>0$ such that 
    \begin{equation*}
        \inf_M f\ge\varepsilon>0.
    \end{equation*}
\end{prop}
\begin{proof}
 This result follows from the fact that $R_\omega + n \ge \varepsilon > 0$ for some positive constant $\varepsilon$, which holds for any asymptotically conical gradient Kähler-Ricci expander $(M, g, X)$. For the proof of this fact, see \cite[Lemma 2.2]{2025arXiv250500167C}.
\end{proof}
Our Kähler–Ricci expander $(M,g,X)$ is of gradient type. Such expanders enjoy additional symmetries since they admit a Killing vector field $JX$. This fact is closely related to the following lemma:
\begin{lemma}\label{Killing vector}
Let $(M,g,J)$ be a Kähler manifold with complex structure $J$ and Kähler metric $g$. Then a real holomorphic vector field $JX$ is Killing if $X$ is the gradient of a smooth function with respect to $g$.
\end{lemma}
\begin{proof}
    We suppose that $X=\nabla^g f$ for some smooth function $f$. Let $Y,Z$ be two vector fields, now we compute,
    \begin{equation*}
        \begin{split}
            (\mathcal{L}_{JX}g)(Y,Z)&=g(\nabla^g_{Y}JX,Z)+g(Y,\nabla^g_{Z}JX)\\
            &=g(J\nabla^g_YX,Z)+g(Y,J\nabla^g_ZX)\\
            &=-g(\nabla^g_YX,JZ)-g(JY,\nabla^g_ZX)\\
            &=-(\nabla^g)^2f(Y,JZ)-(\nabla^g)^2f(JY,Z).
        \end{split}
    \end{equation*}
Since $X$ is real-holomorphic, the Hessian $(\nabla^g)^2 f$ is Hermitian. Consequently, the last term vanishes, and we conclude that $JX$ is a Killing vector field with respect to $g$.
\end{proof}
\subsection{Geometric interpretation of Condition I and II}
In this section, we provide a geometric interpretations of Kähler potentials satisfying Condition I or II. First, we show that if $\psi_0$ satisfies Condition I, then certain properties in the asymptotic condition follow as corollaries of the metric condition and the Killing condition.
\begin{prop}\label{metric+killing to asymptotic}
Let $(M,g,X)$ be an asymptotically conical gradient Kähler–Ricci expander with normalized soliton potential $f$, and let $\psi_0$ be a $JX$-invariant Kähler potential, i.e. $JX \cdot \psi_0 = 0$. Suppose there exists a constant $C_0 > 0$ such that
\begin{equation*}
    \frac{1}{C_0} g \leq\; g_{\psi_0}\leq C_0 g,
\end{equation*}
where $g_{\psi_0}=g+\partial\bar\partial\psi_0 $.
Then there exists a constant $C > 0$ such that on $M$, $ |\psi_0| \leq C f$ holds.
\end{prop}
\begin{proof}
Since $X$ is a real holomorphic vector field, let $U$ denote the holomorphic part of $X$, so that $X=U+\overline U$. Notice that $JX\cdot\psi_0=0$, we obtain that 
\begin{equation*}
    X\cdot\psi_0=2U\cdot\psi_0=2\overline U\cdot\psi_0.
\end{equation*} 
Locally, the vector field $X$ can be written as
\begin{equation*}
    X = U + \overline{U} = X^i \partial_i + X^{\bar{j}} \partial_{\bar{j}},
\end{equation*}
with $X^i$ holomorphic and $X^{\bar{j}}$ anti-holomorphic. Then, for any vector field $Y$ of the form
\begin{equation*}
    Y = Y^i \partial_i + Y^{\bar{j}} \partial_{\bar{j}},
\end{equation*}
we have
\begin{equation*}
 \begin{split}
        \partial\bar\partial\psi_0(X,Y)&=X^iY^{\bar j}\partial_i\bar\partial_{\bar j}\psi_0+Y^iX^{\bar j}\partial_i\bar\partial_{\bar j}\psi_0\\
        &=Y^{\bar j}\partial_{\bar j}\left(X^i\partial_i\psi_0\right)+Y^{i}\partial_i\left(X^{\bar j}\partial_{\bar j}\psi_0\right)\\
        &=Y^{\bar j}\partial_{\bar j}\left(\frac{1}{2}X\cdot\psi_0\right)+Y^i\partial_i\left(\frac{1}{2}X\cdot\psi_0\right)\\
        &=Y\cdot\left(\frac{1}{2}X\cdot\psi_0\right)
 \end{split}
\end{equation*}
Therefore, we conclude that
\begin{equation*}
   g_{\psi_0}(X,\cdot)=g(X,\cdot)+\frac{1}{2}d(X\cdot\psi_0).
\end{equation*}
    Since $\frac{1}{C_0}g\le g_{\psi_0}\le C_0 g$, then there exists a constant $C_1>0$ such that
    \begin{equation*}
        |X\cdot(X\cdot\psi_0)|=|<X,d(X\cdot\psi_0)>|\le C_1|X|_g^2=C_1X\cdot   f.
    \end{equation*}
   From Lemma \ref{proper function}, we know that $f$ is a proper function. Let us define \begin{equation*}
       C_2=\max_{f(x)\le1}X\cdot\psi_0(x)-C_1f(x).
   \end{equation*} Since $X\cdot(X\cdot\psi_0-C_1f)\le 0$, it follows that on $M$
   \begin{equation*}
       X\cdot\psi_0-C_1f\le C_2.
   \end{equation*}
   By Proposition \ref{lower bound of f}, there exists an $\varepsilon>0$ such that $f\ge \varepsilon$ on $M$. Hence 
   \begin{equation*}
       X\cdot\psi_0\le \left(\frac{C_2}{\varepsilon}+C_1\right)f:=C_3f.
   \end{equation*}
   Recall the soliton identity, $f=n+R_\omega+|\partial f|_g^2$, we define $C_4=n+\sup_M|R_\omega|$, hence $f\le C_4+|\partial f|_g^2$. It follows that
   \begin{equation}\label{derive fonction}
       X\cdot\psi_0\le \frac{C_3}{2}X\cdot f+C_3C_4.
   \end{equation}
   Let $\Phi^{s}$ denote the flow of $X$ for all $s\in \R$, then \eqref{derive fonction} implies that for all $x\in M$, $s<0$, we have
   \begin{equation}\label{integration simple}
       \left(\psi_0-\frac{C_3}{2}f\right)(x)-\left(\psi_0-\frac{C_3}{2}f\right)(\Phi^s(x))\le -C_3C_4s.
   \end{equation}
   We define $C_5=\max_{f(y)\le 1}\left(\psi_0-\frac{C_3}{2}f\right)(y)$. Now we fix $x\in M$ such that $f(x)\ge 1$ and let $s_0\le 0$ such that $f(\Phi^s(x))=1$. It follows that, by \eqref{integration simple},
   \begin{equation*}
       \left(\psi_0-\frac{C_3}{2}f\right)(x)\le -C_3C_4s_0+C_5.
   \end{equation*}
   Now we estimate $s_0$. Let us define $h(s)=f(\Phi^s(x))$ for $s\le0$, we compute
   \begin{equation*}
       \frac{d}{ds}h(s)=X\cdot f(\Phi^s(x))\le 2f(\Phi^s(x)).
   \end{equation*}
   By integration, we have for $s_0\le 0$ such that $h(s_0)=f(\Phi^{s_0}(x))=1$
   \begin{equation*}
       e^{-2s_0}=e^{-2s_0}h(s_0)\le h(0)=f(x).
   \end{equation*}
   Thus $-s_0\le \frac{1}{2}\log f(x)\le f(x)$. Therefore, we have that for $x\in M$ such that $f(x)\ge 1$,
   \begin{equation*}
        \left(\psi_0-\frac{C_3}{2}f\right)(x)\le C_3C_4f(x)+C_5\le \left(C_3C_4+\frac{C_5}{\varepsilon}\right)f(x).
   \end{equation*}
   For $x\in M$ such that $f(x)\le 1$, by the definition of $C_5$, we have,
   \begin{equation*}
          \left(\psi_0-\frac{C_3}{2}f\right)(x)\le C_5\le \frac{C_5}{\varepsilon}f(x).
   \end{equation*}
   We conclude that there exists a constant $C>0$ such that on $M$, $\psi_0\le Cf$ holds. Similarly, we can prove $\psi_0\ge-Cf$.
\end{proof}
\begin{remark}
    The bi-Lipschitz condition $ \frac{1}{C_0} g \leq\; g_{\psi_0}\leq C_0 g$ could not control the full covariant derivatives of $\psi_0$.
\end{remark}
Let $(C,g_0)$ be the Kähler cone appearing in Theorem \ref{CDS theorem}. It is clear that $(C,g_0)$ is the unique asymptotic cone of $(M,g)$. Let $(S,g_S)$ denote the closed Sasaki manifold such that $(C,g_0)$ is the Kähler cone with link $(S,g_S)$. Suppose $\psi_0$ is a Kähler potential satisfying Condition I. In fact, Condition I implies that $\psi_0$ is asymptotic to some function on $C$ at spatial infinity.

In the remainder of this article, we identify $M\setminus E$ with its $C\setminus \{o\}$ via the biholomorphism $\pi$ as in Theorem \ref{CDS theorem}.
\begin{prop}\label{geometric interpretation of condition I}
    Let $r$ denote the radial function of $(C,g_0)$, let $\psi_0$ be a Kähler potential satisfying Condition I. Then there exists a function $\psi_0^S\in C^1(S)$ such that at spatial infinity, we have
    \begin{equation}\label{asymp behavior condition I}
           \begin{split}
                &\left|\frac{2\psi_0}{r^2}-\psi_0^S\right|=O(r^{-2});\\
                &\left|\nabla^{g_S}\left(\frac{2\psi_0}{r^2}-\psi_0^S\right)\right|_{g_S}=O(r^{-2}).
           \end{split}
    \end{equation}
\end{prop}
\begin{proof}
    First we recall \cite[Lemma 3.4]{2025arXiv250500167C} and \cite[Corollary 5.6]{2025arXiv250500167C}: there exist constants $\lambda, C>0$ such that on $\{r^2\ge\lambda\}$, we have
    \begin{equation}\label{comparision}
        \begin{split}
        &\frac{1}{C}g\le g_0\le Cg;\\
        &|\nabla^gg_0|_g\le \frac{C}{r};\\
            &\left|f-\frac{r^2}{2}\right|\le C.
        \end{split}
    \end{equation}
    Since on $C$, $X=r\partial_r$ thanks to Theorem \ref{CDS theorem},
    thus $\frac{X}{2}\psi_0-\psi_0=O(1),|\nabla^g\left(\frac{X}{2}\cdot\psi_0-\psi_0\right)|_g=O(f^{-\frac{1}{2}})$ in Condition I as in Definition \ref{condition II} becomes
    \begin{equation*}
           \begin{split}
               & |r\frac{\partial_r\psi_0}{2}-\psi_0|=O(1);\\
               &\left|\nabla^{g_0}\left(r\frac{\partial_r\psi_0}{2}-\psi_0\right)\right|_{g_0}=O(r^{-1}).
           \end{split}
    \end{equation*}
    Define $\kappa:=\frac{2\psi_0}{r^2}$, thus at spatial infinity of $C$, we have,
    \begin{equation*}
           \begin{split}
                &|\partial_r\kappa|=O(r^{-3});\\
                &|\nabla^{g_0}\partial_r\kappa|_{g_0}=O(r^{-4}).
                \end{split}
    \end{equation*}
    Then $\kappa$ converges uniformly to a continuous function $\psi_0^S$ which is a function defined on $S$ such that $\left|\frac{2\psi_0}{r^2}-\psi_0^S\right|=O(r^{-2})$.
    
    Moreover, $|\nabla^{g_0}\partial_r\kappa|_{g_0}=O(r^{-4})$ implies that $|\nabla^{g_S}\partial_r\kappa|_{g_S}=O(r^{-3})$, thus $\psi_0^S$ is a $C^1$ function and $\left|\nabla^{g_S}\left(\frac{2\psi_0}{r^2}-\psi_0^S\right)\right|_{g_S}=O(r^{-2})$.
\end{proof}
\begin{remark}\label{refined initial condition of condition I}
In the proof of Proposition \ref{geometric interpretation of condition I}, we only require the conditions
\begin{equation*}
    \Bigl|\tfrac{X}{2}\cdot \psi_0 - \psi_0\Bigr| = O(1), \qquad \Bigl|\nabla^g\left(\tfrac{X}{2}\cdot \psi_0 - \psi_0\right)\Bigr|_g = O(r^{-1}).
\end{equation*}
Since [\eqref{asymp behavior condition I}, Proposition \ref{geometric interpretation of condition I}] holds, and because $\psi_0^S$ is a $C^1$ function defined on $S$, it follows that at spatial infinity
\begin{equation*}
    |\psi_0| = O(r^2), \qquad |\partial \psi_0|_{g_0} = O(r).
\end{equation*}
Moreover, by [\eqref{comparision}, Proposition \ref{geometric interpretation of condition I}], we conclude that on $M$
\begin{equation*}
    |\psi_0| = O(f), \qquad |\partial \psi_0|_g = O(f^{\frac{1}{2}}).
\end{equation*}

Here we reprove $|\psi_0|=O(f)$ by asymptotic condition instead of metric condition. We need to pay attention the fact that from the metric condition, we cannot conclude that $|\frac{X}{2}\cdot\psi_0-\psi_0|=O(1)$.
\end{remark}
\begin{prop}\label{geometric interpretation of condition II}
Let $\psi_0$ be a Kähler potential satisfying Condition II. Then the associated metric $g_{\psi_0}$ is asymptotically conical.

More precisely, there exists a Sasaki metric $g_S'$ on $S$  such that the corresponding metric space $(M,g_{\psi_0})$ admits a unique asymptotic K\"ahler cone $(C,g_0')$ with link $(S,g_S')$. Moreover, $(C,g_0')$ is bi-Lipschitz to $(C,g)$, and there exists a smooth function $\psi_0^S$, defined on the link $S$, such that
\begin{equation*}
    g_0' \;=\; \partial \bar{\partial} \!\left( \tfrac{1 + \psi_0^S}{2} \, r^2 \right).
\end{equation*}
\begin{proof}
Since the soliton metric $g$ is asymptotically conical (see \cite[Theorem A]{MR4711837}), indeed, at spatial infinity, we have that 
\begin{equation*}
    |(\nabla^{g_0})^j (g-g_0)|_{g_0}=O(r^{-2-j}), \quad \textrm{for all $j\in\N_0$}.
\end{equation*}
Then the asymptotic condition in Definition \ref{condition II} becomes
\begin{equation*}
    \left|(\nabla^{g_0})^j\left(\mathcal{L}_\frac{X}{2}g_{\psi_0}-g_{\psi_0}\right)\right|_{g_0}=O(r^{-2-j}).
\end{equation*}

   From Proposition \ref{geometric interpretation of condition I}, we know that $\psi_0$ is asymptotic to $\psi_0^S\frac{r^2}{2}$ for some $C^1$ function $\psi_0^S$. Now we prove the function $\psi_0^S$ has more regularities if $\psi_0$ satisfies Condition II. 

  Now we consider $t\Phi_t^*\psi_0$ for all $t>0$ where $\Phi_t$ is the flow of $-\frac{1}{2t}X$ and we compute
  \begin{equation*}
      \frac{\partial}{\partial t}(t\Phi_t^*\psi_0)=\Phi_t^*\left(\psi_0-\frac{X}{2}\cdot\psi_0\right)=O(1).
  \end{equation*}
  It follows that for any $x\in C$ such that $r(x)>0$, $t\Phi_t^*\psi_0(x)$ has a limit. Similarly, by the asymptotic condition in Definition \ref{condition II}, for any $\lambda>0$, on $\{r(x)^2>\lambda\}$, for all $k\in\N_0$, $(\nabla^{g_0})^k\left(\partial\bar\partial (t\Phi_t^*\psi_0)\right)$ converges uniformly when $t$ tends to $0$. 
  
  Let $(\rho,s)\in (0,\infty)\times S=C$, then $t\Phi_t^*\psi_0(\rho,s)=t\psi_0(\frac{\rho}{\sqrt{t}},s)$. By Proposition \ref{geometric interpretation of condition I}, we have
  \begin{equation*}
      \left|t\Phi_t^*\psi_0(\rho,s)-\frac{\rho^2}{2}\psi_0^S(s)\right|=\left|t\psi_0(\frac{\rho}{\sqrt{t}},s)-\frac{\rho^2}{2}\psi_0^S(s)\right|=tO(\frac{t}{\rho^2}).
  \end{equation*}
  This implies when $t$ goes to $0$, the function $t\Phi_t^*\psi_0$ converges to $\frac{r^2}{2}\psi_0^S$. Therefore for any $\lambda>0$, on $\{r(x)^2>\lambda\}$, for any $k\in \N_0$ $(\nabla^{g_0})^k\left(\partial\bar\partial (t\Phi_t^*\psi_0)\right)$ converges uniformly to $(\nabla^{g_0})^{k}\left(\partial\bar\partial\left(\frac{r^2}{2}\psi_0^S\right)\right)$.

For any $\lambda>0$, on $\{r(x)^2>\lambda\}$, by Theorem \ref{CDS theorem}, the self-similar solution $g(t)=t\Phi_t^*g$ converges smoothly uniformly to $g_0=\partial\bar\partial\left(\frac{r^2}{2}\right)$. It turns out that for any $\lambda>0$, on $\{r(x)^2>\lambda\}$, the metric $t\Phi_t^*g_{\psi_0}=g(t)+\partial\bar\partial (t\Phi_t^*\psi_0)$ converges smoothly uniformly to the metric $  g_0' \;=\; \partial \bar{\partial} \!\left( \tfrac{1 + \psi_0^S}{2} \, r^2 \right).$ 

We define $G=\frac{1+\psi_0^S}{2}$, now we prove there exists a Sasaki metric $g_S'$ on $S$ such that $(C,g_0')$ is a K\"ahler cone with link $(S,g_S')$.

  Since $JX$ is a Killing vector field of $g_{\psi_0}$, in Proposition \ref{without changing the metric}, we can assume that $JX\cdot\psi_0=0$, that is, $JX\cdot G=0$ without changing the metric $g_{\psi_0}$.

  First, we compute $\partial\bar\partial G$. Since $JX\cdot G=X\cdot G=0$, by the same reason as in Proposition \ref{metric+killing to asymptotic}, we conclude that for any vector field $Y$, $\partial\bar\partial G(X,Y)=Y\cdot(\frac{1}{2}X\cdot G)=0$. Furthermore $\mathcal{L}_X\partial\bar\partial G=\partial\bar\partial(X\cdot G)=0$, these facts imply that $\partial\bar\partial G$ is a symmetric 2-tensor which is only defined on $S$.

  Now we are going to find $g_S'$ and the corresponding radial function. Notice that
  \begin{equation*}
      \begin{split}
          g_0'&=G\partial\bar\partial\left(\frac{r^2}{2}\right)+\partial\left(\frac{r^2}{2}\right)\otimes\bar\partial G+\partial G\otimes\bar\partial \left(\frac{r^2}{2}\right)+\frac{r^2}{2}\partial\bar\partial G\\
          &=Gdr^2+Gr^2\left(g_S+\frac{1}{2G}\partial\bar\partial G\right)+r\partial r\otimes\bar\partial G+\partial G\otimes r\bar \partial r.
      \end{split}
  \end{equation*}
  The metric $g_0'$  is a conical metric, thus we conclude that $G>0$ and $\left(g_S+\frac{1}{2G}\partial\bar\partial G\right)$ is a Riemannian metric on $S$. Define $R:=\sqrt{G}r$ and $g_S':=g_S+\frac{1}{2G}\partial\bar\partial G$, now we prove
  \begin{equation*}
      g_0'=dR^2+R^2g_S'.
  \end{equation*}
  It suffices to prove that $dr\otimes dG=2\partial r\otimes\bar\partial G$. Consider $Y$ being a vector field defined on $S$, then we have
  \begin{equation*}
      dr\otimes dG(X,Y)=rY\cdot G.
  \end{equation*}
  On the one hand
  \begin{equation*}
    g_0'(X,Y)=  r\partial r\otimes\bar\partial G(X,Y)=\frac{r^2}{2}<Y,\bar\partial G>.
  \end{equation*}
  One the other hand since $JX\cdot R=0$ and $X$ is real-holomorphic,
  \begin{equation*}
     g_0'(X,Y)=\partial\bar\partial \left(\frac{R^2}{2} \right)(X,Y)=Y\cdot\left(\frac{1}{2}X\cdot R^2\right)=\frac{r^2}{2}Y\cdot G.
  \end{equation*}
  hence we have $Y\cdot G=<Y,\bar\partial G>$, therefore $dr\otimes dG=2\partial r\otimes\bar\partial G$.

  Observe that $g_0'=dR^2+R^2g_S'$, thus $g_0'=(\nabla^{g_0'})^2\frac{R^2}{2}$. At the same time $g_0'$ is K\"ahler, this implies that the radial vector field $R\partial_R$ of $g_0'$ is real-holomorphic. Now we prove that $R\partial_R=r\partial_r=X$. It suffices to prove the Hamiltonian potential of $X$ with respect to $g_0'$ is $\frac{R^2}{2}$. Suppose that locally we have $X=X^idz_i+X^{\bar j}dz_{\bar j}$, we compute
  \begin{equation*}
     g'_{0,i\bar j}X^i=X^i\partial_i\partial_{\bar j}\frac{R^2}{2}=\partial_{\bar j}\left(\frac{X}{2}\cdot\frac{R^2}{2}\right)=\partial_{\bar j}\frac{R^2}{2}.
  \end{equation*}
  Hence we get $r\partial_r=X=R\partial_R$.

  Since $G$ is a function which is defined on $S$, then $g_0'$ is bi-Lipschitz equivalent to $g_0$, i.e. there exists a constant $C_0>0$ such that $\frac{1}{C_0}g_0\le g_0'\le C_0 g_0$.

  We now prove that $(M,g_{\psi_0})$ is asymptotically conical. In fact, we prove that for all $j\in\N_0$, on $\{r(x)^2\ge 1\}$,
 \begin{equation*}
     |(\nabla^{g_0'})^{j}\left(g_{\psi_0}-g_0'\right)|_{g_0'}=O(R^{-2-j}).
 \end{equation*}
  For any $j\in\N_0$, we have 
  \begin{equation*}
      \frac{\partial}{\partial t}\left((\nabla^{g_0})^j\partial\bar\partial (t\Phi_t^*g_{\psi_0})\right)=\Phi_t^*\left((\nabla^{g_0})^j\left(\mathcal{L}_{\frac{X}{2}}g_{\psi_0}-g_{\psi_0}\right)\right).
  \end{equation*}
  By the asymptotic condition in Definition \ref{condition II}, we have
 \begin{equation*}
     \left|(\nabla^{g_0})^j\left(\mathcal{L}_\frac{X}{2}g_{\psi_0}-g_{\psi_0}\right)\right|_{g_0}=O(r^{-2-j}).
 \end{equation*}
  We conclude that $\left|(\nabla^{g_0})^j(g_{\psi_0}-g_0')\right|_{g_0}=O(r^{-2-j})$ for all $j\in\N_0$. Thus it is sufficient to show that for any $j\in\N_0$,
 \begin{equation*}
     |(\nabla^{g_0})^{j}g_0'|_{g_0}=O(r^{-j}).
 \end{equation*}
 Notice that $g_0'-Gg_0=r^2R$, where $R=Gg_S'-Gg_S$ is a symmetric 2-tensor defined on $S$, now we compute for all $j\in\N_0$, on the one hand,
 \begin{equation*}
     (\nabla^{g_0})^j(g_0'-Gg_0)=(\nabla^{g_0})^jg_0'-g_0(\nabla^{g_0})^jG=(\nabla^{g_0})^jg_0'-(\nabla^{g_0} )^jG\otimes g_0.
 \end{equation*}
 On the other hand,
\begin{equation*}
    (\nabla^{g_0})^j(g_0'-Gg_0)=(\nabla^{g_0})^j(r^2R)=\sum_{i\le j}(\nabla^{g_0})^ir^2*(\nabla^{g_0})^{j-i}R.
\end{equation*}
 Hence we conclude that \begin{equation*}
     (\nabla^{g_0})^jg_0'=\sum_{i\le j}(\nabla^{g_0})^ir^2*(\nabla^{g_0})^{j-i}R+(\nabla^{g_0} )^jG\otimes g_0,
 \end{equation*}
 and there exists a dimensional constant $C(n)>0$ such that
 \begin{equation*}
     \left|(\nabla^{g_0})^jg_0'\right|_{g_0}\le C(n)\sum_{i\le j}|(\nabla^{g_0})^ir^2|_{g_0}|(\nabla^{g_0})^{j-i}R|_{g_0}+|g_0|_{g_0}|(\nabla^{g_0} )^jG|_{g_0}=O(r^{-j}).
 \end{equation*}
\end{proof}
\end{prop}
\begin{remark}\label{refined geometric condition II}
In the proof of Proposition \ref{geometric interpretation of condition II}, we relied solely on the assumption
\begin{equation*}
    \bigl|(\nabla^g)^j\left(\mathcal{L}_{\frac{X}{2}}g_{\psi_0}-g_{\psi_0}\right)\bigr|_g=O\left(f^{-1-\frac{j}{2}}\right), \quad j\in\mathbb{N}_0.
\end{equation*}
We now show that if this condition holds, then it follows that
\begin{equation*}
    \bigl|(\nabla^g)^j(g_{\psi_0}-g)\bigr|_g=O\left(f^{-\frac{j}{2}}\right), \quad j\in\mathbb{N}_0.
\end{equation*}
Since for any $j\in\N_0$, $|(\nabla^{g_0})^{j}\left(g_{\psi_0}-g_0'\right)|_{g_0}=O(r^{-2-j}),$ it suffices to show that for any $j\in\N_0$,
\begin{equation*}
    |(\nabla^{g_0})^{j}g_0'|_{g_0}=O(r^{-j}).
\end{equation*}
And this is proved at the end of the proof of Proposition \ref{geometric interpretation of condition II}.

Now we show that the metric $g_{\psi_0}$ is bi-Lipschitz to $g$. Since $g_{\psi_0}$ (resp. $g$) is asymptotically conical to conical metric $g_0'$ (resp. $g_0$), and $g_0'$ is bi-Lipschitz to $g_0$, we conclude that $g_{\psi_0}$ is bi-Lipschitz to $g$, that is, there exists a constant $C_0>1$ such that $\frac{1}{C_0}g\le g_{\psi_0}\le C_0 g$.
\end{remark}
\section{Shi's solution}\label{Shi's solution section}
In 1989, Shi \cite{MR1001277} constructed a solution to the Ricci flow on non compact manifolds. Given the correspondence between the Kähler-Ricci flow and the normalized Kähler-Ricci flow (Proposition \ref{correspondence}), the Ricci flow evolving from the initial data $g_{\psi_0}$ arises as a natural candidate for the establishment of Theorem \ref{longtime existence theorem}.

Let $(M,g,X)$ be an asymptotically conical gradient K\"ahler-Ricci expander and let $\psi_0$ be a smooth function defined on $M$ satisfying Condition I, as in Definition \ref{condition I}.
\subsection{Shi's solution to Ricci flow}
Since the initial metric $g_{\psi_0}$ is a complete Kähler metric with bounded Riemann curvature, Shi's existence theorem \cite{MR1001277} ensures the existence of a complete solution to the Kähler-Ricci flow.
\begin{theorem}[Shi's solution]\label{shi's solution}
    There exists an $1<\tilde{T}_{\max}\le\infty$ such that
there is a complete smooth solution to Ricci flow $(g_\varphi(t))_{t\in[1,\tilde{T}_{\max})}$ starting from $g_{\psi_0}$ such that
    \begin{enumerate}
        \item $\frac{\partial}{\partial t}g_\varphi(t)=-\Ric(g_\varphi(t))$, for all $t\in [1,\tilde{T}_{\max})$;
        \item for all $1<T<\tilde{T}_{\max}$, for all $m\in\N_0$, there exists a constant $C_m=C(n,m,T,\psi_0)$ such that
        \begin{equation*}
\sup_M\left|\left(\nabla^{g_\varphi(t)}\right)^m\Rm(g_\varphi(t))\right|_{g_\varphi(t)}\le \frac{C_m}{(t-1)^{\frac{m}{2}}},\quad \textrm{for all $t\in (1,T]$};
        \end{equation*}
        \item if $\tilde{T}_{\max}\neq+\infty$, then $\limsup_{t\to{\tilde{T}}_{\max}}\sup_M|\Rm(g_\varphi(t)|_{g_\varphi(t)}=\infty$.
    \end{enumerate}
\end{theorem}
\begin{proof}
  We apply Shi's existence theorem \cite[Theorem 1.1]{MR1001277}; for additional details on Shi's estimates, see \cite[Theorem 6.9]{MR2274812}.
\end{proof}
\begin{remark}
Here we slightly abuse notation by writing $g_\varphi$, since a priori the metric $g_\varphi(t)$ is not necessarily a solution to the Kähler–Ricci flow.
\end{remark}
Since Shi's solution possesses uniformly bounded Riemann curvature on the entire space time $M\times [1,T]$ for all $T<\tilde{T}_{\max}$, several geometric properties are preserved along the Ricci flow. In particular, both the Kähler structure and the Killing vector field $JX$ are preserved throughout the evolution.

\begin{lemma}\label{Kahler for shi's solution}
 The solution $(g_\varphi(t))_{t \in [1, \tilde{T}_{\max})}$ from Theorem \ref{shi's solution} is a solution to the Kähler-Ricci flow; that is, $g_\varphi(t)$ remains a Kähler metric for all $t \in [1, \tilde{T}_{\max})$.
\end{lemma}
\begin{proof}
 On each compact time interval, the Riemann curvature of $g_\varphi(t)$ remains bounded. Moreover, since the initial metric $g_{\psi_0}$ is Kähler, it follows from Huang and Tam's theorem \cite[Theorem 1.1]{MR3749193} that $g_\varphi(t)$ remains a Kähler metric for all $t \in [1, \tilde{T}_{\max})$.
\end{proof}
\begin{lemma}
    The Reeb vector field $JX$ is a Killing vector field for the solution $g_\varphi(t)$ from Theorem \ref{shi's solution} for all $t\in [1,\tilde{T}_{\max})$.
\end{lemma}
\begin{proof}
    On each compact time interval, the Riemann curvature of $g_\varphi(t)$ is bounded and the Reeb vector field $JX$ is a Killing vector field for $g_{\psi_0}$, due to the uniqueness theorem of Chen-Zhu \cite[Theorem 1.1]{MR2260930}, $JX$ is Killing for all $g_\varphi(t)$ with $t\in [1,\tilde{T}_{\max})$.
\end{proof}
Thanks to the correspondence (Proposition \ref{correspondence}) between the Kähler-Ricci flow and the normalized Kähler-Ricci flow, we can obtain a solution to the normalized Kähler-Ricci flow based on Shi's solution in Theorem \ref{shi's solution}.
\begin{lemma}\label{correspondence of NKRF and RF}
    The flow $g_\psi(\tau)$ defined by $g_\psi(\tau):=e^{-\tau}\Phi_{e^{-\tau}}^*g_\varphi(e^\tau)$ for $\tau\in [0,\log\tilde{T}_{\max})$ is a solution to the normalized K\"ahler-Ricci flow, where $\Phi_t=\Phi_X^{-\frac{1}{2}\log t}$ for $t>0$ and $\Phi_X^\cdot\colon M\mapsto M$ denotes the flow of $X$.
\end{lemma}
\begin{proof}
    The metric $g_\psi(\tau)$ is a K\"ahler metric due to the Lemma \ref{Kahler for shi's solution} and the fact that $X$ is real-holomorphic. Now we compute
    \begin{equation*}
        \begin{split}
            \frac{\partial}{\partial\tau}g_\psi(\tau)&=-e^{-\tau}\Phi_{e^{-\tau}}^*g_\varphi(e^\tau)+e^{-\tau}\frac{\partial}{\partial\tau}\left(\Phi^*_{e^{-\tau}} g_\varphi(e^\tau)\right)\\
            &=-g_\psi(\tau)+e^{-\tau}\mathcal{L}_{\frac{X}{2}}\left(\Phi_{e^{-\tau}}^*g_\varphi(e^\tau)\right)-\Phi_{e^{-\tau}}^*\Ric(g_\varphi(e^\tau))\\
            &=\mathcal{L}_{\frac{X}{2}}g_\psi(\tau)-\Ric(g_\psi(\tau))-g_\psi(\tau).
        \end{split}
    \end{equation*}
\end{proof}
By combining Theorem \ref{shi's solution} and Lemma \ref{correspondence of NKRF and RF}, we obtain the following theorem:
\begin{theorem}\label{shi's solution to NKRF}
   There exists an $0<T_{\max}\le\infty$ such that there is a complete smooth solution to normalized K\"ahler-Ricci flow $g_\psi(\tau)_{\tau\in [0,T_{\max})}$ starting from initial data $g_{\psi_0}$. Moreover, 
    \begin{enumerate}
        \item for all $T<T_{\max}$, for all $m\in\N_0$ there exists a constant $C_m=C(n,m,\psi_0,T)>0$ such that
        \begin{equation*}
            \sup_M\left|\left(\nabla^{g_\psi(\tau)}\right)^m\Rm(g_\psi(\tau))\right|_{g_\psi(\tau)}\le\frac{C_m}{\tau^m},\quad \textrm{for all $\tau\in (0,T]$.}
        \end{equation*}
        \item If $T_{\max}\neq+\infty$, then $\limsup_{\tau\to T_{\max}}\sup_M|\Rm(g_{\psi}(\tau))|_{g_\psi(\tau)}=\infty$.
    \end{enumerate}
\end{theorem}
\subsection{Reduction to complex Monge-Ampère equation}\label{reduction}
In Kähler geometry, a common strategy is to reduce the Kähler-Ricci flow equation to a complex Monge-Ampère equation for the Kähler potential.
 
 Our first reduction concerns the initial Kähler potential $\psi_0$ which satisfies Condition I. We show that there exists another Kähler potential, inducing the same initial metric, but enjoying additional symmetries.

\begin{prop}\label{without changing the metric}
    There exists a smooth function $\bar\psi_0\in C^{\infty}(M)$ satisfying Condition I such that $JX\cdot \psi_0=0$ and $\partial\bar\partial\psi_0=\partial\bar\partial\bar\psi_0$.

    Moreover, if $\psi_0$ satisfies Condition II, then $\bar\psi_0$ also satisfies Condition II.
\end{prop}
\begin{proof}
      We consider the isometry group of $(M,g)$ that fixes $E$ as in Theorem \ref{CDS theorem} endowed with the topology induced by uniform convergence on compact
subsets of $M$. By the Arzel\`a-Ascoli theorem, this is a compact Lie group. And it has been proved that the flows of $X$ and $JX$ preserve $E$ (see \cite[Lemma 2.6]{MR4711837}). Thus in particular, the flow of $JX$ lies in such isometry group, taking the closure of the flow of $JX$ in this group therefore yields the holomorphic isometric action of a torus $T$ on $(M,g)$. Let $\mu_T$ be the normalized Haar measure of $T$, and we define $\bar\psi_0:=\int_Ta^*\psi_0d\mu_T(a)$, since $f$ and $\partial\bar\partial\psi_0$ is $T-$invariant, hence $\bar\psi_0$ satisfies the same condition as $\psi_0$. Moreover, since $\bar\psi_0$ is the average of $\psi_0$ by the action of $T$, hence $JX\cdot\bar\psi_0=0$.
\end{proof}
Since $\bar\psi_0$ and $\psi_0$ determine the same initial metric and $\bar\psi_0$ satisfies Condition I (resp. II) as in Definition \ref{condition I} (resp. Definition \ref{condition II}) if $\psi_0$ satisfies Condition I (resp. II), from now on we use $\bar\psi_0$ instead of $\psi_0$, for convenience, we still denote it by $\psi_0$.
\begin{prop}\label{cohomology prop}
    Let $g_\varphi(t)_{t\in [1,\tilde{T}_{\max})}$ be Shi's solution to K\"ahler-Ricci flow as in Theorem \ref{shi's solution}. Let $g(t)_{t\ge 1}:=t\Phi_t^*g$ be the self-similar solution associated with $(M,g,X)$ as in Proposition \ref{self-similar solution}. Then there exists a smooth real-valued function $\varphi\in C^\infty(M\times [1,\tilde{T}_{\max}))$ such that 
    \begin{equation*}
        \begin{split}
            &g_\varphi(t)=g(t)+\partial\bar\partial\varphi(t),\\
            &JX\cdot \varphi(t)=0,\quad \forall t\in [1,\tilde{T}_{\max}),\\
            &\varphi(1)=\psi_0.
        \end{split}
    \end{equation*}
\end{prop}
\begin{proof}
    Recall the K\"ahler-Ricci flow equation
   \begin{equation*}
       \begin{split}
           &\frac{\partial}{\partial t}g_\varphi(t)=-\Ric(g_\varphi(t));\\
           &\frac{\partial}{\partial t}g(t)=-\Ric(g(t)).
       \end{split}
   \end{equation*}
   We compute the difference, if $\omega_\varphi(t)$ (resp. $\omega(t)$) denotes the K\"ahler form of $g_\varphi(t)$(resp. $g(t)$), we get
   \begin{equation*}
       \frac{\partial}{\partial t}\left(g_\varphi(t)-g(t))=\Ric(g(t)\right)-\Ric(g_\varphi(t))=\partial\bar\partial\log\frac{\omega_\varphi(t)^n}{\omega(t)^n}.
   \end{equation*}
   By integrating, we get
   \begin{equation*}
       \begin{split}
           g_\varphi(t)-g(t)&=g_{\psi_0}-g+\partial\bar\partial\int_1^t\log\frac{\omega_\varphi(s)^n}{\omega(s)^n}ds\\
    &=\partial\bar\partial\left(\psi_0+\int_1^t\log\frac{\omega_\varphi(s)^n}{\omega(s)^n}ds\right).
       \end{split}
   \end{equation*}

   Now we keep the notations as in Proposition \ref{without changing the metric} and we define 
\begin{equation*}
    \varphi(t):=\int_Ta^*\left(\psi_0+\int_1^t\log\frac{\omega_\varphi(s)^n}{\omega(s)^n}ds\right)d\mu_T(a).
\end{equation*}
We observe that $\varphi(1)=\psi_0$.
Since $\varphi(t)$ is the average of a function over the torus $T$, we get $JX\cdot \varphi(t)=0$. Moreover, since $\mathcal{L}_{JX}g_{\varphi}(t)=0$, we have
\begin{equation*}    \mathcal{L}_{JX}\left(\partial\bar\partial\left(\psi_0+\int_1^t\log\frac{\omega_\varphi(s)^n}{\omega(s)^n}ds\right)\right)=0,
\end{equation*}
which implies that $\partial\bar\partial\left(\psi_0+\int_1^t\log\frac{\omega_\varphi(s)^n}{\omega(s)^n}ds\right)$ is invariant under the action of $T$. It turns out that
\begin{equation*}
\partial\bar\partial\varphi(t)=\int_Ta^*\left(\partial\bar\partial\left(\psi_0+\int_1^t\log\frac{\omega_\varphi(s)^n}{\omega(s)^n}ds\right)\right)d\mu_T(a)=\partial\bar\partial\left(\psi_0+\int_1^t\log\frac{\omega_\varphi(s)^n}{\omega(s)^n}ds\right).
\end{equation*}
Hence $ g_\varphi(t)-g(t)=\partial\bar\partial\varphi(t)$.
\end{proof}
We now establish a Liouville-type lemma, which allows us to reduce the Kähler–Ricci flow equation to a complex Monge–Ampère equation.
\begin{lemma}\label{reduction lemma}
    If $\kappa$ is a smooth real-valued function such that $JX\cdot\kappa=0$ and $\partial\bar\partial\kappa=0$, then $\kappa$ is a constant function.
\end{lemma}
\begin{proof}
    We can find a detailed proof in \cite[Proposition 4.4]{2025arXiv250500167C}.

    Let $\pi$ be the K\"ahler resolution as in Theorem \ref{CDS theorem} form $M$ to $(C,g_0)$ with exceptional set $E$. If $X^{1,0}$ denotes the holomorphic part of $X$, then $JX\cdot\kappa$ is equivalent to saying that $X\cdot\kappa=2X^{1,0}\cdot\kappa.$

    Since $\partial\bar\partial\kappa=0$, hence $\bar\partial\left(X\cdot\kappa\right)=\bar\partial\left(2X^{1,0}\cdot\kappa\right)=0$. This implies that $X\cdot\kappa$ is a real-valued holomorphic function, hence a constant function. Therefore, there exists a constant $c$ such that $X\cdot\kappa=c$.
    
    Now we consider $c=\pi_*(X\cdot\kappa)=d\pi(X)\cdot\pi_*\kappa$, since $d\pi(X)=r\partial_r$, where $r$ denotes the radial function, we deduce that $\pi_*\kappa=c\log r+\alpha$ with $\alpha$ a smooth function that only defines on the link $S$ of $(C,g_0)$. We compute the Laplacian of $\pi_*\kappa$ with respect to $g_0$. On the one hand, it should be 0. On the other hand
    \begin{equation*}
           0=\Delta_{g_0}\pi_*\kappa(t)=cr^{-2}(2n-2)+r^{-2}\Delta_{g_S}\alpha.
    \end{equation*}
    Since $n\ge 2$, this implies that $c=0$ and $\alpha$ is a constant. As a consequence, $\kappa$ is constant on $M\setminus E$. Because $M\setminus E$ is dense, we deduce that $\kappa$ is constant.
\end{proof}
\begin{prop}[Complex Monge-Amp\`ere equation]\label{complex monge-ampere} Let $g_\varphi(t)_{t\in [1,\tilde{T}_{\max})}$ be Shi's solution to K\"ahler-Ricci flow as in Theorem \ref{shi's solution}. Let $\varphi$ be the K\"ahler potential as in Proposition \ref{cohomology prop}.

    There exists a smooth real-valued function $\bar\varphi\in C^\infty(M\times [1,\tilde{T}_{\max}))$ such that
    \begin{equation*}
           \begin{split}
                &\frac{\partial}{\partial t}\bar\varphi(t)=\log\frac{\omega_\varphi(t)^n}{\omega(t)^n},\\
                &JX\cdot\bar\varphi(t)=0, \quad\forall t\in [1,\tilde{T}_{\max}),\\
                &\bar\varphi(1)-\varphi(1)=\textrm{constant}.
           \end{split}
    \end{equation*}
\end{prop}
\begin{proof}
  Recall in Proposition \ref{cohomology prop}, we have that:
\begin{equation*}
    \partial\bar\partial\left(\frac{\partial}{\partial t}\varphi(t)\right)=\partial\bar\partial\log\frac{\omega_\varphi(t)^n}{\omega(t)^n}.
\end{equation*}
Thanks to Lemma \ref{reduction lemma}, there exists a smooth function $c(t)$ which only depends on time such that 
\begin{equation*}
    \frac{\partial}{\partial t}\varphi(t)=\log\frac{\omega_\varphi(t)^n}{\omega(t)^n}+c(t).
\end{equation*}
Let $C(t)$ be the primitive function of $c(t)$ and let $\bar\varphi(t)=\varphi(t)-C(t)$, we get
\begin{equation*}
    \frac{\partial}{\partial t}\bar\varphi(t)=\log\frac{\omega_\varphi(t)^n}{\omega(t)^n}.
\end{equation*}

In particular $\bar\varphi(1)=\varphi(1)-C(1)$, thus $\bar\varphi(1)$ satisfies Condition I (resp. II) as in Definition \ref{condition I} (resp. Definition \ref{condition II}) if $\psi_0$ satisfies Condition I (resp. II).
\end{proof}
\textbf{From now on, we still use $\varphi$ to denote the K\"ahler potential $\bar\varphi$, and we still use $\psi_0$ to denote $\bar\varphi(1)$ which satisfies Condition I (resp. II) as in Definition \ref{condition I} (resp. Definition \ref{condition II}) if $\psi_0$ satisfies Condition I (resp. II)..}

As a result of Theorem \ref{shi's solution to NKRF} and Proposition \ref{complex monge-ampere}, we can drive an evolution equation to the K\"ahler potential along the normalized K\"ahler-Ricci flow:
\begin{prop}[Normalized complex Monge-Amp\`ere equation]\label{normalized complex monge-ampere}Let $g_\psi(\tau)_{\tau\in[0,T_{\max})}$ be Shi's solution to the normalized K\"ahler-Ricci flow as in Theorem \ref{shi's solution to NKRF}. Let $\omega_{\psi}$ be the K\"ahler form of $g_\psi$. Then there exists a real-valued function $\psi\in C^\infty(M\times [0,T_{\max}))$ such that for $\tau\in [0,T_{\max})$ we have $g_\psi(\tau)=g+\partial\bar\partial\psi(\tau)$, and
\begin{equation}\label{NCMA eq}
\begin{split}
        &\frac{\partial}{\partial \tau}\psi(\tau)=\log\frac{\omega_\psi(\tau)^n}{\omega^n}+\frac{X}{2}\cdot\psi(\tau)-\psi(\tau),\\
        &JX\cdot\psi(\tau)=0,\quad\forall \tau\in [0,T_{\max}),\\
        &\psi(0)=\psi_0.
\end{split}
\end{equation}
\begin{proof}
Let $\varphi$ be the solution to the complex Monge-Amp\`ere equation as in Proposition \ref{complex monge-ampere}. Let $\Phi_t=\Phi_X^{-\frac{1}{2}\log t}$ for $t>0$ and $\Phi_X^\cdot\colon M\mapsto M$ denotes the flow of $X$.

    For all $\tau\in [0,T_{\max})$, we define $\psi(\tau)=e^{-\tau}\Phi_{e^{-\tau}}^*\varphi(e^\tau)$. Therefore $\psi(0)=\varphi(1)=\psi_0$. Moreover, we can verify that $g_\psi(\tau)=g+\partial\bar\partial\psi(\tau)$ and 
    \begin{equation*}
\begin{split}
        &\frac{\partial}{\partial \tau}\psi(\tau)=\log\frac{\omega_\psi(\tau)^n}{\omega^n}+\frac{X}{2}\cdot\psi(\tau)-\psi(\tau),\\
        &JX\cdot\psi(\tau)=0.
\end{split}
\end{equation*}
\end{proof}
\end{prop}
\section{Proof of long time existence theorem \ref{longtime existence theorem}}\label{proof of 1}
In this section, we give a detailed proof of Theorem \ref{longtime existence theorem}.
Let $g(t)=t\Phi_{t}^*g$ be the self-similar solution associated with asymptotically conical gradient K\"ahler-Ricci expander $(M,g,X)$ as in Proposition \ref{self-similar solution}. Let $\psi_0$ be a smooth function defined on $M$ satisfying Condition I.

Let $(g_\psi(\tau))_{\tau\in [0,T_{\max})}$ be the solution to normalized K\"ahler-Ricci flow as in Theorem \ref{shi's solution to NKRF} and let $(g_{\varphi}(t))_{t\in [1,e^{T_{\max}})}$ be Shi's solution as in Theorem \ref{shi's solution}. 
\subsection{Rough estimates along normalized K\"ahler-Ricci flow}
In this section, we show some rough estimates results along this normalized K\"ahler-Ricci flow.
\begin{prop}\label{rough metric equivalence}
    For all $0\le T<T_{\max}$, there exists a constant $B>1$ such that for all $(x,\tau)\in M\times [0,T]$,
    \begin{equation*}
        \frac{1}{B}g\le g_\psi(x,\tau)\le B g.
    \end{equation*}
\end{prop}
\begin{proof}
  According to Theorem \ref{shi's solution to NKRF}, there exists a constant $B(T)>0$ such that 
    \begin{equation*}
        \sup_{M\times[1,e^T]}|\Rm(g_\varphi(\cdot))|_{g_\varphi(\cdot)}\le B(T).
    \end{equation*}
   Hence by \cite[Corollary 6.11]{MR2274812}, there exists a constant $B_1>1$ such that
    \begin{equation*}
        \frac{1}{B_1}g_\varphi(x,1)\le g_\varphi(x,t)\le B_1g_\varphi(x,1),\quad \textrm{for all $(x,t)\in M\times [1,e^T]$}.
    \end{equation*}
    Moreover, we observe that for all $t\in [1,e^T]$
    \begin{equation*}
        \sup_M|\Rm(g(t))|_{g(t)}=\frac{1}{t}\sup_M|\Rm(g)|_g\le \sup_M|\Rm(g)|_g<\infty.
    \end{equation*}
    We deduce that there exists a constant $B_2>1$ such that for all $t\in [1,e^T]$
    \begin{equation*}
        \frac{1}{B_2}g(0)\le g(t)\le B_2 g(0).
    \end{equation*}
    At time $t=1$, $g_\varphi(1)=g_{\psi_0},g(1)=g$, the metric assumption in condition I, as in Definition \ref{condition I} implies that there exists a constant $B_3>1$ such that
    \begin{equation*}
        \frac{1}{B_3}g(t)\le g_\varphi(t)\le B_3 g(t),\quad \textrm{for all $t\in [1,e^T]$}.
    \end{equation*}
    Since $g(t)=t\Phi_{t}^*g$ and $g_\psi(\tau)=e^{-\tau}\Phi_{e^{-\tau}}^*g_\varphi(e^\tau)$ for $t=e^\tau$, there exists a constant $B=B(T)>1$ such that
    \begin{equation*}
        \frac{1}{B}g\le g_\psi(\tau)\le Bg,\quad \textrm{for all $\tau\in [0,T]$}.
    \end{equation*}
\end{proof}
Along the normalized K\"ahler-Ricci flow, the Christoffel symbols are also roughly bounded on every compact time interval.
\begin{prop}\label{rough c3 estimate}
    For all $0\le T<T_{\max}$, there exists a constant $B=B(T)>0$ such that on $M\times [0,T]$,
    \begin{equation*}
        |\nabla^gg_\psi|_g\le B.
    \end{equation*}
\end{prop}
\begin{proof}
    Recall the correspondence $\nabla^gg_\psi(\tau)=e^{-\tau}\Phi_{e^{-\tau}}^*\left(\nabla^{g(e^\tau)}g_\varphi(e^\tau)\right)$ and the fact $\nabla^{g(e^\tau)}g_\varphi(e^\tau)=\nabla^{g}g(e^\tau)*g_\varphi(e^\tau)+\nabla^gg_\varphi(e^\tau)$. Hence we only need to estimate $\nabla^gg_\varphi(e^\tau)$, and the control $\nabla^{g}g(e^\tau)$ comes from the same proof.

    Now we estimate $\nabla^gg_\varphi(e^\tau)$. Take $t=e^\tau$ and define $y(s)=|\nabla^gg_\varphi(x,s)|_g^2$ for $1\le s\le t$ and fixed $x\in M$. Now we compute
    \begin{equation*}
        \frac{d}{ds}y(s)\le 2|\nabla^gg_\varphi(s)|_g|\nabla^g\Ric(g_\varphi(s))|_g.
    \end{equation*}
Since $\nabla^g\Ric(g_\varphi(s))=\nabla^gg_\varphi(s)*\Ric(g_\varphi(s))+\nabla^{g_\varphi(s)}\Ric(g_\varphi(s))$, we conclude that there is a dimensional constant $C(n)$ such that
\begin{equation*}
    \frac{d}{ds}y(s)\le C(n)\sqrt{y(s)}\left(\sqrt{y(s)}|\Ric(g_\varphi(s))|_g+|\nabla^{g_\varphi(s)}\Ric(g_\varphi(s))|_g\right)
\end{equation*}
    
    By the curvature estimates of Shi's solution in Lemma \ref{shi's solution} and the metric equivalence indicated in Proposition \ref{rough metric equivalence}, there is a constant $C_1>1$ such that for all $1\le s\le e^T$
    \begin{equation*}
         |\Ric(g_\varphi(s))|_g\le C_1,
    \end{equation*}
    and
    \begin{equation*}
        |\nabla^{g_\varphi(s)}\Ric(g_\varphi(s))|_g\le \frac{C_1}{\sqrt{s-1}}.
    \end{equation*}
    Hence we get
    \begin{equation*}
       \frac{d}{ds}y(s)\le C(n)C_1\left(y(s)+\frac{\sqrt{y(s)}}{\sqrt{s-1}}\right),
    \end{equation*}
    which implies that
    \begin{equation*}
        \frac{d}{ds}\sqrt{y(s)}\le \frac{1}{2}C(n)C_1\left(\sqrt{y(s)}+\frac{1}{\sqrt
        {s-1}}\right).
    \end{equation*}
   When $s=1$, $y(1)=|\nabla^gg_{\psi_0}|_g^2$ is bounded globally due to Definition \ref{condition I}. Since $\frac{1}{\sqrt{s-1}}$ is integrable at $1$ thus by integration, we conclude that there is a constant $B_0>0$ such that $y(t)\le B_0$ for all $t\in [1,e^T]$. Hence, there is a constant $B>0$ such that $|\nabla^gg_\psi|_g\le B$ holds on $M\times [0,T]$ as required.
\end{proof}
A key observation is that the vector field $X$ remains a gradient vector field along the normalized Kähler–Ricci flow.
\begin{def+prop}
Let $f$ be the normalized Hamiltonian potential of $X$ with respect to $g$, as in Lemma \ref{soliton indentities}. Let $g_\psi(\tau)_{\tau \in [0,T_{\max})}$ denote the solution to the normalized Kähler–Ricci flow given by Theorem \ref{shi's solution to NKRF}. Then, for each $\tau \in [0,T_{\max})$, the vector field $X$ is gradient with respect to $g_\psi(\tau)$, with Hamiltonian potential
\begin{equation*}
    f + \tfrac{1}{2} \, X \cdot \psi(\tau).
\end{equation*}
Hence we define $f_\psi(\tau):= f + \frac{1}{2} \, X \cdot \psi(\tau)$ for each $\tau \in [0,T_{\max})$.
\end{def+prop}
\begin{proof}
     Take local holomorphic coordinates of $M$, since $X$ is real holomorphic and $JX\cdot\psi=0$, let $X^{1,0}$ denote the holomorphic part of $X$, then we have $X^{1,0}\cdot\psi=\frac{1}{2}X\cdot\psi$. This allows us to compute
    \begin{equation*}
        \begin{aligned}
            \bar{\partial}_{\bar{j}}(f+\frac{1}{2}X\cdot\psi)&=\bar{\partial}_{\bar{j}}f+\bar{\partial}_{\bar{j}}(X^{1,0}\cdot\psi)\\
            &=g_{i\bar{j}}X^{i}+(X^{1,0})^i\partial_i\bar{\partial}_{\bar{j}}\psi\\
            &=g_{i\bar{j}}X^{i}+X^{i}\partial_i\bar{\partial}_{\bar{j}}\psi\\
            &={g_{\psi}}_{i\bar{j}}X^i.
        \end{aligned}
    \end{equation*}
    Hence $f+\frac{1}{2}X\cdot\psi$ is a Hamiltonian potential of $X$ with respect to $g_\psi$.
\end{proof}
These two Hamiltonian potentials are roughly comparable along the flow.
\begin{prop}\label{universal lower bound of f_psi}
     For all $0\le T<T_{\max}$, there exists a constant $B>1$ such that for all $(x,\tau)\in M\times [0,T]$,
    \begin{equation}\label{rough equivalence of potentials}
        \frac{1}{B}f(x)\le f_\psi(x,\tau)\le B f(x).
    \end{equation}
    Moreover, there exists an $\varepsilon>0$ such that for each $\tau\in [0,T_{\max})$ $\inf_M f_\psi(\tau)\ge\inf_M f\ge\varepsilon>0.$
\end{prop}
\begin{proof}
  By Proposition \ref{rough metric equivalence}, there exists a constant $B=B(T)>1$ such that 
  \begin{equation*}
      \frac{1}{B}g\le g_\psi(\tau)\le Bg,\quad \textrm{for all $\tau\in [0,T]$}.
  \end{equation*}
  We observe that
  \begin{equation*}
      X\cdot( f_\psi-Bf)=g_\psi(X,X)-Bg(X,X)\le 0.
  \end{equation*}
  Let $\Phi_X^\cdot$ denotes the flow of $X$
 then for all $x\in M, \rho<0$, $(f_\psi-Bf)(x)\le (f_\psi-Bf)(\Phi_X^\rho(x))$.
  
  For all $\rho<0$, $f(\Phi^\rho_X(x))\le f(x)$, thus $\{\Phi_X^{\rho}(x)\}_{\rho<0}\subset \{f\le f(x)\}$, where $\{f\le f(x)\}$ is a compact set. The set $\{\Phi_X^{\rho}(x)\}_{\rho<0}$ is a pre-compact set. We suppose that $x_0\in M$ is a limit point of $\Phi^{\rho_i}_X(x)$ with $\lim_{i\to\infty}\rho_i=-\infty$. Hence, we have $(f_\psi-Bf)(x)\le (f_\psi-Bf)(x_0)$. Since $f(\Phi_X^{\rho}(x))$ is increasing with respect to $\rho$, thus
  \begin{equation*}
      \lim_{\rho\to-\infty}f(\Phi_X^{\rho}(x))=\lim_{i\to\infty}f(\Phi^{\rho_i}_X(x))=f(x_0).
  \end{equation*}
  For all $\eta\in\R$, on the one hand,
  \begin{equation*}
      \lim_{i\to\infty}f\circ\Phi^\eta_X(\Phi^{\rho_i}_X(x))=f(\Phi^\eta_X(x_0)).
  \end{equation*}
  On the other hand
  \begin{equation*}
      \lim_{i\to\infty}f\circ\Phi^\eta_X(\Phi^{\rho_i}_X(x))=\lim_{i\to\infty}f(\Phi^{\rho_i+\eta}_X(x))=\lim_{\rho\to-\infty}f(\Phi_X^{\rho}(x))=f(x_0).
  \end{equation*}
  It turns out that $f(x_0)=f(\Phi^\eta_X(x_0))$ for all $\eta\in\R$, which implies that $X(x_0)=0$. In particular, $(f_\psi-Bf)(x_0)=f(x_0)-Bf(x_0)=(1-B)f(x_0)\le 0$. Therefore, $(f_\psi-Bf)(x)\le 0$ for all $x\in M$.

  For the same reason, we prove that
  \begin{equation*}
      f_\psi\ge \frac{1}{B}f.
  \end{equation*}
  From \eqref{rough equivalence of potentials} and the fact that $f$ is a proper function, we deduce that $f_\psi(\tau)$ is a proper function for all $\tau\in [0,T_{\max})$. Suppose that for fixed $\tau$, $x_0\in M$ is a minimum point of $f_\psi(\tau)$, i.e. $f_\psi(x_0,\tau)=\inf_M f_\psi(\tau)$, then $X(x_0)=\left(\nabla^{g_\psi(\tau)}f_\psi(\tau)\right)(x_0)=0$, hence $f_\psi(x_0,\tau)=f(x_0)\ge\inf_Mf$. Thanks to Proposition \ref{lower bound of f} there exists an $\varepsilon>0$ such that on $M\times [0,T_{\max})$,
  \begin{equation}\label{fin inf for fpsi}
      \inf_Mf_\psi\ge\inf_Mf\ge\varepsilon>0.
  \end{equation}
\end{proof}
The following Proposition gives us rough estimates on the K\"ahler potential $\psi$ and its radial derivative $X\cdot\psi$.
\begin{prop}\label{rough estimate of psi and xpsi}
     For all $0\le T<T_{\max}$, there exists a constant $B>1$ such that for all $(x,\tau)\in M\times [0,T]$,
    \begin{equation*}
        \left(|X\cdot\psi|+|\psi|\right)(x,\tau)\le B f.
    \end{equation*}
\end{prop}
\begin{proof}
    First we prove that there exists a constant $B_1=B_1(T)$ such that on $M\times [0,T]$
    \begin{equation*}
        |X\cdot\psi|\le Bf.
    \end{equation*}
    Notice that on the one hand by [\eqref{rough equivalence of potentials}, Proposition \ref{rough equivalence of potentials}], there exists a constant $B>1$ such that
    \begin{equation*}
        X\cdot\psi=2f_\psi-2f\le 2f_\psi\le 2Bf.
    \end{equation*}
    On the other hand,
    \begin{equation*}
        X\cdot\psi=2f_\psi-2f\ge -2f.
    \end{equation*}
    Thus, we can take $B_1=2B$ so that $|X\cdot\psi|\le B_1f$ on $M\times [0,T]$.

    Now we consider $\varphi$ defined in Proposition \ref{complex monge-ampere}. Since $\dot\varphi(t)=\log\frac{\omega_\varphi(t)^n}{\omega(t)^n}$, Lemma \ref{rough metric equivalence} ensures that there exists a constant $B_2=B_2(T)>0$ such that for $t\in[1,e^T]$
    \begin{equation*}
        |\varphi(t)-\varphi(1)|\le B_2.
    \end{equation*}
    Since $e^{-\tau}\Phi_{e^{-\tau}}^*\varphi(e^\tau)=\psi(\tau)$ and $\varphi(1)=\psi_0$, we have for $\tau\in [0,T]$, $x\in M$,
    \begin{equation*}
        |e^\tau\psi(\Phi_{e^\tau}(x),\tau)-\psi_0(x)|\le B_2.
    \end{equation*}
    Since $\psi_0=O(f)$ by Condition I (see Definition \ref{condition I}), there exists a constant $B_3>0$ such that $\psi_0(x)\le B_3f(x)$, hence
    \begin{equation*}
         |e^\tau\psi(\Phi_{e^\tau}(x),\tau)|\le B_2+B_3f(x).
    \end{equation*}
    This yields that for all $(x,\tau)\in M\times [0,T]$
    \begin{equation}\label{4,28}
         |\psi(x,\tau)|\le B_2e^{-\tau}+B_3e^{-\tau}f(\Phi_{e^{-\tau}}(x)).
    \end{equation}
    Let us consider the function $e^{-\rho}f(\Phi_{e^{-\rho}}(x))$ for all $\rho\ge 0$. We compute as follows:
    \begin{equation}\label{4,29}
        \frac{d}{d\rho}\left(e^{-\rho}f(\Phi_{e^{-\rho}}(x))\right)=e^{-\rho}(-f+|\partial f|_g^2)(\Phi_{e^{-\rho}}(x))=e^{-\rho}(-n-R_\omega)(\Phi_{e^{-\rho}}(x))\le 0.
    \end{equation}
    This implies that $e^{-\tau}f(\Phi_{e^{-\tau}}(x))\le f(x)$ for $(x,\tau)\in M\times [0,T]$, the combination of \eqref{4,28} and \eqref{4,29} leads to
    \begin{equation*}
         |\psi(x,\tau)|\le B_2e^{-\tau}+B_3f(x)\le B_2+B_3f(x)\le \left(\frac{B_2}{\varepsilon}+B_3\right)f(x).
    \end{equation*}
    We can take $B=\left(\frac{B_2}{\varepsilon}+B_3\right)$ such that on $M\times [0,T]$,
    \begin{equation*}
        |\psi|\le Bf.
    \end{equation*}
\end{proof}
\subsection{Uniform estimates on space-time along normalized K\"ahler-Ricci flow}
In this section, we omit the time dependence in $g_\psi(\tau)$ for ease of reading. Let $\omega_\psi$ be the K\"ahler form of $g_\psi$. We denote the drift Laplacian $\Delta_{\omega_\psi} + \frac{X}{2}$ by $\Delta_{\omega_\psi, X}$, and define $\dot{\psi} := \frac{\partial \psi}{\partial \tau}$. This first observation is that the Hamiltonian potential $f_\psi$ satisfies a \emph{good} evolution equation.
\begin{prop}\label{proposition 4.6}
    The Hamiltonian potential $f_\psi$ satisfies the following evolution equation:
    \begin{equation}\label{evolution equation of fpsi}
        \frac{\partial}{\partial \tau}f_\psi=\Delta_{\omega_\psi,X}f_\psi-f_\psi.
    \end{equation}
\end{prop}
\begin{proof}
    First we compute
    \begin{equation*}
        \frac{\partial}{\partial\tau}f_\psi=\frac{X}{2}\cdot\dot\psi=\frac{X}{2}\cdot\left(\log\frac{\omega_\psi^n}{\omega^n}+\frac{X}{2}\cdot\psi-\psi\right).
    \end{equation*}
    Since \begin{equation*}
        \frac{X}{2}\cdot\log\frac{\omega_\psi^n}{\omega^n}=\tr_{\omega_\psi}\mathcal{L}_{\frac{X}{2}}\omega_\psi-\tr_{\omega}\mathcal{L}_{\frac{X}{2}}\omega=\Delta_{\omega_\psi}f_\psi-\Delta_\omega f,
    \end{equation*} and 
    \begin{equation*}
        \frac{X}{2}\cdot\left(\frac{X}{2}\cdot\psi-\psi\right)=\frac{X}{2}\cdot\left(f_\psi-f\right)-f_\psi+f=\frac{X}{2}\cdot f_\psi-f_\psi+f-|\partial f|_g^2.
    \end{equation*}
    The soliton identities (Lemma \ref{soliton indentities}) tell us that $f=\Delta_\omega f+|\partial f|_g^2$, hence
    \begin{equation*}
        \frac{\partial}{\partial \tau}f_\psi=\Delta_{\omega_\psi,X}f_\psi-f_\psi,
    \end{equation*}
    holds.
\end{proof}
The Hamiltonian potential $f_\psi$ plays a role analogous to $\frac{r^2}{2}$ in the case of the Gaussian soliton. On Euclidean space, a maximum principle holds for subsolutions to the heat equation with exponential growth (see \cite[Theorem 6 of Chapter 2.3]{MR2597943}). Therefore, it is natural to expect a similar maximum principle to hold on our asymptotically conical gradient Kähler-Ricci expander $(M, g, X)$.

\begin{prop}[Maximum principle]\label{maximum principle}
    If $u$ is a real-valued smooth function defined on $M\times [0,T_{\max})$ such that \begin{enumerate}
        \item  the function $u$ is a subsolution to the drift heat equation along normalized K\"ahler-Ricci flow, i.e. $\left(\frac{\partial}{\partial\tau}-\Delta_{\omega_\psi,X}\right)u\le 0$;
        \item for all $0<T<T_{\max}$, there exists a constant $B>0$ and integer $k>0$ such that $u\le Bf^k$;
        \item $\sup_Mu|_{\tau=0}\le 0$.
    \end{enumerate}
    Then $\sup_{M\times [0,T_{\max})}u\le 0$.
\end{prop}
\begin{proof}
    We fix $0<T<T_{\max}$ and consider a weight function $h(x,\tau):=\frac{e^\tau f_\psi(x,\tau)}{T-\beta \tau}$ with $\beta>0$ to be determined for $(x,\tau)\in M\times [0,\frac{T}{2\beta}]$. By Proposition \ref{rough metric equivalence} and Proposition \ref{universal lower bound of f_psi}, there exists a constant $B>0$ such that
    \begin{equation*}
        2|\partial f_\psi|_{g_\psi}^2=|X|_{g_\psi}^2\le B|X|_g^2=2B|\partial f|_g^2\le 2Bf\le 2B^2f_\psi.
    \end{equation*}

    Now we compute 
    \begin{equation*}
        \begin{split}
            \left(\frac{\partial}{\partial\tau}-\Delta_{\omega_\psi,X}\right)e^h(x,\tau)&=\left(\left(\frac{\partial}{\partial\tau}-\Delta_{\omega_\psi,X}\right)h-|\partial h|_{g_\psi}^2\right)e^h\\
            &=\left(\frac{\beta e^\tau f_\psi-e^{2\tau}|\partial f_\psi|^2_{g_\psi}}{(T-\beta\tau)^2}+\frac{\left(\frac{\partial}{\partial\tau}-\Delta_{\omega_\psi,X}\right)(e^\tau f_\psi(x,\tau))}{T-\beta\tau}\right)e^h\\
            &=\frac{\beta e^\tau f_\psi-e^{2\tau}|\partial f_\psi|^2_{g_\psi}}{(T-\beta\tau)^2}e^h\\
            &\ge \frac{\beta e^\tau f_\psi-e^{2\tau}B^2f_\psi}{(T-\beta\tau)^2}e^h.
        \end{split}
    \end{equation*}
    Here in line 3, we have used the \eqref{evolution equation of fpsi} in Proposition \ref{proposition 4.6}.
    Then we take $\beta>e^TB^2+1$ so that 
    \begin{equation*}
         \left(\frac{\partial}{\partial\tau}-\Delta_{\omega_\psi,X}\right)e^h>0,
    \end{equation*}
    on $M\times [0,\frac{T}{2\beta}]$. 

    Now we consider auxiliary function $u_\eta=u-\eta e^h$ for $\eta>0$. Since the growth of $u$ at infinity is polynomial according to assumption (ii), the auxiliary function $u_\eta$ tends to $-\infty$ uniformly at spatial infinity. Then $u_\eta$ admits a maximum point. Now we apply the weak maximum principle, we conclude that
    \begin{equation*}
        \sup_{M\times [0,\frac{T}{2\beta}]}u_\eta\le \sup_M u_\eta|_{\tau=0}\le 0.
    \end{equation*}
Letting $\eta \to 0$, we obtain $\sup_{M \times [0, \frac{T}{2\beta}]} u \le 0$. We then repeat the above argument inductively on the intervals $[\frac{T}{2\beta}, \frac{T}{\beta}], [\frac{T}{\beta}, \frac{3T}{2\beta}], \ldots$, and ultimately conclude that $\sup_{M \times [0, T]} u \le 0$. Therefore, we have

\begin{equation*}
    \sup_{M \times [0, T_{\max})} u \le 0,
\end{equation*}
as desired.
\end{proof}
The first application of the above maximum principle is the following: we can give a precise bound for $\dot\psi$ on the entire space-time $M\times [0,T_{\max})$.
\begin{corollary}\label{fin estimate of dotpsi}
    There exists a constant $C>0$ such that for all $(x,\tau)\in M\times [0,T_{\max})$
    \begin{equation*}
        |\dot\psi(x,\tau)|\le Ce^{-\tau}.
    \end{equation*}
\end{corollary}
\begin{proof}
    Recall the complex Monge-Amp\`ere equation 
    \begin{equation*}
        \dot\psi=\log\frac{\omega_\psi^n}{\omega^n}+\frac{1}{2}X\cdot\psi-\psi.
    \end{equation*}
    Now we compute the evolution equation of $\dot\psi$:
    \begin{equation}\label{evolution equation of dotpsi}
       \left( \frac{\partial}{\partial\tau}-\Delta_{\omega_\psi,X}\right)\dot\psi=-\dot\psi.
    \end{equation}
    It turns out that $e^{\tau}\dot\psi$ is a solution to the drift heat equation along the normalized K\"ahler-Ricci flow, i.e.
    \begin{equation*}
         \left( \frac{\partial}{\partial\tau}-\Delta_{\omega_\psi,X}\right)e^\tau\dot\psi=0.
    \end{equation*}
At time $\tau = 0$, we have

\begin{equation*}
    \dot{\psi}(0) = \log\left(\frac{\omega_{\psi_0}^n}{\omega^n}\right) + \frac{X}{2} \cdot \psi_0 - \psi_0.
\end{equation*}

By the asymptotic condition in Definition \ref{condition I}, there exists a constant $C > 0$ such that $|\dot{\psi}(0)| \le C$. For all $0 < T < T_{\max}$, Proposition \ref{rough estimate of psi and xpsi} guarantees the existence of a constant $B > 0$ such that $|\dot{\psi}| \le Bf$. Applying Proposition \ref{maximum principle}, we conclude that

\begin{equation*}
    e^\tau |\dot{\psi}| \le C,
\end{equation*}
holds on $M \times [0, T_{\max})$.
\end{proof}
\begin{remark}
   From the above estimates, the Kähler potential $\psi$ converges uniformly to a continuous function $\psi_{\infty}$ on $M$. In general, however, one cannot conclude that $\psi_{\infty}$ is constant, unlike in the compact case.

When $M$ is a closed complex manifold with $c_1(M)<0$, we have $X \equiv 0$ and $\omega$ is a Kähler–Einstein metric. By \cite[Chapter 3.4.1]{MR3185331}, the Kähler potential then converges uniformly to a smooth function, which must be constant by the uniqueness of Kähler–Einstein metrics on closed manifolds.

In contrast, for non compact manifolds the general uniqueness of asymptotically conical gradient expanders is not known, and hence no such conclusion can be drawn.
\end{remark}
Another application of the maximum principle above is to control $|X|_{g_\psi}^2$. Before doing so, we require Bochner-type formulas to derive the evolution equation of the gradient of a solution to the heat equation along the normalized Kähler-Ricci flow.
\begin{lemma}\label{order 1 heat equation}
    Let $u$ be a smooth function defined on $M\times [0,T_{\max})$ such that
    \begin{equation*}
        \left(\frac{\partial}{\partial\tau}-\Delta_{\omega_\psi,X}\right)u=-u.
    \end{equation*}
    Then the evolution equation of $|\nabla^{g_\psi} u|_{g_\psi}^2$ along the normalized K\"ahler-Ricci flow is given by:
    \begin{equation}\label{bonchner for gradient}
         \left(\frac{\partial}{\partial\tau}-\Delta_{\omega_\psi,X}\right)|\nabla^{g_\psi} u|_{g_\psi}^2= -|\nabla^{g_\psi} u|_{g_\psi}^2-|(\nabla^{g_\psi})^2u|^2_{g_\psi}.
    \end{equation}
\end{lemma}
\begin{proof}
    First we compute
    \begin{equation}\label{lemma 5.3 1}
       \begin{split}
            \frac{\partial}{\partial\tau}|\nabla^{g_\psi}u|_{g_\psi}^2&=-(\mathcal{L}_{\frac{X}{2}}g_\psi-\Ric(g_\psi)-g_\psi)(\nabla^{g_\psi}u,\nabla^{g_\psi}u)+2g_\psi(\nabla^{g_\psi}\Delta_{\omega_\psi,X}u,\nabla^{g_\psi}u)\\
            &\quad -2|\nabla^{g_\psi}u|_{g_\psi}^2\\
            &= -|\nabla^{g_\psi}u|_{g_\psi}^2-(\mathcal{L}_{\frac{X}{2}}g_\psi-\Ric(g_\psi))(\nabla^{g_\psi}u,\nabla^{g_\psi}u)+g_\psi(\nabla^{g_\psi}\Delta_{g_\psi}u,\nabla^{g_\psi}u)\\
            &\quad +g_\psi(\nabla^{g_\psi}(X\cdot u),\nabla^{g_\psi}u)\\
            &=-|\nabla^{g_\psi}u|_{g_\psi}^2+\Ric(g_\psi)(\nabla^{g_\psi}u,\nabla^{g_\psi}u)+g_\psi(\nabla^{g_\psi}\Delta_{g_\psi}u,\nabla^{g_\psi}u)\\
            &\quad +\nabla^{g_\psi}u\cdot g_\psi(\nabla^{g_\psi}f_\psi,\nabla^{g_\psi}u)-\mathcal{L}_{\frac{X}{2}}g_\psi(\nabla^{g_\psi}u,\nabla^{g_\psi}u)\\
            &=-|\nabla^{g_\psi}u|_{g_\psi}^2+\Ric(g_\psi)(\nabla^{g_\psi}u,\nabla^{g_\psi}u)+g_\psi(\nabla^{g_\psi}\Delta_{g_\psi}u,\nabla^{g_\psi}u)\\
            &\quad +(\nabla^{g_\psi})^2u(\nabla^{g_\psi}f_\psi,\nabla^{g_\psi}u).
       \end{split}
    \end{equation}
    By Bochner's formula,
    \begin{equation}\label{lemma 5.3 2}
      \begin{split}
            \Delta_{\omega_\psi}|\nabla^{g_\psi}u|_{g_\psi}^2&=\frac{1}{2}\Delta_{g_\psi}|\nabla^{g_\psi}u|_{g_\psi}^2\\
            &=|(\nabla^{g_\psi})^2u|^2+\Ric(g_\psi)(\nabla^{g_\psi}u,\nabla^{g_\psi}u)+g_\psi(\nabla^{g_\psi}\Delta_{g_\psi}u,\nabla^{g_\psi}u).
      \end{split}
    \end{equation}
    Moreover,
    \begin{equation}\label{lemma 5.3 3}
        \frac{X}{2}\cdot |\nabla^{g_\psi}u|_{g_\psi}^2=g_\psi(\nabla^{g_\psi}_X\nabla^{g_\psi}u,\nabla^{g_\psi}u)=(\nabla^{g_\psi})^2u(\nabla^{g_\psi}f_\psi,\nabla^{g_\psi}u).
    \end{equation}
    Combining \eqref{lemma 5.3 1} \eqref{lemma 5.3 2} and \eqref{lemma 5.3 3}, we have \eqref{bonchner for gradient} as desired.
\end{proof}
\begin{corollary}\label{bound of X}
     There exists a constant $A>0$ such that 
     \begin{equation*}
         |X|_{g_\psi}^2\le Af_\psi,
     \end{equation*}
     holds on $M\times [0,T_{\max})$.
\end{corollary}
\begin{proof}
    Since $X=\nabla^{g_\psi}f_\psi$ and the evolution equation of $f_\psi$ from Proposition \ref{proposition 4.6}:
    \begin{equation*}
        \left(\frac{\partial}{\partial\tau}-\Delta_{\omega_\psi,X}\right)f_\psi=-f_\psi.
    \end{equation*}
    Thus by [\eqref{bonchner for gradient}, Lemma \ref{order 1 heat equation}], we deduce that
    \begin{equation*}
        \left(\frac{\partial}{\partial\tau}-\Delta_{\omega_\psi,X}\right)e^\tau|X|^2_{g_\psi}\le 0.
    \end{equation*}
    By the initial condition (metric condition) in Definition \ref{condition I}, there exists a constant $A>0$ such that $|X|^2_{g_{\psi_0}}\le Af_\psi(0)$. Hence by Proposition \ref{maximum principle}, $e^\tau|X|^2_{g_\psi}-Ae^\tau f_\psi\le 0$ on $M\times [0,T_{\max})$. Indeed, the assumption (ii) of Proposition \ref{maximum principle} is justified as follows, for any $0<T<T_{\max}$, by Proposition \ref{rough metric equivalence}, there exists a constant $B>1$ such that $|X|_{g_\psi}^2\le B|X|_g^2\le Bf$.
    
    Therefore, \begin{equation*}
         |X|_{g_\psi}^2\le Af_\psi,
     \end{equation*}
     holds on $M\times [0,T_{\max})$.
\end{proof}
With control over $|X|_{g_\psi}^2$, we construct a useful barrier function in the remainder of this article, which will allow us to effectively and uniformly eliminate negligible terms.
\begin{lemma}[Barrier function]\label{barrier function}
    There exist constants $K,\delta>0$ such that on $M\times [0,T_{\max})$
    \begin{equation*}
         \left(\frac{\partial}{\partial\tau}-\Delta_{\omega_\psi,X}\right)\frac{1}{f_\psi+K}\ge \delta \frac{1}{f_\psi}.
    \end{equation*}
\end{lemma}
\begin{proof}
    Let $A>0$ be the constant stated in Corollary \ref{bound of X}, take $K\ge A$, now we compute
    \begin{equation*}
         \begin{split}
             \left(\frac{\partial}{\partial\tau}-\Delta_{\omega_\psi,X}\right)\frac{1}{f_\psi+K}&=-\frac{(\partial_\tau-\Delta_{\omega_\psi,X})f_\psi}{(f_\psi+K)^2}-2\frac{|\partial f_\psi|_{g_\psi}^2}{(f_\psi+K)^3}\\
             &=\frac{f_\psi}{(f_\psi+K)^2}-\frac{|X|^2_{g_\psi}}{(f_\psi+K)^3}\\
             &\ge \frac{f_\psi(f_\psi+K)-Af_\psi}{(f_\psi+K)^3}\ge \frac{f_\psi^2}{(f_\psi+K)^3}.
         \end{split}
    \end{equation*}
    Due to [\eqref{fin inf for fpsi}, Proposition \ref{rough equivalence of potentials}], $\frac{f_\psi^2}{(f_\psi+K)^3}\ge (\frac{\varepsilon}{\varepsilon+K})^3\frac{1}{f_\psi}$. Taking $\delta=(\frac{\varepsilon}{\varepsilon+K})^3$, we have that 
    \begin{equation*}
         \left(\frac{\partial}{\partial\tau}-\Delta_{\omega_\psi,X}\right)\frac{1}{f_\psi+K}\ge \delta \frac{1}{f_\psi}.
    \end{equation*}
\end{proof}
The following lemma, which can be viewed as a variation of Proposition \ref{maximum principle}, will be fundamental in the proofs of Theorem \ref{longtime existence theorem} and Theorem \ref{convergence theorem}.
\begin{lemma}\label{useful lemma}
    Assume that $u\in C^\infty(M\times [0,T_{\max}))$ satisfies 
    \begin{equation*}
        \left(\frac{\partial}{\partial\tau}-\Delta_{\omega_\psi,X}\right)u\le Df_\psi^{-1},
    \end{equation*}
    for some positive constant $
    D>0$. Moreover, assume that for all $0<T<T_{\max}$, there exists $k\in\N_0$ and $B>0$ such that $u\le Bf^k$ on $M\times [0,T]$. Then there exists a constant $C>0$ which only depends on $D$, the lower bound of $f$ as in Proposition \ref{lower bound of f} and $\delta,K$ as in Lemma \ref{barrier function} such that
    \begin{equation*}
        \sup_{M\times [0,T_{\max})}u\le \sup_M u(0)+C.
    \end{equation*}
\end{lemma}
\begin{proof}
    Let $\delta,K$ be the constants as in Lemma \ref{barrier function}, we have
    \begin{equation*}
         \left(\frac{\partial}{\partial\tau}-\Delta_{\omega_\psi,X}\right)\frac{1}{f_\psi+K}\ge \delta\frac{1}{f_\psi}.
    \end{equation*}
    Hence
    \begin{equation*}
        \left(\frac{\partial}{\partial\tau}-\Delta_{\omega_\psi,X}\right)\left(u-\frac{D}{\delta}\frac{1}{f_\psi+K}\right)\le 0.
    \end{equation*}
    Since $u-\frac{D}{\delta}\frac{1}{f_\psi+K}$ is bounded by $Bf^k$ for some positive constant $B$ and integer $k$ on $M\times[0,T]$ for all $0<T<T_{\max}$, the maximum principle (Proposition \ref{maximum principle}) implies that
    \begin{equation*}
         \sup_{M\times [0,T_{\max})}u-\frac{D}{\delta}\frac{1}{f_\psi+K}\le \sup_Mu(0)-\frac{D}{\delta}\frac{1}{f_\psi+K}(0)\le \sup_Mu(0).
    \end{equation*}
    By Proposition \ref{rough equivalence of potentials}, there exists a constant $\varepsilon>0$ such that $\inf_Mf_\psi\ge\varepsilon>0$. Take $C=\frac{D}{\delta}\frac{1}{\varepsilon+K}$ to get
    \begin{equation*}
         \sup_{M\times [0,T_{\max})}u\le \sup_Mu(0)+C,
    \end{equation*}
    as desired.
\end{proof}
With all the tools developed above, we now present the second uniform estimate for the Kähler potential $\psi$:
\begin{corollary}\label{estimate of gradient dotpsi}
     There exists a constant $C>0$ such that on $M\times [0,T_{\max})$ we have:
     \begin{equation*}
         |\nabla^{g_\psi}\dot\psi|_{g_\psi}^2\le Ce^{-2\tau}\frac{1}{f_\psi}.
     \end{equation*}
\end{corollary}
\begin{proof}
    First we compute 
    \begin{equation}\label{order 1}
        \begin{split}
            \left(\frac{\partial}{\partial\tau}-\Delta_{\omega_\psi,X}\right)(e^{2\tau}f_\psi|\nabla^{g_\psi}\dot\psi|_{g_\psi}^2)&=e^{2\tau}f_\psi\left(\frac{\partial}{\partial\tau}-\Delta_{\omega_\psi,X}\right)|\nabla^{g_\psi}\dot\psi|_{g_\psi}^2\\
            &\quad +|\nabla^{g_\psi}\dot\psi|_{g_\psi}^2\left(\frac{\partial}{\partial\tau}-\Delta_{\omega_\psi,X}\right)(e^{2\tau}f_\psi)\\
            &\quad -2e^{2\tau}\Re\left(<\partial f_\psi,\bar\partial|\nabla^{g_\psi}\dot\psi|_{g_\psi}^2>_{g_\psi}\right)\\
            &=-e^{2\tau}f_\psi|(\nabla^{g_\psi})^2\dot\psi|^2_{g_\psi}-e^{2\tau}f_\psi|\nabla^{g_\psi}\dot\psi|_{g_\psi}^2\\
            &\quad +|\nabla^{g_\psi}\dot\psi|_{g_\psi}^2e^{2\tau}f_\psi-e^{2\tau}X\cdot |\nabla^{g_\psi}\dot\psi|_{g_\psi}^2\\
            &=-e^{2\tau}f_\psi|(\nabla^{g_\psi})^2\dot\psi|^2_{g_\psi} -e^{2\tau}X\cdot |\nabla^{g_\psi}\dot\psi|_{g_\psi}^2.
        \end{split}
    \end{equation}
   Let $A>1$ be the constant as in Corollary \ref{bound of X}. We notice that by the Cauchy-Schwarz inequality
    \begin{equation}
        X\cdot |\nabla^{g_\psi}\dot\psi|_{g_\psi}^2\le 2|(\nabla^{g_\psi})^2\dot\psi|_{g_\psi}|X|_{g_\psi}|\nabla^{g_\psi}\dot\psi|_{g_\psi}\le \frac{1}{A}|X|^2_{g_\psi}|(\nabla^{g_\psi})^2\dot\psi|_{g_\psi}^2+A|\nabla^{g_\psi}\dot\psi|^2_{g_\psi}.
    \end{equation}
    Hence we conclude that
    \begin{equation}
       \begin{split}
            -e^{2\tau}X\cdot |\nabla^{g_\psi}\dot\psi|_{g_\psi}^2&\le e^{2\tau}\left(\frac{1}{A}|X|^2_{g_\psi}|(\nabla^{g_\psi})^2\dot\psi|_{g_\psi}^2+A|\nabla^{g_\psi}\dot\psi|^2_{g_\psi}\right)\\
            &\le Ae^{2\tau}|\nabla^{g_\psi}\dot\psi|^2_{g_\psi}+e^{2\tau}f_\psi|(\nabla^{g_\psi})^2\dot\psi|^2_{g_\psi}
       \end{split}
    \end{equation}
    Thus we get
    \begin{equation}
        \left(\frac{\partial}{\partial\tau}-\Delta_{\omega_\psi,X}\right)(e^{2\tau}f_\psi|\nabla^{g_\psi}\dot\psi|_{g_\psi}^2)\le Ae^{2\tau}|\nabla^{g_\psi}\dot\psi|^2_{g_\psi}.
    \end{equation}
    Now we consider the evolution equation of $e^{2\tau}\dot\psi^2$:
    \begin{equation*}
        \left(\frac{\partial}{\partial\tau}-\Delta_{\omega_\psi,X}\right)e^{2\tau}\dot\psi^2=2e^{2\tau}\dot\psi^2+e^{2\tau}\left(\frac{\partial}{\partial\tau}-\Delta_{\omega_\psi,X}\right)\dot\psi^2=-e^{2\tau}|\nabla^{g_\psi}\dot\psi|_{g_\psi}^2.
    \end{equation*}
    As a consequence, on $M\times [0,T_{\max})$, we have
    \begin{equation*}
        \left(\frac{\partial}{\partial\tau}-\Delta_{\omega_\psi,X}\right)(e^{2\tau}f_\psi|\nabla^{g_\psi}\dot\psi|_{g_\psi}^2+Ae^{2\tau}\dot\psi^2)\le 0.
    \end{equation*}
    Observe that $Ae^{2\tau}\dot\psi^2$ is uniformly bounded on $M\times [0,T_{\max})$ by Corollary \ref{fin estimate of dotpsi}. To use the maximum principle stated in Proposition \ref{maximum principle}, we need a rough estimate on $|\nabla^{g_\psi}\dot\psi|^2_{g_\psi}$. In the remainder of this proof, we use $B$ to denote a constant that may vary from line to line.

    Since $|\nabla^{g_\psi}\dot\psi|_{g_\psi}^2=2|\partial\dot\psi|_{g_\psi}^2$, by Proposition \ref{rough c3 estimate} we only need a rough estimate of $|\partial\dot\psi|_g^2$. Recall that
    \begin{equation*}
        \partial\dot\psi=\partial\log\frac{\omega_\psi^n}{\omega^n}+\partial\left(\frac{X}{2}\cdot\psi\right)-\partial\psi.
    \end{equation*}
    Since $X$ is real holomorphic and $JX\cdot\psi=0$, $\partial\left(\frac{X}{2}\cdot\psi\right)=\partial(X^{\bar i}\bar\partial_i\psi)=X^{\bar i}\partial\bar\partial_i\psi$. Hence for all $0<T<T_{\max}$, by Proposition \ref{rough metric equivalence} there exists a constant $B>0$ such that on $M\times [0,T]$,
    \begin{equation}\label{coro 5.6 1}
        \left|\partial\left(\frac{X}{2}\cdot\psi\right)\right|_g\le Bf.
    \end{equation}
    Moreover, $|\nabla^g\log\frac{\omega_\psi^n}{\omega^n}|_g\le C(n)|\nabla^gg_\psi|_g$ by some dimensional constant $C(n)$. By Proposition \ref{rough c3 estimate}, for all $0<T<T_{\max}$, there is a constant $B>0$ such that on $M\times [0,T]$,
    \begin{equation}\label{coro 5.6 2}
        \left|\partial \log\frac{\omega_\psi^n}{\omega^n}\right|_g\le B.
    \end{equation}
    Combining \eqref{coro 5.6 1} and \eqref{coro 5.6 2}, we get for $\tau\in [0,T]$, $x\in M$
    \begin{equation}\label{ode for nabla psi}
        |\partial\dot\psi|_g(x,\tau)\le B\left(f(x)+|\partial\psi|_g(x,\tau)\right).
    \end{equation}
When $\tau=0$, Remark \ref{refined initial condition of condition I} implies that $|\partial \psi_0|_g=O(f^{\frac{1}{2}})$. By integration, there are constants $B,D>0$ such that $|\partial\psi|_g(x,\tau)\le Bf(x)$ and hence $|\partial\dot\psi|_g(x,\tau)\le Bf(x)$ for all $(x,\tau)\in M\times [0,T]$.

    With this rough estimate, by the maximum principle as in Proposition \ref{maximum principle}, 
    \begin{equation*}
        \sup_{M\times [0,T_{\max})}e^{2\tau}f_\psi|\nabla^{g_\psi}\dot\psi|_{g_\psi}^2+Ae^{2\tau}\dot\psi^2\le \sup_Mf_\psi(0)|\nabla^{g_\psi}\dot\psi|_{g_\psi}^2(0)+A\dot\psi^2(0).
    \end{equation*}
    Since $\psi_0$ satisfies Condition I, we have 
    \begin{equation*}
        \sup_Mf_\psi(0)|\nabla^{g_\psi}\dot\psi|_{g_\psi}^2(0)=\sup_{M}f_\psi(0)\left|\partial\left(\log\frac{\omega_{\psi_0}^n}{\omega^n}+\frac{X}{2}\cdot\psi_0-\psi_0\right)\right|_{g_{\psi_0}}^2<\infty.
    \end{equation*}
    
    Take $C=\sup_Mf_\psi(0)|\nabla^{g_\psi}\dot\psi|_{g_\psi}^2(0)+A\dot\psi^2(0)<\infty$, we have that 
    \begin{equation*}
        |\nabla^{g_\psi}\dot\psi|_{g_\psi}^2\le e^{-2\tau}\frac{C}{f_\psi},
    \end{equation*}
    holds on $M\times [0,T_{\max})$.
\end{proof}
\begin{corollary}\label{good cosntant for psi and psi dot}
    There exists a constant $C>1$ such that for all $(x,\tau)\in M\times [0,T_{\max})$,
    \begin{enumerate}
        \item $|\dot\psi|+|X\cdot\dot\psi|(x,\tau)\le Ce^{-\tau}$;
        \item $|\frac{X}{2}\cdot\psi-\psi|(x,\tau)\le C;$
        \item $\frac{1}{C}\omega^n\le \omega_\psi(\tau)^n\le C\omega^n$;
        \item $|\psi|(x,\tau)\le Cf(x)$, $\psi(x,\tau)\ge -Cf_\psi(x,\tau)$;
        \item $\frac{1}{C}f(x)\le f_\psi(x,\tau)\le Cf(x)$.
        
    \end{enumerate}
\end{corollary}
\begin{proof}
    We notice that (i) comes by Corollary \ref{fin estimate of dotpsi}, Corollary \ref{bound of X} and Corollary \ref{estimate of gradient dotpsi}. Since $\frac{X}{2}\cdot\psi_0-\psi_0=O(1)$, hence by integration, we have $|\frac{X}{2}\cdot\psi-\psi|\le C$. Together with the bound of $\dot\psi$, we get (iii).

    For (iv), by integration, we get $|\psi-\psi_0|\le C$, since $\psi_0=O(f)$ and $f$ is bounded form below by $\varepsilon>0$, hence there exists a constant $C>0$ such that $|\psi|\le Cf$. Moreover, recall that $\psi$ is a supersolution to the drift heat equation along the normalized K\"ahler-Ricci flow, i.e.
    \begin{equation*}
        \dot\psi\ge \Delta_{\omega_\psi,X}\psi-\psi.
    \end{equation*}
    At time $\tau=0$, there exists a constant $C>0$ such that $\psi_0\ge -Cf_\psi(0)$, therefore by the maximum principle stated in Proposition ]\ref{maximum principle}, we have $\psi\ge-Cf_\psi $. (Recall the evolution equation of $f_\psi$ stated in Proposition \ref{proposition 4.6}: $(\frac{\partial}{\partial\tau}-\Delta_{\omega_\psi,X})f_\psi=-f_\psi$.)

    For (v), since $|\frac{X}{2}\cdot\psi-\psi|\le C$, on the one hand, $\psi\le Cf$, we have $\frac{X}{2}\cdot\psi\le C+Cf$, and hence $f_\psi\le C+Cf+f\le (C+1+\frac{C}{\varepsilon})f$. On the other hand, $\psi\ge-Cf_\psi$, we have that
    \begin{equation*}
        \frac{X}{2}\cdot\psi\ge -C-Cf_\psi.
    \end{equation*}
    Hence $f_\psi\ge -C-Cf_\psi+f$, since $f_\psi$ is also bounded from below by $\varepsilon>0$, there exists a constant $C>0$ such that $f_\psi>\frac{1}{C}f$.
\end{proof}
\subsection{The second order estimates along normalized K\"ahler-Ricci flow}
In this section, we adopt Yau's $C^2$-estimate to prove the uniform equivalence of metrics evolving under the normalized Kähler-Ricci flow.
\begin{prop}\label{metric equivalence}
    There exists a constant $C>1$ such that on $M\times [0, T_{\max})$,
    \begin{equation*}
       \frac{1}{C}\le  g_\psi\le Cg
    \end{equation*}
\end{prop}
Before proving Proposition \ref{metric equivalence}, we will need the evolution equation of $\tr_{\omega}\omega_\psi$ along the normalized K\"ahler-Ricci flow.
\begin{lemma}\label{evolution equation of trace}
Along the normalized K\"ahler-Ricci flow, we have,
    \begin{equation}\label{trace equation}
       \begin{split}
            \left(\frac{\partial}{\partial\tau}-\Delta_{\omega_\psi,X}\right)\tr_\omega \omega_\psi&=\Ric(g)^{i\bar{j}}g_{\psi i\bar{j}}-g_\psi^{p\bar{q}}g_{\psi i\bar{j}}g^{i\bar{s}}g^{m\bar{j}}\Rm(g)_{p\bar{q}m\bar{s}}\\
            &\quad-g_\psi^{p\bar{q}}g^{i\bar{j}}g_\psi^{a\bar{b}}\nabla^g_pg_{\psi i\bar{b}}\nabla^g_{\bar{q}}g_{\psi a\bar{j}}.
       \end{split}
    \end{equation}
\end{lemma}
\begin{proof}
         In the normal holomorphic coordinate of $g$, let's compute the Laplacian of $\tr_\omega \omega_\psi$ with respect to metric $g_\psi$,
         \begin{equation}\label{lapalcian of trace}
             \begin{split}
\Delta_{\omega_\psi}\tr_\omega \omega_\psi&=g_\psi^{p\bar{q}}\partial_p\partial_{\bar{q}}(g^{i\bar{j}}g_{\psi i\bar{j}})\\
             &=g_\psi^{p\bar{q}}g^{i\bar{j}}\partial_p\partial_{\bar{q}}g_{\psi i\bar{j}}+g_\psi^{p\bar{q}}g_{\psi i\bar{j}}\partial_p\partial_{\bar{q}}g^{i\bar{j}}\\
             &=-g_\psi^{p\bar{q}}g^{i\bar{j}}\Rm(g_\psi)_{p\bar{q}i\bar{j}}+g_\psi^{p\bar{q}}g^{i\bar{j}}g_\psi^{a\bar{b}}\partial_pg_{\psi i\bar{b}}\partial_{\bar{q}}g_{\psi a\bar{j}}+g_\psi^{p\bar{q}}g_{\psi i\bar{j}}g^{i\bar{s}}g^{m\bar{j}}\Rm(g)_{p\bar{q}m\bar{s}}\\
             &=-g^{i\bar{j}}\Ric(g_\psi)_{i\bar{j}}+g_\psi^{p\bar{q}}g^{i\bar{j}}g_\psi^{a\bar{b}}\nabla^g_pg_{\psi i\bar{b}}\nabla^g_{\bar{q}}g_{\psi a\bar{j}}+g_\psi^{p\bar{q}}g_{\psi i\bar{j}}g^{i\bar{s}}g^{m\bar{j}}\Rm(g)_{p\bar{q}m\bar{s}}.
             \end{split}
         \end{equation}
         Now we compute $\frac{X}{2}\cdot\tr_\omega \omega_\psi=\mathcal{L}_{\frac{X}{2}}\tr_\omega \omega_\psi$:
       \begin{equation}\label{lie derivative of trace}
           \begin{split}
                          \frac{X}{2}\cdot\tr_\omega \omega_\psi&=\mathcal{L}_{\frac{X}{2}}\tr_\omega \omega_\psi\\
             &=\mathcal{L}_{\frac{X}{2}}(g^{i\bar{j}}g_{\psi i\bar{j}})\\
             &=-(\mathcal{L}_{\frac{X}{2}}g)^{i\bar{j}}g_{\psi i\bar{j}}+g^{i\bar{j}}\mathcal{L}_{\frac{X}{2}}g_{\psi i\bar{j}}\\
             &=-\Ric(g)^{i\bar j}{g_\psi}_{i\bar j}-\tr_\omega \omega_\psi+g^{i\bar{j}}\mathcal{L}_{\frac{X}{2}}g_{\psi i\bar{j}}.
           \end{split}
       \end{equation}
       Here the last line is ensured by the soliton equation \eqref{soliton2}.
       The time derivative of $\tr_\omega \omega_\psi$ is given by:
        \begin{equation}\label{time derivative of trace}
             \frac{\partial}{\partial\tau}\tr_\omega \omega_\psi=\tr_\omega(\mathcal{L}_{\frac{X}{2}}\omega_\psi-\Ric(\omega_\psi)-\omega_\psi).
        \end{equation}
         Combine these three equations \eqref{lapalcian of trace},\eqref{lie derivative of trace} and \eqref{time derivative of trace}, to get
        \begin{equation*}
     \begin{split}
            \left(\frac{\partial}{\partial\tau}-\Delta_{\omega_\psi,X}\right)\tr_\omega \omega_\psi&=\Ric(g)^{i\bar{j}}g_{\psi i\bar{j}}-g_\psi^{p\bar{q}}g_{\psi i\bar{j}}g^{i\bar{s}}g^{m\bar{j}}\Rm(g)_{p\bar{q}m\bar{s}}\\
            &\quad -g_\psi^{p\bar{q}}g^{i\bar{j}}g_\psi^{a\bar{b}}\nabla^g_pg_{\psi i\bar{b}}\nabla^g_{\bar{q}}g_{\psi a\bar{j}}.
     \end{split}
    \end{equation*}
\end{proof}
To get a universal metric equivalence, we perform the Yau's $C^2$ estimate in the parabolic setting.
\begin{lemma}\label{log trace lemma}
    There exists a constant $C>0$ independent of time such that
    \begin{equation}\label{log trace}
         \left(\frac{\partial}{\partial\tau}-\Delta_{\omega_\psi,X}\right)\log \tr_\omega \omega_\psi\le \frac{C}{f_\psi}(\tr_{\omega_\psi} \omega+1).
    \end{equation}
\end{lemma}
\begin{proof}
         Firstly recall that
        \begin{equation*}
            \Delta_{\omega_\psi,X}\log\tr_\omega \omega_\psi=\frac{\Delta_{\omega_\psi,X}\tr_\omega \omega_\psi}{\tr_\omega \omega_\psi}-\frac{|\partial\tr_\omega \omega_\psi|_{g_\psi}^2}{\tr_\omega \omega_\psi^2}.
        \end{equation*}
         By [\eqref{trace equation}, Lemma \ref{evolution equation of trace}], we have
         \begin{equation*}
             \begin{split}
                 \left(\frac{\partial}{\partial\tau}-\Delta_{\omega_\psi,X}\right)\log \tr_\omega \omega_\psi&=\frac{\left(\frac{\partial}{\partial\tau}-\Delta_{\omega_\psi,X}\right)\tr_\omega \omega_\psi}{\tr_\omega \omega_\psi}+\frac{|\partial\tr_\omega \omega_\psi|_{g_\psi}^2}{\tr_\omega \omega_\psi^2}  \\
          &=\frac{1}{\tr_\omega \omega_\psi}(\Ric(g)^{i\bar{j}}g_{\psi i\bar{j}}-g_\psi^{p\bar{q}}g_{\psi i\bar{j}}g^{i\bar{s}}g^{m\bar{j}}\Rm(g)_{p\bar{q}m\bar{s}})\\
          &+\frac{1}{\tr_\omega \omega_\psi}\left(\frac{|\partial\tr_\omega \omega_\psi|^2_{g_\psi}}{\tr_\omega \omega_\psi}-g_\psi^{p\bar{q}}g^{i\bar{j}}g_\psi^{a\bar{b}}\nabla^g_pg_{\psi i\bar{b}}\nabla^g_{\bar{q}}g_{\psi a\bar{j}}\right).
             \end{split}
         \end{equation*}
         We claim that $\frac{|\partial\tr_\omega \omega_\psi|^2_{g_\psi}}{\tr_\omega \omega_\psi}-g_\psi^{p\bar{q}}g^{i\bar{j}}g_\psi^{a\bar{b}}\nabla^g_pg_{\psi i\bar{b}}\nabla^g_{\bar{q}}g_{\psi a\bar{j}}\le 0$. This is a celebrated inequality in Yau's $C^2$ estimate. For detailed proof, see \cite[Proposition 3.2.5]{MR3185331}.
        
         The K\"ahler-Ricci expander $(M,g,X)$ being asymptotically conical, there exists a constant $C(g)>0$ such that $f|\Rm(g)|_g\le C(g)$.  In normal holomorphic coordinates,
         \begin{equation*}
              -g_\psi^{p\bar{q}}g_{\psi i\bar{j}}g^{i\bar{s}}g^{m\bar{j}}\Rm(g)_{p\bar{q}m\bar{s}}=\sum_{p,i}-g_\psi^{p\bar{p}}g_{\psi i\bar{i}}\Rm(g)_{p\bar{p}i\bar{i}}\le \frac{C(g)}{f}\tr_\omega \omega_\psi\tr_{\omega_\psi}\omega
         \end{equation*}
         Finally we get 
         \begin{equation*}
             \left(\frac{\partial}{\partial\tau}-\Delta_{\omega_\psi,X}\right)\log \tr_\omega \omega_\psi\le \frac{C(g)}{f}(\tr_{\omega_\psi}\omega+1).
         \end{equation*}
         Since $f$ and $f_\psi$ are comparable by Corollary \ref{good cosntant for psi and psi dot}, \eqref{log trace} holds as required.
     \end{proof}
     \begin{remark}
         If $(M,g,X)$ admits positive holomorphic bisectional curvature, then we can reduce \eqref{log trace} to
         \begin{equation*}
             \left(\frac{\partial}{\partial\tau}-\Delta_{\omega_\psi,X}\right)\log\tr_{\omega}\omega_\psi\le \frac{C(g)}{f},
         \end{equation*}
         by some constant that only depends on $g$.
         
         In 2016, Chodosh and Fong \cite{MR3466851} have proved that an asymptotically conical gradient K\"ahler-Ricci expander with positive holomorphic bisectional curvature must be isometric to one of the $U(n)-$
rotationally symmetric expanding gradient solitons on $C^n$, as constructed by Cao \cite{MR1449972}.
     \end{remark}
     With this key lemma, we can now prove Proposition \ref{metric equivalence}.
\begin{proof}[Proof of Proposition \ref{metric equivalence}]
    We choose to use $\frac{\psi}{f_\psi}$ as a barrier function to control $\log\tr_\omega \omega_\psi$.

First we compute the evolution equation of $\frac{\psi}{f_\psi}$
    \begin{equation*}
        \begin{split}
            \left(\frac{\partial}{\partial\tau}-\Delta_{\omega_\psi,X}\right)\frac{\psi}{f_\psi}&=-\frac{\Delta_{\omega_\psi}\psi}{f_\psi}+\frac{\dot\psi}{f_\psi}-\frac{\frac{X}{2}\cdot\psi}{f_\psi}+\psi\left(\frac{\partial}{\partial\tau}-\Delta_{\omega_\psi,X}\right)\frac{1}{f_\psi}\\
            &\quad -2\Re\left(<\partial \psi,\bar\partial\frac{1}{f_\psi}>_{g_\psi}\right)\\
            &=-\frac{\Delta_{\omega_\psi}\psi}{f_\psi}+\frac{\dot\psi}{f_\psi}-\frac{\frac{X}{2}\cdot\psi}{f_\psi}+\psi\left(\frac{1}{f_\psi}-\frac{|X|^2_{g_\psi}}{f_\psi^3}\right)+\frac{X\cdot\psi}{f_\psi^2}\\
            &=-\frac{\Delta_{\omega_\psi}\psi}{f_\psi}+\frac{\dot\psi}{f_\psi}+\frac{\psi-\frac{X}{2}\cdot\psi}{f_\psi}-\psi\frac{|X|^2_{g_\psi}}{f_\psi^3}+\frac{X\cdot\psi}{f_\psi^2}.
        \end{split}
    \end{equation*}
    By Corollary \ref{good cosntant for psi and psi dot}, there exists a constant $C_1>0$ independent of time such that
    \begin{equation*}
         \left(\frac{\partial}{\partial\tau}-\Delta_{\omega_\psi,X}\right)\frac{\psi}{f_\psi}\ge -\frac{\Delta_{\omega_\psi}\psi}{f_\psi}-C_1\frac{1}{f_\psi}=\frac{\tr_{\omega_\psi }\omega}{f_\psi}-\frac{C_1+n}{f_\psi}.
    \end{equation*}
    Due to Lemma \ref{log trace lemma}, there exists a constant $C>0$ such that
    \begin{equation*}
         \left(\frac{\partial}{\partial\tau}-\Delta_{\omega_\psi,X}\right)\log \tr_\omega \omega_\psi\le \frac{C}{f_\psi}(\tr_{\omega_\psi} \omega+1).
    \end{equation*}
    Hence we compute the evolution equation of $\log\tr_\omega \omega_\psi-C\frac{\psi}{f_\psi}$ as follows:
    \begin{equation*}
        \left(\frac{\partial}{\partial\tau}-\Delta_{\omega_\psi,X}\right)\left(\log\tr_\omega \omega_\psi-C\frac{\psi}{f_\psi}\right)\le \frac{C}{f_\psi}+C\frac{C_1+n}{f_\psi}:=\frac{C_3}{f_\psi}.
    \end{equation*}
    Thanks to Proposition \ref{rough metric equivalence}, $\log\tr_\omega \omega_\psi$ is bounded on $M\times [0,T]$ for any $0<T<T_{\max}$. By applying Lemma \ref{useful lemma}, there exists a constant $C_4>0$ such that for all $\tau\in [0,T_{\max})$,
    \begin{equation*}
       \sup_M \left(\log\tr_\omega \omega_\psi-C\frac{\psi}{f_\psi}\right)(\tau)\le \sup_M\left(\log\tr_\omega \omega_{\psi_0}-C\frac{\psi_0}{f_\psi(0)}\right)+C_4<\infty.
    \end{equation*}
    Since $\frac{\psi}{f_\psi}$ and $\frac{1}{f_\psi+K}$ are bounded due to Corollary \ref{good cosntant for psi and psi dot}, we thus find a constant $C'>0$ such that on $M\times [0,T_{\max})$,
    \begin{equation*}
        \tr_\omega \omega_\psi\le C'.
    \end{equation*}
    Moreover, $\frac{\omega_\psi^n}{\omega^n}$ being globally bounded from below on $M\times [0,T_{\max})$ due to Corollary \ref{good cosntant for psi and psi dot}. Thanks to the elementary inequality:
    \begin{equation*}
        \frac{\omega_\psi^n}{\omega^n}\tr_{\omega_\psi}\omega\le (\tr_\omega\omega_\psi)^{n-1},
    \end{equation*}
     we can find another constant $C''>0$ such that on $M\times [0,T_{\max})$
    \begin{equation*}
        \tr_{\omega_\psi} \omega\le C''.
    \end{equation*}
    Take $C=\max\{C',C''\}$, we get the desired result.
\end{proof}
\subsection{The higher order estimates along normalized K\"ahler-Ricci flow}
In this section we omit all the time dependence of $g_\psi(\tau)$, we denote it simply by $g_\psi$. Moreover, we use $\nabla$(resp. $\overline{\nabla}$) to denote the holomorphic (resp. anti-holomorphic) part of $\nabla^{g_\psi}$, we also use $|\cdot|$ to denote $|\cdot|_{g_\psi}$ and $<>$ to denote the scalar product of $g_\psi$.

Define a tensor $\Psi$ by 
\begin{equation*}
    \Psi_{ij}^k=\Gamma(g_\psi)_{ij}^k-\Gamma(g)_{ij}^k=g_\psi^{k\bar l}\nabla^g_ig_{\psi j\bar l}.
\end{equation*}
Define a smooth function $S$ by
\begin{equation*}
    S:=|\Psi|^2=g_\psi^{i\bar j}g_\psi^{p\bar q}g_{\psi k\bar l}\Psi_{ip}^k\overline{\Psi_{jq}^l}.
\end{equation*}
\begin{prop}[Modified Phong-\v Se\v sum-Sturm equality \cite{MR2379807}]\label{pss lemma}
Along the normalized K\"ahler-Ricci flow, $S$ evolves by
    \begin{equation}\label{Modified PSS eq}
       \begin{split}
            \left(\frac{\partial}{\partial\tau}-\Delta_{\omega_\psi,X}\right)S&=-|\overline{\nabla}\Psi|^2-|\nabla\Psi|^2+S\\
            &\quad +2\Re\left({g_\psi^{i\bar j}g_\psi^{p\bar q}g_{\psi k\bar l}\left(\nabla^g\Ric(g)_{ip}^k-\nabla^{\bar b}\Rm(g)_{i\bar b p}^k\right)}\overline{\Psi_{jq}^l}\right),
       \end{split}
    \end{equation}
    where $\nabla^{\bar b}=g_\psi^{a\bar b}\nabla_a$, $\Rm(g)_{i\bar b p}^k:=g^{m\bar s}\Rm(g)_{i\bar b p \bar s}$ and $\nabla^g\Ric(g)_{ip}^k:=g^{k\bar s}\nabla^g_{\bar s}\Ric(g)_{ip}$.
\end{prop}
\begin{proof}
    Compute
    \begin{equation}\label{full laplacian}
        \Delta_{\omega_\psi}S=g_\psi^{i\bar j}g_\psi^{p\bar q}g_{\psi k\bar l}\left(\Delta_{\omega_\psi}\Psi_{ip}^k\overline{\Psi_{jq}^l}+\Psi_{ip}^l\overline{\overline{\Delta}_{\omega_\psi}\Psi_{jq}^l}\right)+|\overline{\nabla}\Psi|^2+|\nabla\Psi|^2,
    \end{equation}
    where we are writing $\Delta_{\omega_\psi}=g_\psi^{a\bar b}\nabla_a\nabla_{\bar b}$ for the \emph{rough} Laplacian and $\overline{\Delta}_{\omega_\psi}=g_\psi^{\bar b a}\nabla_{\bar b}\nabla_a$ for its conjugate. In particular, using the commutation formulae 
    \begin{equation}\label{commutation}
\overline{\Delta}_{\omega_\psi}\Psi_{jq}^l=\Delta_{\omega_\psi}\Psi_{jq}^l+\Ric(g_\psi)_j^b\Psi_{bq}^l+\Ric(g(\psi))_q^b\Psi_{jb}^l-\Ric(g_\psi)^l_b\Psi_{jq}^b.
        \end{equation}
        Combining \eqref{full laplacian} and \eqref{commutation},
        \begin{equation}\label{use laplacian of S}
        \begin{split}
                \Delta_{\omega_\psi}S&=2\Re(g_\psi^{i\bar j}g_\psi^{p\bar q}g_{\psi k\bar l}\Delta_{\omega_\psi}\Psi_{ip}^k\overline{\Psi_{jq}^l})+|\overline{\nabla}\Psi|^2+|\nabla\Psi|^2\\
                &\quad +\Ric(g_\psi)^{i\bar j}g_\psi^{p\bar q}g_{\psi k\bar l}\Psi_{ip}^k\overline{\Psi_{jq}^l}+g_\psi^{i\bar j}\Ric(g_\psi)^{p\bar q}g_{\psi k\bar l}\Psi_{ip}^k\overline{\Psi_{jq}^l}-g_\psi^{i\bar j}g_\psi^{p\bar q}\Ric(g_\psi)_{k\bar l}\Psi_{ip}^k\overline{\Psi_{jq}^l}.
        \end{split}
        \end{equation}
        Now we compute
        \begin{equation}\label{Lie derivative of S}
            \mathcal{L}_{\frac{X}{2}}S=\frac{X}{2}\cdot S=2\Re(g_\psi^{i\bar j}g_\psi^{p\bar q}g_{\psi k\bar l}\nabla_{\frac{X}{2}}\Psi_{ip}^k\overline{\Psi_{jq}^l}).
        \end{equation}
        Notice that 
        \begin{equation}\label{different of lie and covariant}
            \nabla_{\frac{X}{2}}\Psi_{ip}^k=\mathcal{L}_{\frac{X}{2}}\Psi_{ip}^k-\frac{1}{2}(\mathcal{L}_{\frac{X}{2}}g_\psi)^a_i\Psi_{ap}^k-\frac{1}{2}(\mathcal{L}_{\frac{X}{2}}g_\psi)^a_p\Psi_{ia}^k+\frac{1}{2}(\mathcal{L}_{\frac{X}{2}}g_\psi)^k_a\Psi_{ip}^a.
        \end{equation}
        Combining \eqref{Lie derivative of S} and \eqref{different of lie and covariant},
        \begin{equation}\label{full lie of S}
          \begin{split}
                \frac{X}{2}\cdot S&=2\Re(g_\psi^{i\bar j}g_\psi^{p\bar q}g_{\psi k\bar l}\mathcal{L}_{\frac{X}{2}}\Psi_{ip}^k\overline{\Psi_{jq}^l})\\
            &\quad -(\mathcal{L}_{\frac{X}{2}}g_\psi)^{i\bar j}g_\psi^{p\bar q}g_{\psi k\bar l}\Psi_{ip}^k\overline{\Psi_{jq}^l}-g_\psi^{i\bar j}(\mathcal{L}_{\frac{X}{2}}g_\psi)^{p\bar q}g_{\psi k\bar l}\Psi_{ip}^k\overline{\Psi_{jq}^l}+g_\psi^{i\bar j}g_\psi^{p\bar q}(\mathcal{L}_{\frac{X}{2}}g_\psi)_{k\bar l}\Psi_{ip}^k\overline{\Psi_{jq}^l}.
          \end{split}
        \end{equation}
        We now compute the time derivative of $S$. We claim that
        \begin{equation}\label{Time derivative of Psi}
            \frac{\partial}{\partial\tau}\Psi_{ip}^k=\Delta_{\omega_\psi}\Psi_{ip}^k+\mathcal{L}_{\frac{X}{2}}\Psi_{ip}^k-\nabla^{\bar b}\Rm(g)_{i\bar b p}^k+g^{k\bar s}\nabla^g_{\bar s}\Ric(g)_{ip}.
        \end{equation}
        Given this, together with 
        \begin{equation*}
            \frac{\partial}{\partial\tau}g_\psi=\mathcal{L}_{\frac{X}{2}}g_\psi-\Ric(g_\psi)-g_\psi,
        \end{equation*}
        we obtain
        \begin{equation}\label{full time derivative of S}
            \begin{split}
                \frac{\partial}{\partial\tau}S&=\Ric(g_\psi)^{i\bar j}g_\psi^{p\bar q}g_{\psi k\bar l}\Psi_{ip}^k\overline{\Psi_{jq}^l}+g_\psi^{i\bar j}\Ric(g_\psi)^{p\bar q}g_{\psi k\bar l}\Psi_{ip}^k\overline{\Psi_{jq}^l}-g_\psi^{i\bar j}g_\psi^{p\bar q}\Ric(g_\psi)_{k\bar l}\Psi_{ip}^k\overline{\Psi_{jq}^l}\\
                &\quad -(\mathcal{L}_{\frac{X}{2}}g_\psi)^{i\bar j}g_\psi^{p\bar q}g_{\psi k\bar l}\Psi_{ip}^k\overline{\Psi_{jq}^l}-g_\psi^{i\bar j}(\mathcal{L}_{\frac{X}{2}}g_\psi)^{p\bar q}g_{\psi k\bar l}\Psi_{ip}^k\overline{\Psi_{jq}^l}+g_\psi^{i\bar j}g_\psi^{p\bar q}(\mathcal{L}_{\frac{X}{2}}g_\psi)_{k\bar l}\Psi_{ip}^k\overline{\Psi_{jq}^l}\\
                &\quad +S+2\Re\left(g_\psi^{i\bar j}g_\psi^{p\bar q}g_{\psi k\bar l}(\Delta_{\omega_\psi}\Psi_{ip}^k+\mathcal{L}_{\frac{X}{2}}\Psi_{ip}^k)\overline{\Psi_{jq}^l}\right)\\
                &\quad +2\Re\left({g_\psi^{i\bar j}g_\psi^{p\bar q}g_{\psi k\bar l}\left(\nabla^g\Ric(g)_{ip}^k-\nabla^{\bar b}\Rm(g)_{i\bar b p}^k\right)}\overline{\Psi_{jq}^l}\right).
            \end{split}
        \end{equation}
        Then \eqref{Modified PSS eq} follows from \eqref{use laplacian of S}, \eqref{full lie of S} and \eqref{full time derivative of S}.

        To establish \eqref{Time derivative of Psi}, compute
        \begin{equation}\label{baby time derivative of psi}
            \frac{\partial}{\partial\tau}\Psi_{ip}^k=\frac{\partial}{\partial\tau}\Gamma(g_\psi)_{ip}^k=\nabla_i(\mathcal{L}_{\frac{X}{2}}g_\psi-\Ric(g_\psi))^k
_p.        \end{equation}
On the other hand 
\begin{equation}\label{lie of psi}
    \begin{split}
        &\nabla_{\bar b}\Psi_{ip}^k=\partial_{\bar b}(\Gamma(g_\psi)_{ip}^k-\Gamma(g)_{ip}^k)=\Rm(g)_{i\bar b p}^k-\Rm(g_\psi)_{i\bar b p}^k,\\
        &\mathcal{L}_{\frac{X}{2}}\Psi_{ip}^k=\nabla_i(\mathcal{L}_{\frac{X}{2}}g_\psi)^k_p-\nabla^g_i(\mathcal{L}_{\frac{X}{2}}g)^k_p=\nabla_i(\mathcal{L}_{\frac{X}{2}}g_\psi)^k_p-\nabla^g\Ric(g)_{ip}^k,
    \end{split}
\end{equation}
and hence 
\begin{equation}\label{lap of psi}
    \Delta_{\omega_\psi}\Psi_{ip}^k=g_\psi^{a\bar b}\nabla_a\nabla_{\bar b}\Psi_{ip}^k=\nabla^{\bar b}\Rm(g)_{i\bar b p}^k-\nabla_i\Ric(g_\psi)_p^k,
\end{equation}
where for the last equality we have used the second Bianchi identity. Then \eqref{Time derivative of Psi} follows from \eqref{baby time derivative of psi}, \eqref{lie of psi} and \eqref{lap of psi}.
\end{proof}
With this evolution equation of $S$, we can now give a uniform estimate of $S$ along the normalized K\"ahler-Ricci flow.
\begin{prop}\label{bound of S}
    There exists a constant $C>0$ such that on $M\times[0,T_{\max})$, we have
    \begin{equation*}
        S\le C\frac{1}{f_\psi}.
    \end{equation*}
\end{prop}
Before proving Proposition \ref{bound of S}, we need to prove the following lemma. 
\begin{lemma}\label{lemma 5.3}
    There exists a constant $C_1>0$ such that on $M\times [0, T_{\max})$, we have,
    \begin{equation}\label{equation of S}
        \left(\frac{\partial}{\partial\tau}-\Delta_{\omega_\psi,X}\right)(f_\psi S)\le C_1S-\frac{1}{2}f_\psi(|\nabla\Psi|^2+|\overline{\nabla}\Psi|^2)+\frac{C_1}{f_\psi}.
    \end{equation}
\end{lemma}
\begin{proof}
    We compute
    \begin{equation*}
        \begin{split}
            \left(\frac{\partial}{\partial\tau}-\Delta_{\omega_\psi,X}\right)(f_\psi S)&=f_\psi \left(\frac{\partial}{\partial\tau}-\Delta_{\omega_\psi,X}\right)S+S \left(\frac{\partial}{\partial\tau}-\Delta_{\omega_\psi,X}\right)f_\psi\\
            &\quad -2\Re(<\partial f_\psi,\bar\partial S>)\\
            &=f_\psi \left(\frac{\partial}{\partial\tau}-\Delta_{\omega_\psi,X}\right)S-f_\psi S-X\cdot S.
        \end{split}
    \end{equation*}
  Notice that $\nabla\Rm(g)=\nabla g*\Rm(g)+\nabla^g\Rm(g)$, then by [\eqref{Modified PSS eq}, Proposition \ref{pss lemma}], there exits a constant $C_2>0$ such that
    \begin{equation*}
         \left(\frac{\partial}{\partial\tau}-\Delta_{\omega_\psi,X}\right)S\le -|\overline{\nabla}\Psi|^2-|\nabla\Psi|^2+S+C_2(\frac{S}{f_\psi}+\frac{\sqrt{S}}{f_\psi^{\frac{3}{2}}}).
    \end{equation*}
    Thus,
    \begin{equation}\label{lemma for nest lemma}
         \left(\frac{\partial}{\partial\tau}-\Delta_{\omega_\psi,X}\right)(f_\psi S)\le -f_\psi|\overline{\nabla}\Psi|^2-f_\psi|\nabla\Psi|^2+C_2(S+\frac{\sqrt{S}}{f_\psi^{\frac{1}{2}}})-X\cdot S.
    \end{equation}
    We notice that by Cauchy-Schwarz inequality, for any $\sigma>0$
    \begin{equation*}
        |X\cdot S|\le 2|X||\nabla^{g_\psi}\Psi||\Psi|\le \sigma|X|^2|\nabla^{g_\psi}\Psi|^2+\frac{1}{\sigma}S.
    \end{equation*}
    Take $\sigma<1$ such that $\sigma|X|^2_{g_\psi}\le \frac{1}{2}f_\psi$ thanks to Corollary \ref{bound of X}, hence we get
    \begin{equation*}
        \left(\frac{\partial}{\partial\tau}-\Delta_{\omega_\psi,X}\right)(f_\psi S)\le -\frac{1}{2}f_\psi|\overline{\nabla}\Psi|^2-\frac{1}{2}f_\psi|\nabla\Psi|^2+C_2(S+\frac{\sqrt{S}}{f_\psi^{\frac{1}{2}}})+\frac{1}{\sigma}S.    \end{equation*}
        Using Cauchy-Schwarz inequality $\frac{\sqrt
        S}{f_\psi^{\frac{1}{2}}}\le S+\frac{1}{f_\psi} $, there exists a constant $C_1>0$ such that
        \begin{equation*}
            \left(\frac{\partial}{\partial\tau}-\Delta_{\omega_\psi,X}\right)(f_\psi S)\le -\frac{1}{2}f_\psi|\overline{\nabla}\Psi|^2-\frac{1}{2}f_\psi|\nabla\Psi|^2+C_1(S+\frac{1}{f_\psi}),
        \end{equation*}
        as required.
\end{proof}
\begin{proof}[Proof of Proposition \ref{bound of S}]
   By [\eqref{trace equation}, Lemma \ref{evolution equation of trace}], there exist constants $C_2,\alpha>0$ such that
             \begin{equation*}
        \left(\frac{\partial}{\partial\tau}-\Delta_{\omega_\psi,X}\right)\tr_\omega \omega_\psi\le \frac{C_2}{f_\psi}-\alpha S.
    \end{equation*}
Together with [\eqref{equation of S}, Lemma \ref{lemma 5.3}], we get
\begin{equation*}
     \left(\frac{\partial}{\partial\tau}-\Delta_{\omega_\psi,X}\right)\left(f_\psi S+\frac{C_1}{\alpha}\tr_\omega \omega_\psi\right)\le C_1(\frac{C_2}{\alpha}+1)\frac{1}{f_\psi}.
\end{equation*}
Thanks to Proposition \ref{rough c3 estimate}, we have a rough estimate that allows us to apply Lemma \ref{useful lemma}: there exists a constant $C_3>0$ such that for all $\tau\in [0,T_{\max})$,
\begin{equation*}
    \sup_{M}\left(f_\psi S+\frac{C_1}{\alpha}\tr_\omega \omega_\psi\right)\le \sup_M\left(f_\psi S+\frac{C_1}{\alpha}\tr_\omega \omega_\psi\right)\bigg|_{\tau=0}+C_3.
\end{equation*}
Since $\psi(0)$ satisfies Condition I, $\sup_M\left(f_\psi S+\frac{C_1}{\alpha}\tr_\omega \omega_\psi\right)|_{\tau=0}$ is finite, and because $\tr_\omega \omega_\psi$ is positive, we can find a constant $C>0$ such that on $M\times [0,T_{\max})$, $S\le \frac{C}{f_\psi}$ holds.
\end{proof}
Building upon the metric equivalence established in Proposition \ref{metric equivalence} and the $C^3$-estimates obtained in Proposition \ref{bound of S}, we now derive estimates for the Riemann curvature tensor under the normalized Kähler-Ricci flow.
\begin{prop}\label{evolution function of Rm}
    Along the normalized K\"ahler-Ricci flow, the curvature tensor evolves by
    \begin{equation}\label{evolution equation of Rm}
        \begin{split}
            \frac{\partial}{\partial\tau}\Rm(g_\psi)=\frac{1}{2}\Delta_{g_\psi}\Rm(g_\psi)+\mathcal{L}_{\frac{X}{2}}\Rm(g_\psi)-\Rm(g_\psi)+\Rm(g_\psi)*\Rm(g_\psi).
        \end{split}
    \end{equation}
\end{prop}
\begin{proof}
    Recall that
    \begin{equation*}
        \begin{split}
            \frac{\partial}{\partial\tau}\Rm(g_\psi)_{i\bar jk\bar l}&=\frac{\partial}{\partial\tau}(-g_{\psi p\bar j}\partial_{\bar l}\Gamma(g_\psi)_{ik}^p)\\
            &=-(\mathcal{L}_{\frac{X}{2}}g_{\psi}-\Ric(g_\psi)-g_\psi)_{p\bar j}\partial_{\bar l}\Gamma(g_\psi)_{ik}^p+\nabla_{\bar l}\nabla_k\Ric(g_\psi)_{i\bar j}-\nabla_{\bar l}\nabla_k(\mathcal{L}_{\frac{X}{2}}g_\psi)_{i\bar j}\\
            &=\mathcal{L}_{\frac{X}{2}}\Rm(g_\psi)_{i\bar jk\bar l}-\Rm(g_\psi)_{i\bar jk\bar l}-\Ric(g_\psi)_{\bar j}^{\bar a}\Rm(g_\psi)_{i\bar ak\bar l}+\nabla_{\bar l}\nabla_k\Ric(g_\psi)_{i\bar j}\\
        \end{split}
    \end{equation*}
    Using the Bianchi identity and the commutation formulae, we obtain
    \begin{equation*}
       \begin{split}
            \Delta_{\omega_\psi}\Rm(g_\psi)_{i\bar jk\bar l}&=g_\psi^{a\bar b}\nabla_a\nabla_{\bar b}\Rm(g_\psi)_{i\bar jk\bar l}\\
            &=g_\psi^{a\bar b}\nabla_a\nabla_{\bar l}\Rm(g_\psi)_{i\bar jk\bar b}\\
            &=g_\psi^{a\bar b}\nabla_{\bar l}\nabla_{a}\Rm(g_\psi)_{i\bar jk\bar b}+g_\psi^{a\bar b}[\nabla_a,\nabla_{\bar l}]\Rm(g_\psi)_{i\bar jk\bar b}\\
            &=\nabla_{\bar l}\nabla_k\Ric(g_\psi)_{i\bar j}+\Rm(g_\psi)*\Rm(g_\psi).
       \end{split}
    \end{equation*}
    By conjugate formulae of $\overline{\Delta}_{\omega_\psi}$ and the fact $\Delta_{g_\psi}=\Delta_{\omega_\psi}+\overline{\Delta}_{\omega_\psi}$, we conclude that
    \begin{equation*}
        \frac{\partial}{\partial\tau}\Rm(g_\psi)=\frac{1}{2}\Delta_{g_\psi}\Rm(g_\psi)+\mathcal{L}_{\frac{X}{2}}\Rm(g_\psi)-\Rm(g_\psi)+\Rm(g_\psi)*\Rm(g_\psi).
    \end{equation*}
\end{proof}
\begin{corollary}\label{rm corollary}
    There exists a dimensional constant $C(n)>0$ such that
    \begin{equation}\label{evolution equation of Rm square}
       \begin{split}
            \left(\frac{\partial}{\partial\tau}-\Delta_{\omega_\psi,X}\right)|\Rm(g_\psi)|^2&\le -|\nabla\Rm(g_\psi)|^2-|\overline{\nabla}\Rm(g_\psi)|^2\\
            &\quad +C(n)|\Rm(g_\psi)|^3+2|\Rm(g_\psi)|^2,
       \end{split}
    \end{equation}
    and for points $(x,\tau)\in M\times [0,T_{\max})$ such that $|\Rm(g_\psi)|(x,\tau)>0$,
    \begin{equation}\label{evolution equation of Rm1}
        \left(\frac{\partial}{\partial\tau}-\Delta_{\omega_\psi,X}\right)|\Rm(g_\psi)|(x,\tau)\le \frac{C(n)}{2}|\Rm(g_\psi)|^2(x,\tau)+|\Rm(g_\psi)|(x,\tau).
    \end{equation}
\end{corollary}
\begin{proof}
    First, we prove that \eqref{evolution equation of Rm square} implies \eqref{evolution equation of Rm1}. Note that
    \begin{equation}\label{4,22 1}
        \begin{split}
            \left(\frac{\partial}{\partial\tau}-\Delta_{\omega_\psi,X}\right)|\Rm(g_\psi)|&=\frac{1}{2|\Rm(g_\psi)|^2}\left(\frac{\partial}{\partial\tau}-\Delta_{\omega_\psi,X}\right) |\Rm(g_\psi)|^2\\&\quad +\frac{1}{4|\Rm(g_\psi)|^3}\left|\nabla|\Rm(g_\psi)|^2\right|^2,
        \end{split}
    \end{equation}
    and
    \begin{equation}\label{4.22 2}
        \left|\nabla|\Rm(g_\psi)|^2\right|^2\le 2|\Rm(g_\psi)|^2\left(|\nabla\Rm(g_\psi)|^2+|\overline{\nabla}\Rm(g_\psi)|^2\right).
    \end{equation}
    Then \eqref{evolution equation of Rm1} comes from \eqref{4,22 1} and \eqref{4.22 2}.

    Now we prove \eqref{evolution equation of Rm square}. In local holomorphic coordinates, we compute
    \begin{equation*}
        \begin{split}
            \frac{\partial}{\partial\tau}|\Rm(g_\psi)|^2&= \frac{\partial}{\partial\tau}\left(g_\psi^{i\bar p}g_\psi^{q\bar j}g_\psi^{k\bar r}g_\psi^{s\bar l}\Rm(g_\psi)_{i\bar j k\bar l}\Rm(g_\psi)_{\bar p q \bar r s}\right)\\
            &=2\Re\left(<\frac{1}{2}\Delta_{g_\psi}\Rm(g_\psi)+\mathcal{L}_{\frac{X}{2}}\Rm(g_\psi),\Rm(g_\psi)>\right)\\
            &\quad+2\Re\left(<-\Rm(g_\psi)+\Rm(g_\psi)*\Rm(g_\psi),\Rm(g_\psi)>\right)\\
            &\quad +\frac{\partial}{\partial\tau}g_\psi^{i\bar p}g_\psi^{q\bar j}g_\psi^{k\bar r}g_\psi^{s\bar l}\Rm(g_\psi)_{i\bar j k\bar l}\Rm(g_\psi)_{\bar p q \bar r s}\\
            &\quad+g_\psi^{i\bar p}\frac{\partial}{\partial\tau}g_\psi^{q\bar j}g_\psi^{k\bar r}g_\psi^{s\bar l}\Rm(g_\psi)_{i\bar j k\bar l}\Rm(g_\psi)_{\bar p q \bar r s}\\
             &\quad +g_\psi^{i\bar p}g_\psi^{q\bar j}\frac{\partial}{\partial\tau}g_\psi^{k\bar r}g_\psi^{s\bar l}\Rm(g_\psi)_{i\bar j k\bar l}\Rm(g_\psi)_{\bar p q \bar r s}\\
             &\quad+g_\psi^{i\bar p}g_\psi^{q\bar j}g_\psi^{k\bar r}\frac{\partial}{\partial\tau}g_\psi^{s\bar l}\Rm(g_\psi)_{i\bar j k\bar l}\Rm(g_\psi)_{\bar p q \bar r s}.\\
        \end{split}
    \end{equation*}
    Moreover,
    \begin{equation*}
       \begin{split}
            \Delta_{\omega_\psi,X}|\Rm(g_\psi)|^2&=2\Re\left(<\frac{1}{2}\Delta_{g_\psi}\Rm(g_\psi)+\mathcal{L}_{\frac{X}{2}}\Rm(g_\psi),\Rm(g_\psi)>\right)\\
            &\quad +|\nabla\Rm(g_\psi)|^2+|\overline{\nabla}\Rm(g_\psi)|^2\\
             &\quad-\mathcal{L}_{\frac{X}{2}}g_\psi^{i\bar p}g_\psi^{q\bar j}g_\psi^{k\bar r}g_\psi^{s\bar l}\Rm(g_\psi)_{i\bar j k\bar l}\Rm(g_\psi)_{\bar p q \bar r s}\\
             &\quad-g_\psi^{i\bar p}\mathcal{L}_{\frac{X}{2}}g_\psi^{q\bar j}g_\psi^{k\bar r}g_\psi^{s\bar l}\Rm(g_\psi)_{i\bar j k\bar l}\Rm(g_\psi)_{\bar p q \bar r s}\\
             &\quad -g_\psi^{i\bar p}g_\psi^{q\bar j}\mathcal{L}_{\frac{X}{2}}g_\psi^{k\bar r}g_\psi^{s\bar l}\Rm(g_\psi)_{i\bar j k\bar l}\Rm(g_\psi)_{\bar p q \bar r s}\\
             &\quad -g_\psi^{i\bar p}g_\psi^{q\bar j}g_\psi^{k\bar r}\mathcal{L}_{\frac{X}{2}}g_\psi^{s\bar l}\Rm(g_\psi)_{i\bar j k\bar l}\Rm(g_\psi)_{\bar p q \bar r s}.\\
       \end{split}
    \end{equation*}
    Since \begin{equation*}
        \frac{\partial}{\partial\tau}g_\psi^{i\bar p}=-\mathcal{L}_{\frac{X}{2}}g_\psi^{i\bar p}+\Ric(g_\psi)^{i\bar q}+g_\psi^{i\bar p}.
    \end{equation*}
    Then we get
    \begin{equation*}
         \begin{split}
             \left(\frac{\partial}{\partial\tau}-\Delta_{\omega_\psi,X}\right)|\Rm(g_\psi)|^2&=2|\Rm(g_\psi)|^2+\Ric(g_\psi)*\Rm(g_\psi)*\Rm(g_\psi)\\
             &\quad+2\Re\left(<\Rm(g_\psi)*\Rm(g_\psi),\Rm(g_\psi)>\right).
         \end{split}
    \end{equation*}
    Thus \eqref{evolution equation of Rm square} holds as desired.
\end{proof}
The following corollary tells us the Riemann curvature is uniformly bounded along the normalized K\"ahler-Ricci flow.
\begin{corollary}\label{condition I bound of Rm}
    There exists a constant $C>0$ such that on $M\times [0,T_{\max})$
    \begin{equation*}
        |\Rm(g_\psi)|_{g_\psi}\le C.
    \end{equation*}
\end{corollary}
\begin{proof}
    Notice that \begin{equation*}
        \nabla_{\bar b}\Psi_{ip}^k=\partial_{\bar b}(\Gamma(g_\psi)_{ip}^k-\Gamma(g)_{ip}^k)=\Rm(g)_{i\bar b p}^k-\Rm(g_\psi)_{i\bar b p}^k.
    \end{equation*}
    Since $g_\psi$ and $g$ are uniformly comparable on $M\times[0,T_{\max})$ by Proposition \ref{metric equivalence}, there exists a constant $C_1>0$ such that 
    \begin{equation*}
        |\overline{\nabla}\Psi|^2\ge \frac{1}{2}|\Rm(g_\psi)|^2-C_1|\Rm(g)|^2.
    \end{equation*}
   By Proposition \ref{good cosntant for psi and psi dot}, there exists $C_2>0$ such that $|\Rm(g)|\le \frac{C_2}{f_\psi}$. As a consequence, we conclude that there exists a constant $C_3>0$ such that
   \begin{equation*}
       |\overline{\nabla}\Psi|^2\ge \frac{1}{2}|\Rm(g_\psi)|^2-\frac{C_3}{f_\psi^2}.
   \end{equation*}
   Recall [\eqref{Modified PSS eq}, Proposition \ref{pss lemma}] and Proposition \ref{bound of S} and the fact that $f_\psi$ is bounded from below, there exists a constant $C_4>0$ such that
   \begin{equation*}
       \left(\frac{\partial}{\partial\tau}-\Delta_{\omega_\psi,X}\right)S\le -\frac{1}{2}|\Rm(g_\psi)|^2+C_4.
   \end{equation*}
   Recall [\eqref{evolution equation of Rm1}, Corollary \ref{rm corollary}],
    \begin{equation*}
        \left(\frac{\partial}{\partial\tau}-\Delta_{\omega_\psi,X}\right)|\Rm(g_\psi)|\le \frac{C(n)}{2}|\Rm(g_\psi)|^2+|\Rm(g_\psi)|.
    \end{equation*}
    Combining these inequalities, we claim that there exists a constant $C_5>0$ such that
    \begin{equation*}
         \left(\frac{\partial}{\partial\tau}-\Delta_{\omega_\psi,X}\right)\left(|\Rm(g_\psi)|+(C(n)+1)S\right)\le -\frac{1}{2}|\Rm(g_\psi)|^2+|\Rm(g_\psi)|+C_5.
    \end{equation*}
    Recall the evolution equation of $e^\tau f_\psi$: $ \left(\frac{\partial}{\partial\tau}-\Delta_{\omega_\psi,X}\right)e^\tau f_\psi=0$. Thus, for all $\eta>0$, we have
    \begin{equation*}
         \left(\frac{\partial}{\partial\tau}-\Delta_{\omega_\psi,X}\right)\left(|\Rm(g_\psi)|+(C(n)+1)S-\eta e^\tau f_\psi\right)\le -\frac{1}{2}|\Rm(g_\psi)|^2+|\Rm(g_\psi)|+C_5.
    \end{equation*}
For all $0<T<T_{\max}$, $|\Rm(g_\psi)|+(C(n)+1)S$ is bounded on $M\times [0,T]$. Hence, the function $|\Rm(g_\psi)|+(C(n)+1)S-\eta e^\tau f_\psi$ tends to $-\infty$ at spatial infinity uniformly. This implies the existence of maximum point of $|\Rm(g_\psi)|+(C(n)+1)S-\eta e^\tau f_\psi$. We denote it by $(x_0,\tau_0)$. If $\tau_0>0$, then by the weak maximum principle, we have
\begin{equation*}
    \begin{split}
        0&\le  \left(\frac{\partial}{\partial\tau}-\Delta_{\omega_\psi,X}\right)\left(|\Rm(g_\psi)|+(C(n)+1)S-\eta e^\tau f_\psi\right)(x_0,\tau_0)\\
        &\le -\frac{1}{2}|\Rm(g_\psi)|^2(x_0,\tau_0)+|\Rm(g_\psi)|(x_0,\tau_0)+C_5.
    \end{split}
\end{equation*}
Therefore, we get $|\Rm(g_\psi)|(x_0,\tau_0)\le 1+\sqrt{2C_5+1}$. For all $(x,\tau)\in M\times [0,T]$, since $S$ is non negative and bounded uniformly by Proposition \ref{bound of S}, it follows that
\begin{equation*}
   \begin{split}
        |\Rm(g_\psi)|(x,\tau)-\eta e^\tau f_\psi(x,\tau)&\le |\Rm(g_\psi)|(x_0,\tau_0)+(C(n)+1)S(x_0,\tau_0)\\
        &\le 1+\sqrt{2C_5+1}+\sup_{M\times [0,T_{\max})}S(C(n)+1).
   \end{split}
\end{equation*}
If $\tau_0=0$, then for all $(x,\tau)\in M\times [0,T]$, we have
\begin{equation*}
     |\Rm(g_\psi)|(x,\tau)-\eta e^\tau f_\psi(x,\tau)\le\sup_M \left(|\Rm(g_\psi)|(0)+(C(n)+1)S(0)\right)<\infty.
\end{equation*}
Let $\eta$ tend to 0, and we conclude that there exists a constant $C>0$ such that on $M\times [0,T_{\max})$, $\sup_M|\Rm(g_\psi))|(\tau)\le C$ holds.
\end{proof}
\begin{corollary}  Let $T_{\max}$ be the maximal time as in Theorem \ref{shi's solution to NKRF}, then
    $T_{\max}=\infty$.
\end{corollary}
\begin{proof}
    Recall in Theorem \ref{shi's solution to NKRF}, if $T_{\max}<\infty$, then 
    \begin{equation*}
        \limsup_{\tau\to T_{\max}}\sup_M|\Rm(g_{\psi}(\tau))|_{g_\psi(\tau)}=+\infty,
    \end{equation*}
    this gives a contradiction to Corollary \ref{condition I bound of Rm}. Therefore $T_{\max}=\infty$.
\end{proof}
Now to prove Theorem \ref{longtime existence theorem}, it suffices to prove the uniform bound for all orders' covariant derivatives of Riemann curvature operators.
\begin{proof}[Proof of Theorem \ref{longtime existence theorem}]
    This proof relies on Shi's global estimates on curvature along Ricci flow (see \cite{MR1001277}). Recall $g_\varphi(t)_{t\ge 1}$ is the solution to K\"ahler-Ricci flow corresponding to solution $g_\psi(\tau)$ as in Theorem \ref{shi's solution}. The curvature bound of $g_\psi(\tau)$ in Corollary \ref{prop bound of rm} implies that there exists a constant $C>0$ such that for all $t\ge 1$
    \begin{equation}
        \sup_M|\Rm(g_\varphi(t))|_{g_\varphi(t)}\le \frac{C}{t}.
    \end{equation}
    Now we take $\alpha>1$, for all $t\ge\alpha$, we apply Shi's global estimates on the time interval $[\frac{t}{\alpha}, t]$. Since on $M\times [\frac{t}{\alpha}, t]$, we have 
   \begin{equation*}
       \sup_{M\times [\frac{t}{\alpha}, t]}|\Rm(g_\varphi(\cdot))|_{g_\varphi(\cdot)}\le \frac{\alpha C}{t}.
   \end{equation*}
    Thanks to Shi's global estimates, for any $m\in \N_0$, there exists a constant $C_m=C(\dim M,m)$ such that for all $s\in [\frac{t}{\alpha}, t] $
   \begin{equation*}
       \sup_M|(\nabla^{g_\varphi(s)})^m\Rm(g_\varphi(s))|_{g_\varphi(s)}\le C_m\frac{\alpha C}{t}(s-\frac{t}{\alpha})^{-\frac{m}{2}}.
   \end{equation*}
    In particular, there exists a constant $C_m(\alpha)>0$ such that,
    \begin{equation*}
        \sup_M|(\nabla^{g_\varphi(t)})^m\Rm(g_\varphi(t))|_{g_\varphi(t)}\le C_m\frac{\alpha C}{t}(t-\frac{t}{\alpha})^{-\frac{m}{2}}=C_m(\alpha) t^{-1-\frac{m}{2}}.
    \end{equation*}
    By the correspondence of K\"ahler-Ricci flow and normalized K\"ahler-Ricci flow, therefore, for all $m\ge 1$,
    \begin{equation*}
        \sup_M|(\nabla^{g_\psi(\tau)})^m\Rm(g_\psi(\tau))|_{g_\psi(\tau)}\le C_m(\alpha),
    \end{equation*}
    holds for all $\tau\ge \log\alpha$. 
\end{proof}
\section{Proof of convergence theorem \ref{convergence theorem}}\label{proof of 2}
Let $(M,g,X)$ be an asymptotically conical gradient K\"ahler-Ricci expander and let $\psi_0$ be a smooth function defined on $M$ satisfying Condition II. Since Condition II is stronger than Condition I, all the previous uniform estimates along the normalized K\"ahler-Ricci flow still hold. In summary, we have
\begin{prop}\label{summary}
    There exists a real-valued smooth function $\psi\in C^{\infty}(M\times[0,\infty))$ such that $g_\psi:=g+\partial\bar\partial\psi$ is an immortal complete solution to normalized K\"ahler-Ricci flow starting from $g_{\psi_0}$ and $\psi_(\tau)_{\tau\in[0,\infty)}$ satisfies
    \begin{equation*}
        \begin{split}
            &\frac{\partial}{\partial\tau}\psi(\tau)=\log\frac{\omega_\psi(\tau)^n}{\omega^n}+\frac{X}{2}\cdot\psi(\tau)-\psi(\tau),\\
            &JX\cdot\psi=0,\\
            &\psi(0)=\psi_0.
        \end{split}
    \end{equation*}

    Moreover, there exists a constant $C>1$ such that for all $\tau\ge 0$
  \begin{enumerate}
      \item $\frac{1}{C}g\le g_\psi\le Cg$;
      \item $|\frac{X}{2}\cdot\psi-\psi|\le C$;
      \item $\frac{1}{C}f\le f_\psi\le Cf$;
      \item $|\dot\psi(\tau)|\le Ce^{-\tau},f_\psi|\partial\dot\psi|_{g_\psi}^2\le Ce^{-\tau}$;
      \item $f_\psi S\le C$;
      \item $|\Rm(g_\psi)|\le C$.
  \end{enumerate}
\end{prop}
\subsection{A priori estimates}
To study the convergence of flow, we need better uniform estimates on the Riemann curvature $\Rm(g_\psi)$. In this subsection we omit all the time dependence of $g_\psi(\tau)$, we denote it simply by $g_\psi$. Moreover, we use $\nabla$(resp. $\overline{\nabla}$) to denote the holomorphic (resp. anti-holomorphic) part of $\nabla^{g_\psi}$, we also use $|\cdot|$ to denote $|\cdot|_{g_\psi}$ and $<>$ to denote the scalar product of $g_\psi$.

\begin{prop}\label{prop bound of rm}
    There exists a constant $C>0$ such that on $M\times [0,\infty)$,
    \begin{equation*}
        |\Rm(g_\psi)|\le C\frac{1}{f_\psi},
    \end{equation*}
    holds uniformly.
\end{prop}
\begin{proof}
    We consider $f_\psi |\Rm(g_\psi)|$, at a point $(x,\tau)\in M\times [0,\infty)$ such that $|\Rm(g_\psi)|(x,\tau)>0$. Due to [\eqref{evolution equation of Rm1}, Corollary \ref{rm corollary}], there exist constants $C_1,A>0$ such that
    \begin{equation*}
         \begin{split}
             \left(\frac{\partial}{\partial\tau}-\Delta_{\omega_\psi,X}\right)(f_\psi|\Rm(g_\psi)|)&=-f_\psi|\Rm(g_\psi)|+f_\psi\left(\frac{\partial}{\partial\tau}-\Delta_{\omega_\psi,X}\right)|\Rm(g_\psi)|\\
             &\quad -2\Re<\partial f_\psi,\bar\partial|\Rm(g_\psi)|>_{g_\psi}\\
             &\le \frac{C_1}{2}f_\psi|\Rm(g_\psi)|^2-X\cdot |\Rm(g_\psi)|\\
             &=\frac{C_1}{2}f_\psi|\Rm(g_\psi)|^2-\frac{X}{f_\psi}\cdot(f_\psi|\Rm(g_\psi)|)+\frac{|X|^2}{f_\psi}|\Rm(g_\psi)|\\
             &\le \frac{C_1}{2}f_\psi|\Rm(g_\psi)|^2-\frac{X}{f_\psi}\cdot(f_\psi|\Rm(g_\psi)|)+A|\Rm(g_\psi)|\\
             &\le \frac{C_1}{2}f_\psi|\Rm(g_\psi)|^2-\frac{X}{f_\psi}\cdot(f_\psi|\Rm(g_\psi)|)\\
             &\quad +Af_\psi|\Rm(g_\psi)|^2+\frac{A}{f_\psi}\\
         \end{split}
    \end{equation*}
    Here the fourth line inequality is ensured by Corollary \ref{bound of X}, and the last line inequality comes from the Cauchy-Schwarz inequality.
    If we denote $\frac{X}{2}-\frac{X}{f_\psi}$ by $\tilde{X}$, and we still use $\frac{C_1}{2}$ to denote $\frac{C_1}{2}+A$, then we get
    \begin{equation*}
        \left(\frac{\partial}{\partial\tau}-\Delta_{\omega_\psi}-\tilde{X}\right)(f_\psi|\Rm(g_\psi)|)\le \frac{C_1}{2}f_\psi|\Rm(g_\psi)|^2+\frac{A}{f_\psi}.
    \end{equation*}
    Moreover, recall [\eqref{lemma for nest lemma}, Lemma \ref{lemma 5.3}, there is a constant $C_2$ so that
    \begin{equation*}
        \left(\frac{\partial}{\partial\tau}-\Delta_{\omega_\psi,X}\right)(f_\psi S)\le -f_\psi|\overline{\nabla}\Psi|^2-f_\psi|\nabla\Psi|^2+C_2(S+\frac{\sqrt{S}}{f_\psi^{\frac{1}{2}}})-X\cdot S.
    \end{equation*}
    Therefore,
    \begin{equation*}
         \left(\frac{\partial}{\partial\tau}-\Delta_{\omega_\psi}-\tilde X\right)(f_\psi S)\le -f_\psi|\overline{\nabla}\Psi|^2-f_\psi|\nabla\Psi|^2+C_2(S+\frac{\sqrt{S}}{f_\psi^{\frac{1}{2}}})+\frac{|X|^2}{f_\psi}S.
    \end{equation*}
    Recall that there exists a constant $C_3>0$ so that $S\le\frac{C_3}{f_\psi}$ and $|\overline\nabla \Psi|^2\ge \frac{1}{2}|\Rm(g_\psi)|^2-\frac{C_3}{f_\psi}$, we conclude that there exists a constant $C_4>0$ so that
    \begin{equation*}
        \left(\frac{\partial}{\partial\tau}-\Delta_{\omega_\psi}-\tilde X\right)(f_\psi S)\le -\frac{1}{4}f_\psi|\Rm(g_\psi)|^2+C_4\frac{1}{f_\psi}.
    \end{equation*}
    Now we use $e^{2\tau}f_\psi^2$ as a barrier function, we compute:
    \begin{equation*}
    \begin{split}
         \left(\frac{\partial}{\partial\tau}-\Delta_{\omega_\psi}-\tilde{X}\right)(e^{2\tau}f_\psi^2)   &=2e^{2\tau}f_\psi^2+e^{2\tau}(\frac{\partial}{\partial\tau}-\Delta_{\omega_\psi}-\tilde{X})(f_\psi^2)\\
         &=2e^{2\tau}f_\psi^2+e^{2\tau}(-2f_\psi^2+2|\partial f_\psi|^2)\ge 0.
    \end{split} 
    \end{equation*}
    Hence we get there exists a constant $C_5>0$ such that for all $\eta>0$, 
    \begin{equation*}
    \left(\frac{\partial}{\partial\tau}-\Delta_{\omega_\psi}-\tilde{X}\right)\left(f_\psi|\Rm(g_\psi)+(2C_1+4)f_\psi S-\eta e^{2\tau}f_\psi^2\right)    \le -|\Rm(g_\psi)|^2f_\psi+\frac{C_4}{f_\psi}.
    \end{equation*}
  For all $T<\infty$, there exists a constant $B>0$ such that $f_\psi|\Rm(g_\psi)+(2C_1+4)f_\psi S\le Bf$ by the rough bound of Riemann curvature in Theorem \ref{shi's solution to NKRF}, therefore, $f_\psi|\Rm(g_\psi)+(2C_1+4)f_\psi-\eta e^{2\tau}f_\psi^2<0$ tends to $-\infty$ at spatial infinity uniformly so that the maximum point $(x_0,\tau_0)$ exists. 
If $\tau_0>0$, by the weak maximum principle, therefore,
  \begin{equation*}
      0\le -|\Rm(g_\psi)|^2f_\psi(x_0,\tau_0)+\frac{C_4}{f_\psi(x_0,\tau_0)}.
  \end{equation*}
  This implies that $|\Rm(g_\psi)|(x_0,\tau_0)\le\sqrt{C_4}f_\psi(x_0,\tau_0)^{-1}$, and therefore for all $\tau\in [0,T]$
  \begin{equation*}
      f_\psi|\Rm(g_\psi)|+(2C_1+4)f_\psi S-\eta e^{2\tau}f_\psi^2\le f_\psi|\Rm(g_\psi)|(x_0,\tau_0)+(2C_1+4)f_\psi S(x_0,\tau_0).
  \end{equation*}
  Since $|\Rm(g_\psi)|(x_0,\tau_0)\le\sqrt{C_4}f_\psi(x_0,\tau_0)^{-1}$ and $S\le \frac{C}{f_\psi}$ for some $C$ on $M\times [0,\infty)$, we conclude that there exists a constant $C_5>0$ independent of $T$ and $\eta$ such that for all $\tau\in [0,T]$
  \begin{equation*}
       f_\psi|\Rm(g_\psi)|+(2C_1+4)f_\psi S-\eta e^{2\tau}f_\psi^2\le C_5.
  \end{equation*}

  If $\tau_0=0$ , we have for all $\tau\in [0,T]$,
  \begin{equation*}
       f_\psi|\Rm(g_\psi)|+(2C_1+4)f_\psi S-\eta e^{2\tau}f_\psi^2\le \sup_M f_\psi(0)|\Rm(g_\psi)|(0)+(2C_1+4)f_\psi(0) S(0)<\infty.
  \end{equation*}
  Take $C_6=C_5+\sup_M f_\psi(0)|\Rm(g_\psi)(0)+(2C_1+4)f_\psi(0) S(0)$, this constant is independent of $T$ and $\eta$, hence we have for all $\tau\in [0,\infty)$, for all $\eta>0$,
  \begin{equation}\label{5.25}
       f_\psi|\Rm(g_\psi)|\le f_\psi|\Rm(g_\psi)+(2C_1+4)f_\psi S\le C_6+\eta e^{2\tau}f_\psi^2.
  \end{equation}
  Let $\eta$ in \eqref{5.25} tend to 0 to get $ f_\psi|\Rm(g_\psi)|\le C_6$ as desired.
\end{proof}
Now we use $\nabla$ to denote $\nabla^{g_\psi}$, the real covariant derivative with respect to $g_\psi$, and we investigate to estimate higher order covariant derivatives of $\Rm(g_\psi)$.

Before developing the evolution equation of high order covariant derivatives of the Riemann curvature tensor, we need a technical lemma as follows:
\begin{lemma}\label{lie lemma}
   Let $k\in\N^*$ and let $R\in C^\infty(\otimes^kT^*M)$ be a $k-$tensor that may depend on time. Then along the normalized K\"ahler-Ricci flow
   \begin{equation*}
       \frac{\partial}{\partial\tau}\nabla R-\mathcal{L}_{\frac{X}{2}}\left(\nabla R\right)=\nabla\left(\frac{\partial}{\partial\tau}R-\mathcal{L}_{\frac{X}{2}}R\right)+\nabla \Ric*R.
   \end{equation*}
\end{lemma}
\begin{proof}
    We compute it in local real coordinates, let $\{a,i_0,i_1,...,i_k\}$ be any real index.
    \begin{equation}\label{lie lemma 1}
       \begin{split}
            \frac{\partial}{\partial\tau}\nabla_{i_0} R_{i_1...i_k}&=\nabla_{i_0}\left(\frac{\partial}{\partial\tau}R\right)_{i_1...i_k}-\sum_j\frac{\partial}{\partial\tau}\Gamma(g_\psi)_{i_0i_j}^aR_{i_1...a...i_k}\\
            &=\nabla_{i_0}\left(\frac{\partial}{\partial\tau}R\right)_{i_1...i_k}-\sum_jA_{i_0i_j}^aR_{i_1...a...i_k}.
       \end{split}
    \end{equation}
    Here 
    \begin{equation*}
        \begin{split}
            A_{i_0i_j}^a&=\frac{1}{2}g_{\psi}^{ab}\left(\nabla_{i_0} (\mathcal{L}_{\frac{X}{2}}g_\psi)_{i_j b}+\nabla_{i_j}(\mathcal{L}_{\frac{X}{2}}g_\psi)_{i_0 b}-\nabla_{b}(\mathcal{L}_{\frac{X}{2}}g_\psi)_{i_0 i_j}\right)\\
            &\quad -\frac{1}{2}g_{\psi}^{ab}\left(\nabla_{i_0} \Ric(g_\psi)_{i_j b}+\nabla_{i_j}\Ric(g_\psi)_{i_0 b}-\nabla_{b}\Ric(g_\psi)_{i_0 i_j}\right).
        \end{split}
    \end{equation*}
    Moreover the Lie derivative of $\nabla R$ in local holomorphic coordinates is given by:
    \begin{equation}\label{lie lemma 2}
        \begin{split}
            \mathcal{L}_{\frac{X}{2}}\left(\nabla R\right)_{i_0i_1...i_k} &=\nabla_{i_0}\left(\mathcal{L}_{\frac{X}{2}}R\right)_{i_1...i_k}-\sum_j\left(\mathcal{L}_{\frac{X}{2}}\Gamma(g_\psi)\right)_{i_0i_j}^aR_{i_1...a...i_k}\\
            &=\nabla_{i_0}\left(\mathcal{L}_{\frac{X}{2}}R\right)_{i_1...i_k}-\sum_jB_{i_0i_j}^aR_{i_1...a...i_k},
        \end{split}
    \end{equation}
    where $B_{i_0i_j}^a=\frac{1}{2}g_{\psi}^{ab}\left(\nabla_{i_0} (\mathcal{L}_{\frac{X}{2}}g_\psi)_{i_j b}+\nabla_{i_j}(\mathcal{L}_{\frac{X}{2}}g_\psi)_{i_0 b}-\nabla_{b}(\mathcal{L}_{\frac{X}{2}}g_\psi)_{i_0 i_j}\right)$.

    Combing \eqref{lie lemma 1} and $\eqref{lie lemma 2}$, we have that $ \frac{\partial}{\partial\tau}\nabla R-\mathcal{L}_{\frac{X}{2}}\left(\nabla R\right)=\nabla\left(\frac{\partial}{\partial\tau}R-\mathcal{L}_{\frac{X}{2}}R\right)+\nabla \Ric*R$ holds as desired.
\end{proof}
 \begin{lemma}
         Along the normalized K\"ahler-Ricci flow, the $m-$th covariant derivative curvature tensor evolves by
    \begin{equation}\label{evolution function of covariant derivative of Rm}
        \begin{split}
            \frac{\partial}{\partial\tau}\nabla^m\Rm(g_\psi)&=\frac{1}{2}\Delta_{g_\psi}\nabla^m\Rm(g_\psi)+\mathcal{L}_{\frac{X}{2}}\nabla^m\Rm(g_\psi)-\nabla^m\Rm(g_\psi)\\
            &\quad +\sum_{p+q=m}\nabla^p\Rm(g_\psi)*\nabla^q\Rm(g_\psi).
        \end{split}
    \end{equation}
    \end{lemma}
    \begin{proof}
   This result is proved by induction. For $m=0$, \eqref{evolution function of covariant derivative of Rm} holds thanks to Proposition \ref{evolution function of Rm}. We suppose that \eqref{evolution function of covariant derivative of Rm} holds for all integers $k\le m$. By applying Lemma \ref{lie lemma} for tensor $\nabla^m\Rm(g_\psi)$, we have
   \begin{equation}\label{time for co der}
       \begin{split}
           \frac{\partial}{\partial\tau}\nabla^{m+1}\Rm(g_\psi)-\mathcal{L}_{\frac{X}{2}}\nabla^{m+1}\Rm(g_\psi)&=\nabla\left(\frac{\partial}{\partial\tau}\nabla^m\Rm(g_\psi)-\mathcal{L}_{\frac{X}{2}}\nabla^{m}\Rm(g_\psi)\right)\\
           &\quad+\nabla\Ric(g_\psi)*\nabla^m\Rm(g_\psi).
       \end{split}
   \end{equation}
   And by high order Bochner's formula
   \begin{equation}\label{lap and lie for co der}
       \begin{split}
               \frac{1}{2}\Delta_{g_\psi}\nabla^{m+1}\Rm(g_\psi)&=\nabla\left(\frac{1}{2}\Delta_{g_\psi}\nabla^{m}\Rm(g_\psi)\right)\\
           &\quad +\nabla\Rm(g_\psi)*\nabla^m\Rm(g_\psi)+\Rm(g_\psi)*\nabla^{m+1}\Rm(g_\psi).
       \end{split}
   \end{equation}
   Combining \eqref{time for co der} and \eqref{lap and lie for co der} and the induction hypothesis, we have that
   \begin{equation*}
      \begin{split}
           &\quad \frac{\partial}{\partial\tau}\nabla^{m+1}\Rm(g_\psi)-\frac{1}{2}\Delta_{g_\psi}\nabla^{m+1}\Rm(g_\psi)-\mathcal{L}_{\frac{X}{2}}\nabla^{m+1}\Rm(g_\psi)\\
           &=\nabla\left( \frac{\partial}{\partial\tau}\nabla^m\Rm(g_\psi)-\frac{1}{2}\Delta_{g_\psi}\nabla^m\Rm(g_\psi)-\mathcal{L}_{\frac{X}{2}}\nabla^m\Rm(g_\psi)\right)\\
           &\quad-\nabla\Rm(g_\psi)*\nabla^m\Rm(g_\psi)+\Rm(g_\psi)*\nabla^{m+1}\Rm(g_\psi)\\
           &=-\nabla^{m+1}\Rm(g_\psi)+\nabla\left(\sum_{p+q=m}\nabla^p\Rm(g_\psi)*\nabla^q\Rm(g_\psi)\right)\\
           &\quad-\nabla\Rm(g_\psi)*\nabla^m\Rm(g_\psi)+\Rm(g_\psi)*\nabla^{m+1}\Rm(g_\psi).
      \end{split}
   \end{equation*}
   Thus, \eqref{evolution function of covariant derivative of Rm} holds for all $m\in\N_0$.
    \end{proof}
    \begin{corollary}\label{induction formulae}
        For $m\in\N_0$, there exists a constant $C(m,n)>0$ such that
        \begin{equation}
           \begin{split}
                &\quad \left(\frac{\partial}{\partial\tau}-\Delta_{\omega_\psi,X}\right)\left(|\nabla^m\Rm(g_\psi)|^2\right)\\&\le -|\nabla^{m+1}\Rm(g_\psi)|^2+(m+2)|\nabla^m\Rm(g_\psi)|^2\\
                &\quad +C(m,n)\sum_{p+q=m}|\nabla^p\Rm(g_\psi)||\nabla^q\Rm(g_\psi)||\nabla^m\Rm(g_\psi)|.
           \end{split}
        \end{equation}
    \end{corollary}
    \begin{proof}
        This proof is analogous to the proof of [\eqref{evolution equation of Rm square}, Corollary \ref{rm corollary}] and will therefore be omitted.
    \end{proof}
    \begin{corollary}\label{higher covariant derivative bound of Rm} For all $m\in\N_0$
        there exists a constant $C_m>0$ such that on $M\times [0,\infty)$,
        \begin{equation}\label{boun co der rm}
            |\nabla^m\Rm(g_\psi)|\le C_m\frac{1}{f_\psi^{\frac{m}{2}+1}}.
        \end{equation}
    \end{corollary}
    To obtain this estimate, we need a rough estimate on higher order covariant derivatives of $\Rm(g_\psi)$. Now we consider $g_\varphi(t)_{t\in [1,\infty)}$ as the solution to K\"ahler-Ricci flow corresponding to $g_\psi(\tau)_{\tau\in [1,\infty)}$. The following lemma asserts that $\nabla^m \mathrm{Rm}(g_\psi)$ remains uniformly bounded on every compact time interval $[0,T]$. This follows from Shi’s local derivative estimates for the Ricci flow on non-compact manifolds.
    \begin{lemma}\label{a priori estimate of rm}
        For any $0<T<\infty$, for all $m\in\N_0$ there exists a constant $B_m=B(m,T)>0$ such that on $M\times [0,T]$,
        \begin{equation}
            |\nabla^m\Rm(g_\psi)|\le C_m.
        \end{equation}
    \end{lemma}
    \begin{proof}
    Recall that $g_\varphi(t)=t\Phi_t^*g_\psi(\log t)$ for all $t\ge 1$ is the solution to K\"ahler-Ricci flow as in Proposition \ref{correspondence}.
It suffices to prove that there exists a constant $B_m>0$ such that on $M\times [1,e^T]$,
\begin{equation*}
    |(\nabla^{g_\varphi})^m\Rm(g_\varphi)|_{g_\varphi}\le B_m.
\end{equation*}
    By Proposition \ref{prop bound of rm}, on $M\times [0,\infty)$, there exists a constant $C>0$ such that
        \begin{equation*}
            |\Rm(g_\psi)|\le\frac{C}{f}.
        \end{equation*}
        Hence on $M\times [1,e^T]$, there exists a constant $B=B(T)$ such that
        \begin{equation*}
            |\Rm(g_\varphi)|_{g_\varphi}\le \frac{B(T)}{f}.
        \end{equation*}
        Moreover, by our initial condition $|(\nabla^g)^k(g_{\psi_0}-g)|=O(f^{-\frac{k}{2}})$ as in Definition \ref{condition II}, for all $k\le m$, there exists a constant $A_k>0$ such that
        \begin{equation*}
            |(\nabla^{g_\varphi})^k\Rm(g_\varphi)|(0)\le \frac{A_k}{f^{1+\frac{m}{2}}}.
        \end{equation*}
        Thus by Perelman's pseudolocality theorem (see \cite[Lemma A.4]{MR2734348} or \cite{MR2274812}) , we conclude that there exists a constant $B_m$ such that
        \begin{equation*}
            |(\nabla^{g_\varphi})^m\Rm(g_\varphi)|_{g_\varphi}\le B_m.
        \end{equation*}
    \end{proof}
    With the above Lemma, we can prove Corollary \ref{higher covariant derivative bound of Rm}.
    \begin{proof}[Proof of Corollary \ref{higher covariant derivative bound of Rm}]
        For $m\in\N_0$, we compute,
         \begin{equation}\label{proof of corollary 5.7 1}
           \begin{split}
                \left(\frac{\partial}{\partial\tau}-\Delta_{\omega_\psi,X}\right)(f_\psi^{m+2}|\nabla^m\Rm(g_\psi)|^2)&=|\nabla^m\Rm(g_\psi)|^2\left(\frac{\partial}{\partial\tau}-\Delta_{\omega_\psi,X}\right)f_\psi^{m+2}\\
                &\quad +f_\psi^{m+2}\left(\frac{\partial}{\partial\tau}-\Delta_{\omega_\psi,X}\right)|\nabla^m\Rm(g_\psi)|^2\\
                &\quad -2(m+2)f_\psi^{m+1}\Re(<\partial f_\psi,\bar\partial|\nabla^m\Rm(g_\psi)|^2>).
           \end{split}
        \end{equation}
        On the one hand,
        \begin{equation}\label{proof of corollary 5.7 2}
            \left(\frac{\partial}{\partial\tau}-\Delta_{\omega_\psi,X}\right)f_\psi^{m+2}=-(m+2)f_\psi^{m+2}-(m+2)(m+1)|\partial f_\psi|^2f_\psi^m.
        \end{equation}
        On the other hand, by Corollary \ref{induction formulae},
         \begin{equation}\label{proof of corollary 5.7 3}
           \begin{split}
                &\quad\left(\frac{\partial}{\partial\tau}-\Delta_{\omega_\psi,X}\right)\left(|\nabla^m\Rm(g_\psi)|^2\right)\\
                &\le -|\nabla^{m+1}\Rm(g_\psi)|^2+(m+2)|\nabla^m\Rm(g_\psi)|^2\\
                &\quad +C(m,n)\sum_{p+q=m}|\nabla^p\Rm(g_\psi)||\nabla^q\Rm(g_\psi)||\nabla^m\Rm(g_\psi)|.
           \end{split}
        \end{equation}
        By combining \eqref{proof of corollary 5.7 1}, \eqref{proof of corollary 5.7 2} and \eqref{proof of corollary 5.7 3}, we get,
        \begin{equation*}
            \begin{split}
                  &\quad \left(\frac{\partial}{\partial\tau}-\Delta_{\omega_\psi,X}\right)(f_\psi^{m+2}|\nabla^m\Rm(g_\psi)|^2)\\
                  &\le -f_\psi^{m+2}|\nabla^{m+1}\Rm(g_\psi)|^2\\
                  &\quad+f_\psi^{m+2}C(n,m)\sum_{p+q=m}|\nabla^p\Rm(g_\psi)||\nabla^q\Rm(g_\psi)||\nabla^m\Rm(g_\psi)|\\
                  &\quad-(m+2)f_\psi^{m+1}X\cdot|\nabla^m\Rm(g_\psi)|^2.
            \end{split}
        \end{equation*}
     By Cauchy-Schwarz inequality, for any $\sigma>0$, we have
        \begin{equation*}
          \begin{split}
                X\cdot|\nabla^m\Rm(g_\psi)|^2&\le 2|X||\nabla^{m+1}\Rm(g_\psi)||\nabla^m\Rm(g_\psi)|\\
                &\le \sigma|X|^2|\nabla^{m+1}\Rm(g_\psi)|^2+\frac{1}{\sigma}|\nabla^m\Rm(g_\psi)|^2.
          \end{split}
        \end{equation*}
        Choose $\sigma>0$ such that $(m+2)\sigma|X|^2\le \frac{1}{2}f_\psi$ thanks to Corollary \ref{bound of X}, to get
        \begin{equation}\label{general formulae for all m}
            \begin{split}
                  &\left(\frac{\partial}{\partial\tau}-\Delta_{\omega_\psi,X}\right)(f_\psi^{m+2}|\nabla^m\Rm(g_\psi)|^2)\\&\le -\frac{1}{2}f_\psi^{m+2}|\nabla^{m+1}\Rm(g_\psi)|^2+(m+2)\frac{1}{\sigma}f_\psi^{m+1}|\nabla^m\Rm(g_\psi)|^2\\
                  &\quad+f_\psi^{m+2}C(n,m)\sum_{p+q=m}|\nabla^p\Rm(g_\psi)||\nabla^q\Rm(g_\psi)||\nabla^m\Rm(g_\psi)|.
            \end{split}
        \end{equation}
        Now we use induction to prove [\eqref{boun co der rm}, Corollary \ref{higher covariant derivative bound of Rm}], since it holds when $m=0$ by Proposition \ref{prop bound of rm}, we suppose that \eqref{boun co der rm} holds for all $k\le m\in \N_0$, now for $k=m+1$, \eqref{general formulae for all m} implies that there exists a constant $C'>0$ such that
        \begin{equation*}
             \begin{split}
                 \left(\frac{\partial}{\partial\tau}-\Delta_{\omega_\psi,X}\right)(f_\psi^{m+3}|\nabla^{m+1}\Rm(g_\psi)|^2)&\le C'f_\psi^{m+2}|\nabla^{m+1}\Rm(g_\psi)|^2\\
                 &\quad +C'f_\psi^{\frac{m+1}{2}}|\nabla^{m+1}\Rm(g_\psi)|.
             \end{split}
        \end{equation*}
        By Cauchy-Schwarz inequality, there exists a constant $C''>0$ such that
          \begin{equation*}
                 \left(\frac{\partial}{\partial\tau}-\Delta_{\omega_\psi,X}\right)(f_\psi^{m+3}|\nabla^{m+1}\Rm(g_\psi)|^2)\le C''f_\psi^{m+2}|\nabla^{m+1}\Rm(g_\psi)|^2+C''\frac{1}{f_\psi}.
        \end{equation*}
        Now for $m$, by induction hypothesis and \eqref{general formulae for all m}, we can find a constant $C_0>0$ such that
        \begin{equation*}
              \left(\frac{\partial}{\partial\tau}-\Delta_{\omega_\psi,X}\right)(f_\psi^{m+2}|\nabla^{m}\Rm(g_\psi)|^2)\le -\frac{1}{2}f_\psi^{m+2}|\nabla^{m+1}\Rm(g_\psi)|^2+C_0\frac{1}{f_\psi}.
        \end{equation*}
        Hence \begin{equation*}
            \left(\frac{\partial}{\partial\tau}-\Delta_{\omega_\psi,X}\right)(f_\psi^{m+3}|\nabla^{m+1}\Rm(g_\psi)|^2+2C''f_\psi^{m+2}|\nabla^{m}\Rm(g_\psi)|^2)\le (2C''C_0+C'')\frac{1}{f_\psi}.
        \end{equation*}
       Thanks to Lemma \ref{a priori estimate of rm}, we can apply Lemma \ref{useful lemma} on $f_\psi^{m+3}|\nabla^{m}\Rm(g_\psi)|^2+2C''f_\psi^{m+2}|\nabla^{m+1}\Rm(g_\psi)|^2$, there exists a constant $C_1>0$ such that on $M\times [0,\infty)$ such that
       \begin{equation*}
          \begin{split}
               \sup_{M\times [0,\infty)}f_\psi^{m+3}|\nabla^{m}\Rm(g_\psi)|^2+2C''f_\psi^{m+2}|\nabla^{m+1}\Rm(g_\psi)|^2&\le C_1+C_2,
          \end{split}
       \end{equation*}
       with $C_2=\sup_Mf_\psi^{m+3}(0)|\nabla^{m+1}\Rm(g_\psi)|^2(0)+2C''f_\psi^{m+2}(0)|\nabla^{m+1}\Rm(g_\psi)|^2(0)$ which is finite by Condition II in Definition \ref{condition II}, thus [\eqref{boun co der rm}, Corollary \ref{higher covariant derivative bound of Rm}] holds for $m+1$, therefore \eqref{boun co der rm} holds for all integers as expected.
    \end{proof}
 With these crucial curvature estimates in hand, one can investigate other geometric tensors along the normalized Kähler–Ricci flow. Of particular importance is the tensor
\begin{equation*}
    \mathcal{L}_{\tfrac{X}{2}} g_\psi - \Ric(g_\psi) - g_\psi,
\end{equation*}
which measures whether a metric is an expanding Kähler–Ricci soliton. By the normalized Kähler–Ricci flow equation, this tensor is in fact given by $\partial \bar{\partial} \dot{\psi}$.
    \begin{definition}[Obstruction tensor]
        Define $T:=\partial\bar\partial\dot\psi$, or equivalently, by normalized K\"ahler-Ricci flow equation \begin{equation*}
            T=\mathcal{L}_{\frac{X}{2}}g_\psi-\Ric(g_\psi)-g_\psi.
        \end{equation*}
    \end{definition}
   This obstruction tensor plays a crucial role in analyzing the long time behavior of solutions to the normalized K\"ahler-Ricci flow. Establishing precise estimates for this tensor will enable us to prove Theorem \ref{convergence theorem}. Before giving uniform control of $\nabla^kT$ for all $k\in\N_0$, we present some rough estimates of $\nabla^kT$.
    \begin{lemma}[Rough estimates on $\nabla^kT$]\label{A priori estimate of nabla^kT}
        For all $0<T<\infty$ $k\in \N_0$, there exists a constant $B_k=B(k,T)>0$ such that on $M\times [0,T]$, we have
       \begin{equation*}
           |\nabla^kT|\le B_k.
       \end{equation*}
    \end{lemma}
    \begin{proof}
        Recall that $T=\mathcal{L}_{\frac{X}{2}}g_\psi-\Ric(g_\psi)-g_\psi$. Now $\nabla^k(\Ric(g_\psi)+g_\psi)$ is uniformly bounded on space-time by Corollary \ref{higher covariant derivative bound of Rm}. It suffices to show that $(\nabla^{g_\psi})^k\mathcal{L}_{\frac{X}{2}}g_\psi$ is bounded on $M\times [0,T]$. Let $g_\varphi$ be the solution to K\"ahler-Ricci flow corresponding to $g_\psi$ as in Proposition \ref{correspondence}, it suffices for us to control $(\nabla^{g_\varphi})^k\mathcal{L}_{\frac{X}{2}}g_\varphi$ on $M\times [0,T]$.

        Thanks to K\"ahler-Ricci flow equation
        \begin{equation}\label{rough estimates of TT}
            \begin{split}
                \frac{\partial}{\partial t}(\nabla^{g_\varphi})^k\mathcal{L}_{\frac{X}{2}}g_\varphi&=-(\nabla^{g_\varphi})^k\mathcal{L}_{\frac{X}{2}}\Ric(g_\varphi)+\sum_{\substack{p+q=k\\ q<k}}(\nabla^{g_\varphi})^p\Ric(g_\varphi)*(\nabla^{g_\varphi})^q\mathcal{L}_{\frac{X}{2}}g_\varphi\\
                &=(\nabla^{g_\varphi})^k\left(\nabla^{g_\varphi}\Ric(g_\varphi)*X+\nabla^{g_\varphi}X*\Ric(g_\varphi)\right)\\
                &\quad +\sum_{{\substack{p+q=k\\ q<k}}}(\nabla^{g_\varphi})^p\Ric(g_\varphi)*(\nabla^{g_\varphi})^q\mathcal{L}_{\frac{X}{2}}g_\varphi.
            \end{split}
        \end{equation}

        Thanks to Corollary \ref{higher covariant derivative bound of Rm}, $(\nabla^{g_\varphi})^k\left(\nabla^{g_\varphi}\Ric(g_\psi)*X+\nabla^{g_\varphi}X*\Ric(g_\psi)\right)$ is bounded on $M\times[0,T]$ for all $k\in\N_0$. In particular, $\mathcal{L}_{\frac{X}{2}}g_\varphi$ is bounded on $M\times [0,T]$. Now we assume that $(\nabla^{g_\varphi})^l\mathcal{L}_{\frac{X}{2}}g_\varphi$ is bounded for all $l\le k$ on $M\times [0,T]$, then for $l=k+1$, the induction hypothesis, together with curvature bound and ODE method imply that $(\nabla^{g_\varphi})^{l+1}\mathcal{L}_{\frac{X}{2}}g_\varphi$ is bounded on $M\times [0,T]$.

        Thus, on $M\times [0,T]$, $\nabla^kT$ is bounded by some positive constant $B_k=B(k,T)$ as expected.
    \end{proof}
    \begin{lemma}
        The $k-$th covariant derivative $\nabla^k T$ evolves along the normalized K\"ahler-Ricci flow by 
        \begin{equation}\label{inductive eovolution equation of dotpsi}
            \left(\frac{\partial}{\partial\tau}-\frac{1}{2}\Delta_{g_\psi}-\mathcal{L}_{\frac{X}{2}}\right)\nabla^kT=-\nabla^kT+\sum_{p+q=k}\nabla^p\Rm(g_\psi)*\nabla^qT.
        \end{equation}
    \end{lemma}
    \begin{proof}
        Let us prove this Lemma by induction, \eqref{inductive eovolution equation of dotpsi} holds for $k=0$ because
        \begin{equation*}
          \begin{split}
                \frac{\partial}{\partial\tau}T&=\partial\bar\partial \left(\frac{\partial}{\partial\tau}\dot\psi\right)=\partial\bar\partial \left(\Delta_{\omega_\psi,X}\dot\psi-\dot\psi\right)\\
            &=-T+\mathcal{L}_{\frac{X}{2}}T+\partial\bar\partial(\Delta_{\omega_\psi}\dot\psi)\\
            &=\left(\frac{1}{2}\Delta_{g_\psi}+\mathcal{L}_{\frac{X}{2}}\right)T+\Rm(g_\psi)*T-T.
          \end{split}
        \end{equation*}
        Here the last line is ensured by 
        \begin{equation*}
            \partial\bar\partial(\Delta_{\omega_\psi}\dot\psi)=\Delta_{\omega_\psi}\partial\bar\partial \dot\psi+\Rm(g_\psi)*\partial\bar\partial \dot\psi,
        \end{equation*}
        and,
        \begin{equation}
           \Delta_{\omega_\psi}\partial\bar\partial \dot\psi=\overline\Delta_{\omega_\psi}\partial\bar\partial\dot\psi+\Rm(g_\psi)* \partial\bar\partial\dot\psi.
        \end{equation}
        Suppose that \eqref{inductive eovolution equation of dotpsi} holds for all integers $m\le k$. For $\nabla^{k+1}T$, by Lemma \ref{lie lemma}, we compute
        \begin{equation}\label{time der of co dotpsi}
            \frac{\partial}{\partial\tau}\nabla^{k+1}T-\mathcal{L}_{\frac{X}{2}}\nabla^{k+1}T=\nabla\left(\frac{\partial}{\partial\tau}\nabla^kT-\mathcal{L}_{\frac{X}{2}}\nabla^kT\right)+\nabla\Ric(g_\psi)*\nabla^kT.
        \end{equation}
        And by Bochner's formula, we have
        \begin{equation}\label{lie der and lap of dotpsi}
            \begin{split}
                \frac{1}{2}\Delta_{g_\psi}\nabla^{k+1}T=\nabla\left(\frac{1}{2}\Delta_{g_\psi}\nabla^kT\right)+\nabla\Rm(g_\psi)*\nabla^kT+\Rm(g_\psi)*\nabla^{k+1}T.
            \end{split}
        \end{equation}
        Combining \eqref{time der of co dotpsi} and \eqref{lie der and lap of dotpsi}, we have
        \begin{equation*}
           \begin{split}
                 &\quad \left(\frac{\partial}{\partial\tau}-\frac{1}{2}\Delta_{g_\psi}-\mathcal{L}_{\frac{X}{2}}\right)\nabla^{k+1}T\\&=\nabla\left(\left(\frac{\partial}{\partial\tau}-\frac{1}{2}\Delta_{g_\psi}-\mathcal{L}_{\frac{X}{2}}\right)\nabla^kT\right)+\Rm(g_\psi)*\nabla^{k+1}T+\nabla\Rm(g_\psi)*\nabla^kT.
           \end{split}
        \end{equation*}
        By the induction hypothesis, we have that \eqref{inductive eovolution equation of dotpsi} holds for $k+1$.
        
    \end{proof}
    \begin{prop}
        There exists a constant $C(k,n)>0$ such that on $M\times [0,\infty)$,
        \begin{equation*}
            \left(\frac{\partial}{\partial\tau}-\Delta_{\omega_\psi,X}\right)|\nabla^kT|^2\le -|\nabla^{k+1}T|^2+k|\nabla^kT|^2+C(k,n)\sum_{p+q=k}|\nabla^p\Rm(g_\psi)||\nabla^qT||\nabla^kT|.
        \end{equation*}
    \end{prop}
    \begin{proof}
        This proof is analogous to the proof of Corollary \ref{induction formulae} and will therefore be omitted..
    \end{proof}
We observe that $\nabla^k T$ satisfies an evolution equation analogous to that of $\nabla^k \mathrm{Rm}(g_\psi)$. Our goal is therefore to establish that $\nabla^k T$ decays at a polynomial rate in $f$, comparable to the decay of its initial data $\nabla^k T(0)$.

    \begin{corollary}\label{final corollary}
        There exists a constant $C_k>0$ such that on $M\times [0,\infty)$,
        \begin{equation*}
            |\nabla^kT|\le C_ke^{-\tau}f_{\psi}^{-\frac{k}{2}-1}.   
        \end{equation*}
    \end{corollary}
    \begin{proof}
        First, we compute
        \begin{equation*}
             \begin{split}
                \left(\frac{\partial}{\partial\tau}-\Delta_{\omega_\psi,X}\right)(e^{2\tau}f_\psi^{k+2}|\nabla^{k}T|^2)&=|\nabla^kT|^2\left(\frac{\partial}{\partial\tau}-\Delta_{\omega_\psi,X}\right)(e^{2\tau}f_\psi^{k+2})\\
                &\quad +e^{2\tau}f_\psi^{k+2}\left(\frac{\partial}{\partial\tau}
         -\Delta_{\omega_\psi,X}\right)|\nabla^kT|^2\\
         &\quad-2(k+2)e^{2\tau}f_\psi^{k+1}\Re(<\partial f_\psi,\bar\partial|\nabla^kT|^2>).\\
           \end{split}
        \end{equation*}
        On the one hand
        \begin{equation*}
            \left(\frac{\partial}{\partial\tau}-\Delta_{\omega_\psi,X}\right)(e^{2\tau}f_\psi^{k+2})=-ke^{2\tau}f_\psi^{k+2}-(k+2)(k+1)e^{2\tau}f_\psi^{k}|\partial f_\psi|^2.
        \end{equation*}
        On the other hand
             \begin{equation*}
            \left(\frac{\partial}{\partial\tau}-\Delta_{\omega_\psi,X}\right)|\nabla^kT|^2\le -|\nabla^{k+1}T|^2+k|\nabla^kT|^2+C(k,n)\sum_{p+q=k}|\nabla^p\Rm(g_\psi)||\nabla^qT||\nabla^kT|.
        \end{equation*}
        Hence we get 
        \begin{equation*}
            \begin{split}
                 \left(\frac{\partial}{\partial\tau}-\Delta_{\omega_\psi,X}\right)(e^{2\tau}f_\psi^{k+2}|\nabla^{k}T|^2)&\le -e^{2\tau}f_\psi^{k+2}|\nabla^{k+1}T|^2\\
&\quad+C(k,n)e^{2\tau}f_\psi^{k+2}\sum_{p+q=k}|\nabla^p\Rm(g_\psi)||\nabla^qT||\nabla^kT|\\
                 &\quad -(k+2)e^{2\tau}f_\psi^{k+1}X\cdot|\nabla^{k}T|^2
            \end{split}
        \end{equation*}
        Now we notice that for all $\sigma>0$, by Cauchy-Schwarz inequality 
        \begin{equation*}
         X\cdot|\nabla^{k}T|^2\le 2|X||\nabla^{k+1}T||\nabla^k T| \le \sigma|X|^2|\nabla^{k+1}T|^2+\frac{1}{\sigma}|\nabla^k T|^2.  
        \end{equation*}
        Take $\sigma>0$ such that $(k+2)\sigma|X|^2\le\frac{1}{2}f_\psi$, we have
        \begin{equation}\label{general formula on k}
            \begin{split}
                 \left(\frac{\partial}{\partial\tau}-\Delta_{\omega_\psi,X}\right)(e^{2\tau}f_\psi^{k+2}|\nabla^{k}T|^2)&\le -\frac{1}{2}e^{2\tau}f_\psi^{k+2}|\nabla^{k+1}T|^2+(k+2)\frac{1}{\sigma}e^{2\tau}f_\psi^{k+1}|\nabla^kT|^2\\
&\quad+C(k,n)e^{2\tau}f_\psi^{k+2}\sum_{p+q=k}|\nabla^p\Rm(g_\psi)||\nabla^qT||\nabla^kT|.
\end{split}
        \end{equation}
        From now on, we use induction to prove Corollary \ref{final corollary}. First, we prove there exists a constant $C_0>0$ such that $|T|\le C_0 e^{-\tau}f_\psi^{-1}$.
        On the one hand we have by \eqref{general formula on k}
        \begin{equation*}
           \begin{split}
                 \left(\frac{\partial}{\partial\tau}-\Delta_{\omega_\psi,X}\right)(e^{2\tau}f_\psi^{2}|T|^2)\le C(0,n)e^{2\tau}f_\psi|T|^2.
\end{split} 
        \end{equation*}
        Now recall [\eqref{order 1}, Corollary \ref{estimate of gradient dotpsi}], there exists a constant $A>0$ such that 
        \begin{equation*}
            \begin{split}
                \left(\frac{\partial}{\partial\tau}-\Delta_{\omega_\psi,X}\right)(e^{2\tau}f_\psi|\nabla\dot\psi|^2)&=-\frac{1}{2}e^{2\tau}f_\psi|\nabla^2\dot\psi|^2+Af_\psi^{-1}\\
                &\le -\frac{1}{2}e^{2\tau}f_\psi|T|^2+Af_\psi^{-1}.
            \end{split}
        \end{equation*}
 Hence we get
 \begin{equation*}
     \left(\frac{\partial}{\partial\tau}-\Delta_{\omega_\psi,X}\right)(e^{2\tau}f_\psi^{2}|T|^2+2C(0,n)e^{2\tau}f_\psi|\nabla\dot\psi|^2)\le 2C(0,n)Af_\psi^{-1}.
 \end{equation*}
 Thanks to Lemma \ref{A priori estimate of nabla^kT} and initial condition $|\mathcal{L}_{\frac{X}{2}}g_{\psi}-g_{\psi_0}|_g=O(f^{-1})$ as in Definition \ref{condition II}, by Lemma \ref{useful lemma}, there exists a constant $C>0$ such that 
 \begin{equation*}
     \sup_{M\times [0,\infty)}e^{2\tau}f_\psi^{2}|T|^2\le C+\sup_M f_\psi^{2}(0)|T|^2(0)+2C(0,n)f_\psi(0)|\nabla\dot\psi|^2(0)<\infty. 
 \end{equation*}
 Hence there exists a constant $C_0>0$ such that for all $\tau\in [0,\infty)$,
 \begin{equation*}
     |T|\le C_0e^{-\tau}f_\psi^{-1}.
 \end{equation*}
 Suppose that for all $0\le m\le k$, there exists a constant $C_m>0$ such that on $M\times [0,\infty)$,
 \begin{equation*}
     |\nabla^mT|\le C_me^{-\tau}f_\psi^{-1-\frac{m}{2}}.
 \end{equation*}
 For $k+1$, by \eqref{general formula on k}, there exists a constant $C(k+1)>0$ such that
 \begin{equation*}
      \begin{split}
                & \left(\frac{\partial}{\partial\tau}-\Delta_{\omega_\psi,X}\right)(e^{2\tau}f_\psi^{k+3}|\nabla^{k+1}T|^2)\\&\le -\frac{1}{2}e^{2\tau}f_\psi^{k+3}|\nabla^{k+2}T|^2+C(k+1)e^{2\tau}f_\psi^{k+2}|\nabla^{k+1}T|^2\\
&\quad+C(k+1)e^{2\tau}f_\psi^{k+3}\sum_{p+q=k+1}|\nabla^p\Rm(g_\psi)||\nabla^qT||\nabla^{k+1}T|.
\end{split}
 \end{equation*}
 By Corollary \ref{evolution function of covariant derivative of Rm} and the induction hypothesis, we deduce that there exists a constant $C(k+1)>0$ such that
 \begin{equation*}
    \begin{split}
         &\quad \left(\frac{\partial}{\partial\tau}-\Delta_{\omega_\psi,X}\right)(e^{2\tau}f_\psi^{k+3}|\nabla^{k+1}T|^2)\\
         &\le C(k+1)e^{2\tau}f_\psi^{k+2}|\nabla^{k+1}T|^2
+C(k+1)e^{\tau}f_\psi^{\frac{k+1}{2}}|\nabla^{k+1}T|.
    \end{split}
 \end{equation*}
 By Cauchy-Schwarz inequality, there exists a constant $C(k+1)>0$ such that
 \begin{equation*}
       \left(\frac{\partial}{\partial\tau}-\Delta_{\omega_\psi,X}\right)(e^{2\tau}f_\psi^{k+3}|\nabla^{k+1}T|^2)\le C(k+1)e^{2\tau}f_\psi^{k+2}|\nabla^{k+1}T|^2
+\frac{1}{f_\psi}.
 \end{equation*}
 By the induction hypothesis and \eqref{general formula on k} there exists a constant $C(k)>0$ such that
 \begin{equation*}
                 \left(\frac{\partial}{\partial\tau}-\Delta_{\omega_\psi,X}\right)(e^{2\tau}f_\psi^{k+2}|\nabla^{k}T|^2)\le -\frac{1}{2}e^{2\tau}f_\psi^{k+2}|\nabla^{k+1}T|^2+C(k)\frac{1}{f_\psi}
 \end{equation*}
 Hence
 \begin{equation*}
      \left(\frac{\partial}{\partial\tau}-\Delta_{\omega_\psi,X}\right)(e^{2\tau}f_\psi^{k+3}|\nabla^{k+1}T|^2+2C(k+1)e^{2\tau}f_\psi^{k+2}|\nabla^{k}T|^2)\le (1+2C(k+1)C(k))\frac{1}{f_\psi}.
 \end{equation*}
 Thanks to Lemma \ref{A priori estimate of nabla^kT} and the initial condition $|(\nabla^g)^k(\mathcal{L}_{\frac{X}{2}}g_{\psi}-g_{\psi_0})|_g=O(f^{-\frac{k}{2}-1})$ in Definition \ref{condition II}, by Lemma \ref{useful lemma}, there exists a constant $A>0$ such that 
 \begin{equation*}
    \begin{split}
         &\quad \sup_{M\times [0,\infty)}e^{2\tau}f_\psi^{k+3}|\nabla^{k+1}T|^2\\&\le A+\sup_Mf_\psi^{k+3}(0)|\nabla^{k+1}T|^2(0)+2C(k+1)f_\psi^{k+2}(0)|\nabla^{k}T|^2(0)<\infty.
    \end{split}
 \end{equation*}
 Hence there exists a constant $C_{k+1}>0$ such that for all $\tau\in [0,\infty)$,
 \begin{equation*}
     |\nabla^{k+1}T|\le C_{k+1}f_\psi^{-1-\frac{k+1}{2}}.
 \end{equation*}
 Therefore Corollary \ref{final corollary} holds for all $k\in\N_0$
    \end{proof}
  Before proceeding to the proof of Theorem \ref{convergence theorem}, we require a technical lemma concerning covariant derivatives of tensors.
  \begin{lemma}\label{tensorial lemma}
      Let $M$ be a Riemannian manifold with Riemannian metrics $g_1,g_2$. Let $k\in\N$ be an integer, and let $R\in C^\infty(\otimes^kT^*M)$ be a $k-$tensor. Then for any $m\in\N^*$, the difference of $(\nabla^{g_1})^mR$ and $(\nabla^{g_2})^mR$ is given by:
      \begin{equation*}
          \begin{split}
             (\nabla^{g_1})^mR- (\nabla^{g_2})^mR=\sum_{\substack{p+q=m\\ q<m}}\sum_{\substack{i_1+\cdot\cdot\cdot+i_l=p-l\\l\le p}}(\nabla^{g_1})^{i_1}\Gamma*\cdot\cdot\cdot (\nabla^{g_1})^{i_l}\Gamma*(\nabla^{g_2})^qR.
          \end{split}
      \end{equation*}
      Here  $\Gamma=\Gamma(g_1)-\Gamma(g_2)$.
  \end{lemma}
\begin{proof}
    This Lemma is proved by induction. When $m=1$, by the definition of Christoffel symbols, we have
    \begin{equation*}
        \nabla^{g_1}R-\nabla^{g_1}R=\Gamma*R.
    \end{equation*}
    Hence Lemma \ref{tensorial lemma} holds for $m=1$. We suppose that Lemma \ref{tensorial lemma} holds for all integers $j\le m$. Now for $j=m+1$, we have
    \begin{equation*}
        \begin{split}
           (\nabla^{g_1})^{m+1}R- (\nabla^{g_2})^{m+1}R &=(\nabla^{g_1}-\nabla^{g_2})(\nabla^{g_2})^mR+\nabla^{g_1}\left((\nabla^{g_1})^{m}R-(\nabla^{g_2})^mR\right)\\
           &=\Gamma*(\nabla^{g_2})^mR+\nabla^{g_1}\left((\nabla^{g_1})^{m}R-(\nabla^{g_2})^mR\right).
        \end{split}
    \end{equation*}
    By the induction hypothesis, we have
    \begin{equation*}
        \begin{split}
            \nabla^{g_1}\left((\nabla^{g_1})^{m}R-(\nabla^{g_2})^mR\right)&=\nabla^{g_1}\sum_{\substack{p+q=m\\ q<m}}\sum_{\substack{i_1+\cdot\cdot\cdot+i_l=p-l\\l\le p}}(\nabla^{g_1})^{i_1}\Gamma*\cdot\cdot\cdot (\nabla^{g_1})^{i_l}\Gamma*(\nabla^{g_2})^qR\\
            &=\sum_{\substack{p+q=m+1\\ q<m+1}}\sum_{\substack{i_1+\cdot\cdot\cdot+i_l=p-l\\l\le p}}(\nabla^{g_1})^{i_1}\Gamma*\cdot\cdot\cdot (\nabla^{g_1})^{i_l}\Gamma*(\nabla^{g_2})^qR.
        \end{split}
    \end{equation*}
    Therefore, we have \begin{equation}
         (\nabla^{g_1})^{m+1}R- (\nabla^{g_2})^{m+1}R=\sum_{\substack{p+q=m+1\\ q<m+1}}\sum_{\substack{i_1+\cdot\cdot\cdot+i_l=p-l\\l\le p}}(\nabla^{g_1})^{i_1}\Gamma*\cdot\cdot\cdot (\nabla^{g_1})^{i_l}\Gamma*(\nabla^{g_2})^qR,
    \end{equation}
    as expected. Lemma \ref{tensorial lemma}
 is true for all $m\in\N^*$.
 \end{proof}
    With uniform estimates on $T$, we can prove Theorem \ref{convergence theorem}.
\begin{proof}[Proof of Theorem \ref{convergence theorem}]
    For any $k\in \N_0$, by Corollary \ref{final corollary}, there exists a constant $C_k>0$ such that
    \begin{equation*}
        |(\nabla^{g_\psi})^k\partial\bar\partial\dot\psi|\le C_ke^{-\tau}f_{\psi}^{-1-\frac{k}{2}}.
    \end{equation*}
  Therefore, thanks to Proposition \ref{summary} and Proposition \ref{bound of S}, for any $k\in\N_0$, there exists a constant $C_k>0$ such that on $M\times [0,\infty)$,
    \begin{equation}\label{equation for claim 1}
        |(\nabla^{g_\psi})^k\partial\bar\partial\dot\psi|_g\le C_ke^{-\tau}f^{-1-\frac{k}{2}}.
    \end{equation}
\begin{claim}\label{claim 1}
    For any $k\in\N_0$, there exists a constant $C_k>0$ such that on $M\times [0,\infty)$,
    \begin{equation*}
        |(\nabla^{g})^k\partial\bar\partial\psi|_g\le C_kf^{-\frac{k}{2}},
    \end{equation*}
    and,
    \begin{equation*}
        |(\nabla^{g})^k\partial\bar\partial\dot\psi|_g\le C_ke^{-\tau}f^{-1-\frac{k}{2}}.
    \end{equation*} 
\end{claim}
\begin{proof}[Proof of Claim \ref{claim 1}]
    We prove this claim with induction. When $k=0$, according to \eqref{equation for claim 1}, after integration, it is clear to see that Claim \ref{claim 1} holds for $k=0$. Now we suppose that Claim \ref{claim 1} holds for all integers $l\le k$. Now for $l=k+1$, since $g_\psi$ is uniformly bi-Lipschitz to $g$, by Lemma \ref{tensorial lemma} we have
    \begin{equation*}
       ( \nabla^g)^{k+1}\partial\bar\partial\dot\psi-(\nabla^{g_\psi})^{k+1}\partial\bar\partial\dot\psi=\sum_{\substack{p+q=k+1\\ q\le k}}\sum_{\substack{i_1+\cdot\cdot\cdot +i_l=p\\i_j\ge 1}}(\nabla^{g})^{i_1}g_\psi*\cdot\cdot\cdot*(\nabla^g)^{i_l}g_\psi*(\nabla^{g_\psi})^q\partial\bar\partial\dot\psi.
    \end{equation*}
    By the induction hypothesis and \eqref{equation for claim 1}, there exists a constant $D_{k+1}>0$ such that
    \begin{equation}\label{claim 1 1}
       | ( \nabla^g)^{k+1}\partial\bar\partial\dot\psi|_g\le D_{k+1}e^{-\tau}f^{-1}\left(|(\nabla^{g})^{k+1}\partial\bar\partial\psi|_g+f^{-\frac{k+1}{2}}\right).
    \end{equation}
    We define $y(\rho)=|(\nabla^{g})^{k+1}\partial\bar\partial\psi|_g(x,\rho)$ for all $\rho\le\tau$ and fixed $x\in M$. Since $f$ is bounded from below by $\varepsilon>0$, then \eqref{claim 1} implies
    \begin{equation*}
        \begin{split}
            \frac{d}{d\rho}\left(y(\rho)\exp\{D_{k+1}e^{-\rho}f^{-1}(x)\}\right)&\le \exp\{D_{k+1}e^{-\rho}f^{-1}(x)\}D_{k+1}e^{-\rho}f(x)^{-\frac{k+3}{2}}\\
            &\le \exp\{D_{k+1}\varepsilon^{-1}\}D_{k+1}e^{-\rho}f(x)^{-\frac{k+3}{2}}
        \end{split}
    \end{equation*}
    By integration and the initial condition $(\nabla^g)^{k+1}g_{\psi_0}=O(f^{-\frac{1+k}{2}})$, we conclude that there exists a constant $C_{k+1}>0$ such that 
    \begin{equation*}
         |(\nabla^{g})^{k+1}\partial\bar\partial\psi|_g\le C_kf^{-\frac{k+1}{2}}.
    \end{equation*}
    This result together with \eqref{claim 1 1}, implies that
    \begin{equation*}
         | ( \nabla^g)^{k+1}\partial\bar\partial\dot\psi|_g\le D_{k+1}(C_{k+1}+1)e^{-\tau}f^{-\frac{k+3}{2}}.
    \end{equation*}
    Hence Claim \ref{claim 1} holds for all integers.
\end{proof}
\begin{claim}\label{claim 2}
    The solution $g_\psi$ converges uniformly and smoothly to a asymptotically conical gradient expander $(M,g_\infty,X)$. Moreover, as $g_{\psi_0}$ is asymptotically conical with a unique asymptotic K\"ahler cone $(M,g_0')$ as in Proposition \ref{geometric interpretation of condition II}, $(M,g_\infty,X)$ is the unique (up to biholomorphisms) asymptotically conical gradient expander with asymptotic cone $C,g_0'$.
\end{claim}
  \begin{proof}[Proof of Claim \ref{claim 2}]
       By integration in time for $0\le \tau<\rho<\infty$,
      \begin{equation*}
      |(\nabla^g)^k(g_\psi(\tau)-g_\psi(\rho)|_g\le C_k(e^{-\tau}-e^{-\rho})f^{-1-\frac{k}{2}}.
    \end{equation*}
    This implies that $g_\psi(\tau)$, as $\tau$ goes to $\infty$, converges uniformly smoothly to a K\"ahler metric $g_\infty$ such that for all $k\in\N_0$,
     \begin{equation*}
        |(\nabla^g)^k(g_\psi(\tau)-g_\infty)|_g\le C_ke^{-\tau}f^{-1-\frac{k}{2}}.
    \end{equation*}
At the same we have $|(\nabla^{g})^kg_\psi|_g\le C_k f^{-\frac{k}{2}}$ for all $k\in\N_0$, then we conclude that for all $k\in\N_0$,
     \begin{equation*}
        |(\nabla^{g_\psi(\tau)})^k(g_\psi(\tau)-g_\infty)|_{g_\psi(\tau)}\le C_ke^{-\tau}f^{-1-\frac{k}{2}}.
    \end{equation*}
    Moreover, since $\partial\bar\partial\dot\psi=\mathcal{L}_{\frac{X}{2}}g_\psi-\Ric(g_\psi)-g_\psi$, by letting $\tau$ tend to $\infty$, we have
    \begin{equation*}
        \mathcal{L}_{\frac{X}{2}}g_\infty-\Ric(g_\infty)-g_\infty=0.
    \end{equation*}
    As a consequence, the Hamiltonian potential $f_\psi$ converges uniformly to a smooth function $f_\infty$ as $\tau$ tends to $\infty$, and it turns out that $f_\infty$ is a Hamiltonian potential of $X$ with respect to $X$. Furthermore, by Proposition \ref{metric equivalence} and Corollary \ref{higher covariant derivative bound of Rm}, we find that $(M,g_\infty,X)$ is an asymptotically conical gradient K\"ahler-Ricci expander which is bi-Lipschitz to $(M,g,X)$.

    In particular, $|g_{\psi_0}-g_\infty|_g\le C_0f^{-1}$. Let $\Phi_t$ denote the flow of $-\frac{1}{2t}X$ for $t>0$, $g(t)$ be the self-similar solution of $g$ as in Proposition \ref{self-similar solution}, we have
    \begin{equation*}
        |t\Phi_t^*g_{\psi_0}-t\Phi_t^*g_\infty|_{g(t)}\le C_0\frac{1}{\Phi_t^*f}.
    \end{equation*}
Let $E$ denote the exceptional set as in Theorem \ref{CDS theorem}, then if $x\notin E$, $X(x)\neq 0$. Since $\Phi_t$ is the flow of $-\frac{1}{2t}X$ and $f$ is proper, thus, as $t$ tends to $0$, $\Phi_t^*f$ tends to $\infty$. Moreover, we identify $C\setminus \{o\}$ with $M\setminus E$ via the K\"ahler resolution as in Theorem \ref{CDS theorem}. Let $r$ denote the radial function of the K\"ahler cone $(C,g_0)$. Then for any $\varepsilon>0$, $\Phi_t^*f$ tends to $\infty$ uniformly on $\{r(x)\ge \varepsilon\}$. This implies that when $t$ tends to 0, $t\Phi_t^*g_{\psi_0}$ and $t\Phi_t^*g_\infty$ converge to the same limit metric. By Proposition \ref{geometric interpretation of condition II}, as $t$ tends to $0$, $t\Phi_t^*g_\infty$ converges locally to the conical metric $g_0'$ in Proposition \ref{geometric interpretation of condition II}. Hence $g_\infty$ is asymptotically conical with the unique asymptotic cone $(C,g_0')$.

  \end{proof}

    Now if the initial perturbation satisfies
\begin{equation*}
    \left|(\nabla^g)^k (g_{\psi_0} - g)\right|_g = o\left(f^{-\frac{k}{2}}\right)
\quad \text{for all } k \in \mathbb{N}_0,
\end{equation*}
then we have that
\begin{equation*}
    \left|(\nabla^g)^k (g_\infty - g)\right|_g = o\left(f^{-\frac{k}{2}}\right)
\quad \text{for all } k \in \mathbb{N}_0.
\end{equation*}
Let $g_\infty(t)_{t>0}$ be the self-similar solution to K\"ahler-Ricci flow associated to $g_\infty$. Since $\left|(\nabla^g)^k (g_\infty - g)\right|_g = o\left(f^{-\frac{k}{2}}\right)$
for all  $k \in \mathbb{N}_0$, thus for all $t>0$,
\begin{equation*}
   \begin{split}
        \left|(\nabla^{g(t)})^k (g_\infty(t) - g(t))\right|_{g(t)}(x)&=t^{-\frac{k}{2}} \left|(\nabla^g)^k (g_\infty - g)\right|_g (\Phi_t(x))\\
        &=t^{-\frac{k}{2}}o(f\circ\Phi_t^{-\frac{k}{2}})=o(1)(t\Phi_t^*f)^{-\frac{k}{2}},
   \end{split}
\end{equation*}
where $o(1)$ denotes a function which tends to 0 as $\Phi_t^*f$ tends to $\infty$. Since when $t$ tends to 0, on $\{r\ge \varepsilon\}$, $\Phi_t^*f$ tends to $\infty$ uniformly, hence $o(1)$ converges to $0$ uniformly when $t$ converges to $0$ on $\{r\ge\varepsilon\}$. Now we prove that on $\{r\ge\varepsilon\}$, function $t\Phi_t^*f$ is bounded from below. Recall the normalization of $f$ in Lemma \ref{soliton indentities}:
\begin{equation*}
    f=n+R_\omega+|\partial f|_g^2.
\end{equation*}
Now we compute,
\begin{equation*}
    \frac{\partial}{\partial t}t\Phi^*_tf=\Phi_t^*(f-|\partial f|^2)=\Phi_t^*(R_\omega+n)>0.
\end{equation*}
Here the last inequality comes from the fact $R_\omega+n>0$ on $M$. Moreover, for any $x\in C-\{o\}$, 
\begin{equation*}
    \lim_{t\to 0^+}t\Phi^*_tf(x)=\lim_{t\to 0^+}\left(nt+\Phi_t^*R_\omega(x)+\frac{1}{2}g(x,t)(X,X)\right)=\frac{1}{2}r(x)^2.
\end{equation*}
We conclude that on $\{r\ge \varepsilon\}$, for all $t>0$, $t\Phi_t^*f\ge \frac{1}{2}r^2\ge \frac{1}{2}\varepsilon^2.$ Function $t\Phi_t^*f$ is bounded from below on $\{r\ge \varepsilon\}$, and therefore, $g_\infty(t)$ converges locally smoothly to $g_0$ as $g(t)$ converges locally smoothly to $g_0$. The solution $g_\infty(t)$ to K\"ahler-Ricci flow satisfies conical condition in \cite[Theorem 1.4]{2025arXiv250500167C}.

Moreover, $g_\infty$ and $g$ are two asymptotically conical gradient K\"ahler-Ricci expander with bounded curvature, solution $g_\infty(t)$ satisfies cohomology condition and curvature condition in \cite[Theorem 1.4]{2025arXiv250500167C}. For the Killing condition, $JX$ is naturally a Killing vector field for $g_\infty$ because $(M,g_\infty,X)$ is a gradient expander (see Lemma \ref{Killing vector}).

As a result of Theorem 1.4 of \cite{2025arXiv250500167C}, we claim that $g_\infty(t)=g(t)$ for all $t>0$, thus $g_\infty=g$.
\end{proof}
    \subsection{Applications}
    In this section, we characterize the limiting metric $g_\infty$ obtained when perturbing the initial metric g by a term $\alpha\partial\bar\partial f$ (for some suitably chosen constant $\alpha$) using Theorem \ref{convergence theorem}. This characterization applies in particular to the canonical examples of Cao \cite{MR1449972}.
\begin{theorem}\label{last theorem}
    Let $(M,g,X)$ be an asymptotically conical gradient K\"ahler-Ricci expander. Let $f$ be the normalized Hamiltonian potential of $X$ with respect to $g$. 
    For any $\alpha\in\R$ such that $(1+\alpha)g+\alpha\Ric(g)>0$, the following holds:
    \begin{enumerate}
        \item the symmetric two-tensor $g_\alpha:=g+\alpha\partial\bar\partial f$ defines a K\"ahler metric;
        \item the solution to the normalized K\"ahler-Ricci flow starting from $g_\alpha$ exists for all time;
        \item as time goes to $\infty$, this solution converges to $\Phi_{\frac{1}{1+\alpha}}^*g$ smoothly uniformly, where $\Phi_t$ is the flow generated by $-\frac{1}{2t}X$.
    \end{enumerate}  
\end{theorem}
\begin{proof}
    First, we prove that $1+\alpha>0$. Since $\Ric(g)$ converges to 0 at infinity, the condition $(1+\alpha)g+\alpha\Ric(g)>0$ implies $1+\alpha>0$. Let the initial Kähler potential be given by $\psi_0 = \alpha \partial \bar{\partial} f$. Then, as shown in Section \ref{main results}, $\psi_0$ satisfies Condition II in the sense of Definition \ref{condition II}.

    Thanks to Theorem \ref{longtime existence theorem} and Theorem \ref{convergence theorem}, the solution to the normalized K\"ahler-Ricci flow starting from $g_\alpha$ exists for all time and when time goes to $\infty$, this solution converges smoothly uniformly to some asymptotical conical gradient K\"ahler-Ricci expander $(M,g_\infty^\alpha,X)$. Let $g_\alpha(\tau)_{\tau}$ denote this solution.

    Now we consider $\tilde g_\alpha(\tau):=\Phi_{1+\alpha}^*g_\alpha(\tau)$, $\tilde g_\alpha(\tau)$ is a solution to K\"ahler-Ricci flow starting from $\Phi_{1+\alpha}^*g_\alpha$. Let $g(t)_{t>0}$ be the self-similar solution to K\"ahler-Ricci flow associated to $g$. Then we observe that
    \begin{equation*}
        \begin{split}
            \Phi_{1+\alpha}^*g_\alpha=g(1+\alpha)+\alpha\Phi_{1+\alpha}^*\Ric(g)&=g(1+\alpha)+\alpha\Ric(g(1+\alpha))\\
            &=g-\int_1^{1+\alpha}\Ric(g(s))ds+\alpha\Ric(g(1+\alpha))\\
            &=g+\alpha\partial\bar\partial \int_1^{1+\alpha}\log\frac{\omega(s)^n}{\omega(1+\alpha)^n}ds.
        \end{split}
    \end{equation*}
    Let $\tilde\psi_0=\alpha \int_1^{1+\alpha}\log\frac{\omega(s)^n}{\omega(1+\alpha)^n}ds$, now we check that $(\nabla^g)^k\tilde\psi_0=O(f^{-\frac{k}{2}})$ for all $k\in\N_0$. It suffices to show that for any $k\in\N_0$, there exists a constant $C_k>0$ such that for all $s\in [1,1+\alpha]$
    \begin{equation*}
        |(\nabla^g)^kg(s)|_{g}\le C_kf^{-\frac{k}{2}},
    \end{equation*}
    and 
    \begin{equation*}
        |(\nabla^g)^k\Ric(g(s))|_g\le C_kf^{-\frac{k+2}{2}}.
    \end{equation*}
We proceed by induction. First, we prove that there exists a constant $A>1$ such that for all $s\in [1,1+\alpha]$, we have
\begin{equation*}
    \frac{1}{A}f\le s\Phi_s^*f\le Af.
\end{equation*}
Thanks to the soliton identity $f=R_\omega+n+|\partial f|_g^2$, we compute,
\begin{equation*}
   \frac{d}{ds}s\Phi_s^*f=\Phi_s^*(f-|\partial f|_g^2)=\Phi_s^*(n+R_\omega).
\end{equation*}
Let $A_1:=\sup_M |R_\omega|+n$, then we have $|\frac{d}{ds}s\Phi_s^*f|\le A_1$. Thus, for all $s\in [1,1+\alpha]$, we have
\begin{equation*}
    |s\Phi_s^*f-f|\le A_1|s-1|\le A_1|\alpha|.
\end{equation*}
The potential $f$ is bounded from below by $\varepsilon>0$, so on the one hand
\begin{equation*}
    s\Phi_s^*f\le f+A_1|\alpha|\le \left(1+\frac{A_1|\alpha|}{\varepsilon}\right)f.
\end{equation*}
On the other hand
\begin{equation*}
    s\Phi_s^*f\ge f-A_1|\alpha|\ge f-\frac{A_1|\alpha|}{s\varepsilon}s\Phi_s^*f\ge f-\frac{A_1|\alpha|}{\min\{1+\alpha,1\}\varepsilon}s\Phi_s^*f.
\end{equation*}
If we define $A=1+\frac{A_1|\alpha|}{\min\{1+\alpha,1\}}+\frac{A_1|\alpha|}{\varepsilon}$, then we have
\begin{equation}
    \frac{1}{A}f\le s\Phi_s^*f\le Af,
\end{equation}
as expected.

Since $g$ has bounded Riemannian curvature, there exists a constant $C_0>1$ such that for all $s \in [1,1+\alpha]$,
\begin{equation*}
    \frac{1}{C_0} g\leq\; g(s)\leq C_0 g.
\end{equation*}
Moreover, there exists a constant $C'>0$ such that $|\Ric(g)|_g\le \frac{C'}{f}$. Thus, we have \begin{equation*}
    |\Ric(g(s))|_g\le C_0|\Ric(g(s))|_{g(s)}=\frac{C_0C'}{s\Phi_s^*f}\le \frac{C_0C'A}{f}.
\end{equation*}
Hence, the claim holds for the base case $k=0$. We suppose this argument holds for all integers $l\le k$. Now for $l=k+1$. By Lemma \ref{tensorial lemma}, we have
\begin{equation*}
   \begin{split}
        &\quad (\nabla^g)^{k+1}\Ric(g(s))-(\nabla^{g(s)})^{k+1}\Ric(g(s))\\&=\sum_{\substack{p+q=k+1\\ q\le k}}\sum_{i_1+\cdot\cdot\cdot i_l=p}(\nabla^g)^{i_1}g(s)*\cdot\cdot\cdot(\nabla^g)^{i_l}g(s)*(\nabla^{g(s)})^q\Ric(g(s)).
   \end{split}
\end{equation*}
By induction hypothesis and by the curvature bound of $g$, there exists a constant $D_{k+1}>0$ such that
\begin{equation}\label{application}
   | (\nabla^g)^{k+1}\Ric(g(s))|_g\le D_{k+1}f^{-1}\left(|(\nabla^g)^{k+1}g(s)|_g+f^{-\frac{k+1}{2}}\right).
\end{equation}
 Let us define $y(s):=|(\nabla^g)^{k+1}g(s)|_g(x)$ for fixed $x\in M$ and $s\in [1,1+\alpha]$. Then we have the following ODE inequality:
 \begin{equation*}
     \begin{split}
         \frac{d}{ds}y(s)&\le D_{k+1}f^{-1}\left(y(s)+f^{-\frac{k+1}{2}}\right)
     \end{split}
 \end{equation*}
By integration and the initial condition $y(1)=0$, we can find a constant $C_{k+1}>0$ such that $|(\nabla^g)^{k+1}g(s)|_g(x)\le C_{k+1}f^{-\frac{k+1}{2}}$. This result together with \eqref{application} imply that this argument holds for all integers.

    Then by Theorem \ref{convergence theorem}, $\tilde g_\alpha(\tau)$ converges smoothly uniformly to $g$. Therefore, $g_\alpha(\tau)$, as $\tau$ tends to $\infty$, converges smoothly uniformly to $\Phi_{\frac{1}{1+\alpha}}^*g$.
\end{proof}
\begin{corollary}
    Let $(M,g,X)$ be the canonical examples of Cao \cite{MR1449972}. Then for any $\alpha\ge 0$, the following holds:
    \begin{enumerate}
        \item the form $g_\alpha:=g+\alpha\partial\bar\partial f$ defines a K\"ahler metric;
        \item the solution to the normalized K\"ahler-Ricci flow starting from $g_\alpha$ exists for all time;
        \item as time goes to $\infty$, this solution converges to $\Phi_{\frac{1}{1+\alpha}}^*g$ smoothly uniformly, where $\Phi_t$ is the flow generated by $-\frac{1}{2t}X$.
    \end{enumerate}  
\end{corollary}
\begin{proof}
    Since in Cao's models, the Ricci curvature is positive, hence for all $\alpha\ge 0$, $(1+\alpha)g+\alpha\Ric(g)>0.$ Hence we can use Theorem \ref{last theorem}.
\end{proof}
\bibliographystyle{plain}
\bibliography{references}
\end{document}